\documentclass[11pt,reqno]{amsart}
\usepackage{amsmath,amssymb,mathrsfs}
\usepackage{graphicx,cite}
\usepackage{subfigure}
\usepackage{epstopdf}
\usepackage{graphics} 
\usepackage{epsfig} 
\usepackage{tikz}
\usepackage{paralist}
\usepackage{booktabs}
\usepackage{enumerate}
\usepackage{enumitem}
\usepackage{mdwlist}

\newtheorem{theorem}{Theorem}[section]
\newtheorem{corollary}[theorem]{Corollary}
\newtheorem{lemma}[theorem]{Lemma}

\newtheorem{remark}[theorem]{Remark}
\newtheorem{problem}{Problem}

\numberwithin{equation}{section}

\begin{document}

\title[an inverse acoustic-elastic interaction problem]{An inverse acoustic-elastic interaction problem with phased or phaseless far-field data}

\author{Heping Dong}
\address{School of Mathematics, Jilin University, Changchun,  Jilin 130012, P. R. China}
\email{dhp@jlu.edu.cn}

\author{Jun Lai}
\address{School of Mathematical Sciences, Zhejiang University, Hangzhou, Zhejiang 310027, China}
\email{laijun6@zju.edu.cn}

\author{Peijun Li}
\address{Department of Mathematics, Purdue University, West Lafayette, Indiana 47907, USA}
\email{lipeijun@math.purdue.edu}




\begin{abstract}
Consider the scattering of a time-harmonic acoustic plane wave by a bounded elastic obstacle which is immersed in a homogeneous acoustic medium. This paper concerns an inverse acoustic-elastic interaction problem, which is to determine the location and shape of the elastic obstacle by using either the phased or phaseless far-field data. By introducing the Helmholtz decomposition, the model problem is reduced to a coupled boundary value problem of the Helmholtz equations. The jump relations are studied for the second derivatives of the single-layer potential in order to establish the corresponding boundary integral equations. The well-posedness is discussed for the solution of the coupled boundary integral equations. An efficient and high order Nystr\"{o}m-type discretization method is proposed for the integral system. A numerical method of nonlinear integral equations is developed for the inverse problem. For the case of phaseless data, we show that the modulus of the far-field pattern is invariant under a translation of the obstacle. To break the translation invariance, an elastic reference ball technique is introduced. We prove that the inverse problem with phaseless far-field pattern has a unique solution under certain conditions. In addition, a numerical method of  the reference ball technique based nonlinear integral equations is also proposed for the phaseless inverse problem. Numerical experiments are provided to demonstrate the effectiveness and robustness of the proposed methods.
\end{abstract}

\maketitle

\section{Introduction}

Consider the scattering of a time-harmonic acoustic plane wave by a bounded penetrable obstacle, which is immersed in an open space occupied by a homogeneous acoustic medium such as some compressible inviscid air or fluid. The obstacle is assumed to be a homogeneous and isotropic elastic medium. When the incident wave impinges the obstacle, a scattered acoustic wave will be generated in the open space and an elastic wave is induced simultaneously inside the obstacle. This scattering phenomenon leads to an acoustic-elastic interaction problem (AEIP). Given the incident wave and the obstacle, the direct acoustic-elastic interaction problem (DAEIP) is to determine the pressure of the acoustic wave field and the displacement of the elastic wave field in the open space and in the obstacle, respectively; the inverse acoustic-elastic interaction problem (IAEIP) is to determine the elastic obstacle from the far-field pattern of the acoustic wave field. The AEIPs have received ever-increasing attention due to their significant applications in seismology and geophysics \cite{LL-86}. Despite many work done so far for both of the DAEIP and IAEIP, they still present many challenging mathematical and computational problems due to the complex of the model equations and the associated Green tensor, as well as the nonlinearity and ill-posedness.

The phased IAEIP referres to the IAEIP that determines the location and shape of the elastic obstacle from the phased far-field data, which contains both the phase and amplitude information. It has been extensively studied in the recent decades. In \cite{EHR2008, EHR2009}, an optimization based variational method and a decomposition method were proposed to the IAEIP. The direct imaging methods, such as the linear sampling method \cite{MS2009, MS2011} and the factorization method \cite{KR2012, YHXZ2016, HKY2016}, were also developed to the corresponding inverse problems with far-field and near-field data. For the theoretical analysis, the uniqueness results may be found in \cite{MS2009, QYZ2018} for the phased IAEIP.

The phaseless IAEIP is to determine the location and shape of the elastic obstacle from the modulus of the far-field acoustic scattering data, which contains only the amplitude information. Due to the translation invariance property of the phaseless far-field field, it is impossible to uniquely determine the location of the unknown object by plane incident wave, which makes the phaseless inverse problem much more challenging than the phased counterpart. Various numerical methods have been proposed to solve the phaseless inverse obstacle scattering problems, especially for the acoustic waves which are governed by the scalar Helmholtz equation. For the shape reconstruction with one incident plane wave, we refer to the Newton iterative method \cite{RW1997}, the nonlinear integral equation method \cite{Ivanyshyn2007, OR2010}, the fundamental solution method \cite{KarageorghisAPNUM}, and the hybrid method \cite{Lee2016}. In particular, the nonlinear integral equation method, which was proposed by Johansson and Sleeman \cite{TB2007}, was extended to reconstruct the shape of a sound-soft crack by using phaseless far-field data from a single incident plane wave \cite{GDM2018}. To reconstruct the location and shape simultaneously, Zhang et al. \cite{ZhangBo2017, ZhangBo2018} proposed an iterative method by using the superposition of two plane waves with different incident directions to reconstruct the unknown object. In \cite{JiLiuZhang2019}, a phase retrieval technique combined with the direct sampling method was proposed to reconstruct the location and shape of an obstacle from phaseless far-field data. The method was extended to the phaseless inverse elastic scattering problem and phaseless IAEIP \cite{JiLiu2018}. We refer to \cite{XZZ2018, ZG18, SZG2018, ZWGL2019, ZSGL2019} for the uniqueness results on the inverse scattering problems by using phaseless data. Related phaseless inverse scattering problems as well as numerical methods can be found in \cite{Ammari2016, Li2017, JiLiuZhang2018, CH2016, BLL2013, Bao2016, ZGLL18, Klibanov2019}. Recently, a reference ball technique based nonlinear integral equations method was proposed in \cite{DZhG2018} to break the translation invariance from phaseless far-field data by one incident plane wave. In our recent work \cite{DLL2019}, we extended this method to the inverse elastic scattering problem with phaseless far-field data by using a single incident plane wave to recover both the location and shape of a rigid elastic obstacle.

In this paper, we consider both the DAEIP and IAEIP. In particular, we study the IAEIP of determining the location and shape of an elastic obstacle from the phased or phaseless far-field data with a single incident plane wave. The goal of this work is fivefold:

\begin{enumerate}

\item deduce the jump relations for the second derivatives of the single-layer potential and the coupled system of boundary integral equations;

\item prove the well-posedness of the solution for the coupled system and develop a Nystr\"{o}m-type discretization for the boundary integral equations;

\item show the translation invariance of the phaseless far-field pattern and present a uniqueness result for the phaseless IAEIP;

\item propose a numerical method of nonlinear integral equations to reconstruct the obstacle's location and shape by using the phased far-field data from a single plane incident wave;

\item develop a reference ball based method to reconstruct both the obstacle's location and shape by using phaseless far-field data from a single plane incident wave.

\end{enumerate}

For the direct problem, instead of considering directly the coupled acoustic and elastic wave equations, we make use of the Helmholtz decomposition and reduce the model problem into a coupled boundary value problem of the Helmholtz equations. The method of boundary integral equations is adopted to solve the coupled Helmholtz system. However, the boundary conditions are more complicated, since the second derivatives of surface potentials are involved due to the traction operator. Therefore, we investigate carefully the jump relations for the second derivatives of the single-layer potential and establish coupled boundary integral equations. Moreover, we prove the existence and uniqueness for the solution of the coupled boundary integral equations, and develop a Nystr\"{o}m-type discretization to efficiently and accurately solve the direct acoustic-elastic interaction problem. The proposed method is extremely efficient for the direct scattering problem since we only need to solve the scalar Helmholtz equations instead of solving the vector Navier equations. Related work on the direct acoustic-elastic interaction problems and time-domain acoustic-elastic interaction problem can be found in \cite{BGL2018, JL2017, LukeMartin1995, YHX2017}.

For the inverse problem, motivated by the reference ball technique \cite{LiJingzhi2009, ZG18} and the recent work \cite{DLL2019, DZhG2018}, we give a uniqueness result for the phaseless IAEIP by introducing an elastic reference ball, and also propose a nonlinear integral equations based 
iterative numerical scheme to solve the phased and phaseless IAEIP. Since the location of reference ball is known, the method breaks the translation invariance and is able to recover the location information of the obstacle with negligible additional computational costs. Numerical results show that the method is effective and robust to reconstruct the obstacle with either the phased or phaseless far-field data.

The paper is organized as follows. In Section 2, we introduce the coupled acoustic-elastic interaction problem and show the uniqueness for the coupled boundary value problem by using the Helmholtz decomposition. In Section 3, we study the jump properties for the second derivatives of the single-layer potential and establish the coupled boundary integral equations. The existence and uniqueness of the solution for the coupled boundary integral equations are given. Section 4 is devoted to the translation invariance and the uniqueness for the phaseless IAEIP. Section 5 presents a high order Nystr\"{o}m-type discretization to solve the coupled boundary value problem. In Section 6, a method of nonlinear integral equations and a reference ball based method are developed to solve the phased and phaseless inverse problems, respectively. Numerical experiments are provided to demonstrate the effectiveness of the proposed methods in Section 7. The paper is concluded with some general remarks and directions for future work in Section 8.

\section{Problem formulation}

Consider the scattering problem of a time-harmonic acoustic plane wave by a two-dimensional elastic obstacle $D$ with $\mathcal{C}^2$ boundary $\Gamma_D$. The elastic obstacle $D$ is assumed to be homogeneous and isotropic with a mass density $\rho_{\rm e}$; the exterior domain $\mathbb{R}^2\setminus \overline{D}$ is assumed to be filled with a homogeneous and compressible inviscid air or fluid with a mass density $\rho_{\rm a}>0$. Denote by $\nu=(\nu_1, \nu_2)^\top$ and $\tau=(-\nu_2, \nu_1)^\top$ the unit normal vector and the tangential vector on $\Gamma_D$, respectively. Let $\nu_\perp=(\nu_2,-\nu_1)^\top=-\tau$. Given a vector function $\boldsymbol U=(U_1, U_2)^\top$ and a scalar function $u$, we introduce the scalar and vector curl operators
\[
{\rm curl}~\boldsymbol  U=\partial_{x_1}U_2-\partial_{x_2}U_1, \quad {\bf curl}~u=(\partial_{x_2}u, -\partial_{x_1}u)^\top.
\]

Specifically, the time-harmonic acoustic plane wave is given by $u^{\rm inc}(x)=e^{{\rm i}\kappa_{\rm a} x\cdot d}$, where $d=(\cos\theta,\sin\theta)^\top$ is the propagation direction vector, $\theta\in [0, 2\pi)$ is the incident angle. Given the incident field $u^{\rm inc}$, the direct problem is to find the elastic wave displacement $\boldsymbol{U}\in(C^2(D)\cap C(\overline{D}))^2$ and the acoustic wave pressure $u\in C^2(\mathbb{R}^2\setminus \overline{D})\cap C(\mathbb{R}^2\setminus D)$, which satisfy the Navier equation and the Helmholtz equation, respectively:
\begin{align}
\mu\Delta\boldsymbol{U}+(\lambda+\mu)\nabla\nabla\cdot\boldsymbol{U}+\omega^2\rho_{\rm e}\boldsymbol{U}&=0 \quad {\rm in~}\, {D}, \label{Navier equation}\\
\Delta u+\kappa_{\rm a}^2 u&=0 \quad{\rm in~}\,\mathbb{R}^2\setminus \overline{D}. \label{Helmholtz equation}
\end{align}
Moreover, $\boldsymbol U$ and $u$ are required to satisfy the transmission conditions
\begin{equation}\label{transmission}
T(\boldsymbol{U})=-u\nu,\quad  \boldsymbol{U}\cdot\nu=\frac{1}{\omega^2\rho_{\rm a}}\partial_\nu u \quad{\rm on~}\,\Gamma_D.
\end{equation}
The scattered acoustic wave pressure $u^{\rm s}:=u-u^{\rm inc}$ is required to satisfy the Sommerfeld radiation condition 
\begin{align}\label{radiation}
\lim_{r\to\infty}r^{\frac{1}{2}}(\partial_{r}u^{\rm s}-{\rm i}\kappa_{\rm a} u^{\rm s})=0, \quad r=|x|. 
\end{align}
Here $\omega>0$ is the angular frequency, $\kappa_{\rm a}=\omega/c$ is the wavenumber in the air/fluid with the sound speed $c$, and $\lambda, \mu$ are the Lam\'{e} parameters satisfying $\mu>0, \lambda+\mu>0$. The traction operator $T$ is defined by 
\begin{align*}
T(\boldsymbol{U}):=\mu\partial_\nu\boldsymbol{U}+(\lambda+\mu)(\nabla\cdot\boldsymbol{U})\nu.
\end{align*}

It can be shown (cf. \cite{LukeMartin1995,YHX2017}) that the scattering problem \eqref{Navier equation}--\eqref{radiation} admits a unique solution $(\boldsymbol{U}, u)$ for all but some particular frequencies $\omega$, which are called the Jones frequencies \cite{Jones1983}. At the Jones frequency, the acoustic wave field $u$ is unique, but the elastic field $\boldsymbol{U}$ is not unique. Since the Jones frequency happens only for some special geometries \cite{Jones1983}, for simplicity, we assume that $D$ does not admit any Jones mode in this work.

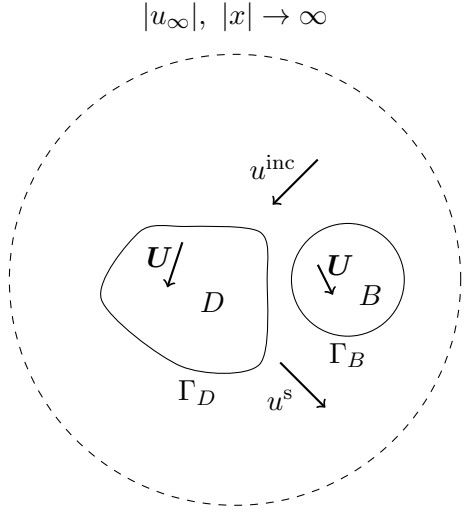
\begin{figure}
	\centering
	\begin{tikzpicture}
	\pgfmathsetseed{8}
	\draw plot [smooth cycle, samples=7, domain={1:8}] (\x*360/8+5*rnd:0.5cm+1cm*rnd) node at (0.2,-0.3) {$D$}; 
	\draw node at (0, -1.5) {$\Gamma_D$};
	\draw (2, 0) circle (0.75cm) node at (2.3,-0.2) {$B$}; 
	\draw node at (2, -1) {$\Gamma_B$};
	\draw [->,thick] (1.1,-1.1)--(1.7,-1.7) node at (1,1.5) {${u}^{\rm inc}$}; 
	\draw [->,thick] (-0.2,0.5)--(-0.4,-0.1) node at (-0.5,0.3) {$\boldsymbol{U}$}; 
	\draw [->,thick] (1.6,0.2)--(1.8,-0.2) node at (1.9,0.2) {$\boldsymbol{U}$}; 
	\draw [->,thick] (1.6,1.6)--(1, 1) 
	node at (1.1,-1.6) {$u^{\rm s}$}; 
	\draw [dashed] (0.5, 0) circle (3cm) node at (0.5, 3.5) {$|u_\infty|, ~|x|\to \infty$ }; 
	\end{tikzpicture}
	\caption{Geometry of the scattering problem with a reference ball.} \label{fig:illustration}
\end{figure}

For any solution $\boldsymbol U$ of the elastic wave equation \eqref{Navier equation}, we introduce the Helmholtz decomposition 
\begin{equation}\label{HelmDeco}
\boldsymbol{U}=\nabla\phi+\boldsymbol{\rm curl}~\psi,
\end{equation}
where $\phi, \psi$ are two scalar potential functions. Substituting \eqref{HelmDeco} into \eqref{Navier equation} yields
\[
\nabla [(\lambda+2\mu)\Delta\phi+\omega^2\rho_{\rm e}\phi]
+\boldsymbol{\rm curl}~(\mu\Delta\psi+\omega^2\rho_{\rm e}\psi)=0,
\]
which is fulfilled if $\phi$ and $\psi$ satisfy the Helmholtz equation with a different wavenumber, respectively:
\[
\Delta\phi+\kappa_{\rm p}^{2}\phi=0, \quad \Delta\psi+\kappa_{\rm s}^{2}\psi=0.
\]
Here
\[
\kappa_{\rm p}=\omega\left(\frac{\rho_{\rm e}}{\lambda+2\mu}\right)^{1/2},\quad\kappa_{\rm s}=\omega\left(\frac{\rho_{\rm e}}{\mu}\right)^{1/2},
\]
are the compressional wavenumber and the shear wavenumber, respectively.

Substituting the Helmholtz decomposition into \eqref{transmission} and taking the dot product with $\nu$ and $\tau$, respectively, we obtain 
\begin{align*}
\mu\nu\cdot\partial_\nu\nabla\phi+\mu\nu\cdot\partial_\nu\boldsymbol{\rm curl}~\psi -(\lambda+\mu)\kappa^2_{\rm p}\phi+u^{\rm s}&=f_1,\\
\mu\tau\cdot\partial_\nu\nabla\phi+\mu\tau\cdot\partial_\nu\boldsymbol{\rm curl}~\psi&=f_2, \\
\partial_\nu\phi+\partial_\tau\psi-\partial_\nu u^{\rm s}/(\omega^2\rho_{\rm a})&=f_3,
\end{align*}
where
\[
f_1=-u^{\rm inc},\quad f_2=0, \quad f_3=\partial_\nu u^{\rm inc}/(\omega^2\rho_{\rm a}).
\]

In summary, the scalar potential functions $\phi, \psi$ and the scattered acoustic wave $u^{\rm s}$ satisfy the following coupled
boundary value problem
\begin{align}\label{HelmholtzDec}
\begin{cases}
\Delta\phi+\kappa_{\rm p}^{2}\phi=0, \quad \Delta\psi+\kappa_{\rm s}^{2}\psi=0, \quad &{\rm in} ~D,\\
\Delta u^{\rm s}+\kappa_{\rm a}^2 u^{\rm s}=0, \quad &{\rm in} ~\mathbb{R}^2\setminus\overline{D},\\
\mu\nu\cdot\partial_\nu\nabla\phi+\mu\nu\cdot\partial_\nu\boldsymbol{\rm curl}~\psi -(\lambda+\mu)\kappa^2_{\rm p}\phi+u^{\rm s}=f_1, \quad &{\rm on} ~ \Gamma_D,\\
\tau\cdot\partial_\nu\nabla\phi+\tau\cdot\partial_\nu\boldsymbol{\rm curl}~\psi=f_2, \quad &{\rm on} ~ \Gamma_D,\\
\partial_\nu\phi+\partial_\tau\psi-\partial_\nu u^{\rm s}/(\omega^2\rho_{\rm a})=f_3, \quad &{\rm on} ~ \Gamma_D,\\
\displaystyle{\lim_{r\to\infty}r^{\frac{1}{2}}(\partial_{r}u^{\rm s}-{\rm i}\kappa_{\rm a}u^{\rm s})=0}, \quad &r=|x|.
\end{cases}
\end{align}

The following result concerns the uniqueness of the boundary value problem \eqref{HelmholtzDec}.

\begin{theorem}\label{uniquethm}
The coupled boundary value problem \eqref{HelmholtzDec} has at most one solution for $\kappa_{\rm p}>0, \kappa_{\rm s}>0, \kappa_{\rm a}>0$.
\end{theorem}

\begin{proof}
It suffices to show that $\phi=\psi=u^{\rm s}=0$ when $f_1=f_2=f_3=0$. It follows from straightforward calculations that 
	\begin{align*}
	-\int_{\Gamma_D}u^{\rm s} \overline{\partial_\nu u^{\rm s}}{\rm d}s
	= &\omega^2\rho_{\rm a}\int_{\Gamma_D}(\mu \nu \cdot \partial_\nu\nabla \phi + \mu \nu \cdot \partial_\nu \boldsymbol{\rm curl}~\psi - (\lambda+\mu )\kappa^2_{\rm p}\phi)(\partial_\nu\overline{\phi}+\partial_{\tau}\overline{\psi}){\rm d}s\notag  \\ 
	=&\omega^2\rho_{\rm a}\int_{\Gamma_D}(\mu \partial_\nu(\nabla\phi+\boldsymbol{\rm curl}~\psi)\cdot \nu - (\lambda+\mu )\kappa^2_{\rm p}\phi)\left((\nabla\overline{\phi}+\boldsymbol{\rm curl}~\overline{\psi}) \cdot\nu\right){\rm d}s \nonumber\\
	& + \omega^2\rho_{\rm a}\int_{\Gamma_D}\left(\mu \partial_\nu(\nabla\phi+\boldsymbol{\rm curl}~\psi)\cdot \tau\right)\left((\nabla\overline{\phi}+\boldsymbol{\rm curl}~\overline{\psi}) \cdot\tau\right){\rm d}s\notag \\
	= &\omega^2\rho_{\rm a}\int_{\Gamma_D} \left(\mu \partial_\nu(\nabla\phi+\boldsymbol{\rm curl}~\psi)\cdot \nu\nu+\mu \partial_\nu(\nabla\phi+\boldsymbol{\rm curl}~\psi)\cdot\tau\tau-(\lambda+\mu)\kappa^2_{\rm p}\phi\nu\right )\nonumber \\  
	&\cdot\left((\nabla\overline{\phi}+\boldsymbol{\rm curl}~\overline{\psi})\cdot\nu \nu +(\nabla\overline{\phi}+\boldsymbol{\rm curl}~\overline{\psi})\cdot\tau\tau\right){\rm d}s\notag\\  
	= &\omega^2\rho_{\rm a}\int_{\Gamma_D} \left(\mu \partial_\nu(\nabla\phi+ \boldsymbol{\rm curl}~\psi)-(\lambda+\mu)\kappa^2_{\rm p}\phi\nu\right )\cdot(\nabla\overline{\phi}+ \boldsymbol{\rm curl}~\overline{\psi}){\rm d}s \nonumber\\ 
	= &\omega^2\rho_{\rm a}\int_D(\mu \nabla(\nabla\phi+\boldsymbol{\rm curl}~\psi): \nabla(\nabla\overline{\phi}+\boldsymbol{\rm curl}~\overline{\psi})+(\lambda+\mu)\nabla\cdot(\nabla\phi+\boldsymbol{\rm curl}~\psi) \notag \\  & \nabla\cdot(\nabla\overline{\phi}+\boldsymbol{\rm curl}~\overline{\psi}) -\omega^2\rho_{\rm e}(\nabla\phi+\boldsymbol{\rm curl}~\psi)\cdot (\nabla\overline{\phi}+\boldsymbol{\rm curl}~\overline{\psi})){\rm d}x =  0,
	\end{align*}
	where $A:B={\rm tr}(AB^\top)$ is the Frobenius inner product of square matrices $A$ and
	$B$. The last two identities follow from Green's formula and the Navier equation \eqref{Navier equation}. Taking the imaginary part of the above equation yields 
	\[
	\Im\int_{\Gamma_D}u^{\rm s} \overline{\partial_\nu u^{\rm s}}{\rm d}s =0.
	\]
which gives that $u^{\rm s}=0$ in $\mathbb{R}^2\backslash\overline{D}$ by Rellich's lemma. Using the continuity conditions \eqref{transmission}, we conclude that $\boldsymbol{U}$ is identically zero in $D$ provided that there is no Jones mode in $D$. Hence, 
	\begin{eqnarray*}
		\nabla \phi = -\mathbf{curl}~\psi \quad \mbox{ in } D,
	\end{eqnarray*}
	which implies $\Delta \phi = 0$ and $\Delta \psi = 0 $. The proof is completed by noting that $\Delta \phi=-\kappa^2_{\rm p}\phi=0$ and $\Delta \psi=-\kappa^2_{\rm s}\psi=0$ in $D$.
\end{proof}

It is known that a radiating solution of the Helmholtz equation \eqref{Helmholtz equation} has the asymptotic behaviour of the form
\begin{equation*}
u^{\rm s}(x)=\frac{\mathrm{e}^{\mathrm{i}\kappa_{\rm a}|x|}}{\sqrt{|x|}}\Big\{u_\infty(\hat{x})+\mathcal{O}\left(\frac{1}{|x|}\right)\Big\} \quad\text{as}~|x|\to\infty,
\end{equation*}
uniformly in all directions $\hat{x}:=x/|x|$. The function $u_\infty$, defined on the unit circle $\Omega=\{x\in\mathbb R^2: |x|=1\}$, is known as the far-field pattern of $u^{\rm s}$. Let $B=\left\{x\in \mathbb{R}^2: |x-x_0|\le R \right\}\subset\mathbb{R}^2$ be an artificially added elastic ball centered at $x_0$ such that $D\cap B=\emptyset$. The problem geometry is shown in Figure \ref{fig:illustration}. For brevity, we denote the boundary of $D$ and $B$ by $\Gamma_D$ and $\Gamma_B$, respectively. The phased and phaseless IAEIP can be stated as follows:

\begin{problem}[Phased IAEIP]\label{problem_1}
Given an incident plane wave $u^{\rm inc}$ with a single incident direction $d$ and the corresponding far-field pattern $u_\infty(\hat x), ~\forall\hat x\in\Omega$ due to the unknown obstacle $D$, the inverse problem is to determine the location and shape of the boundary $\Gamma_D$.
\end{problem}

\begin{problem}[Phaseless IAEIP]\label{problem_2}  Given an incident plane wave ${u}^{\rm inc}$ with a single incident direction $d$ and the corresponding phaseless far-field pattern $|u_\infty(\hat x)|, ~\forall\hat x\in\Omega$ due to the scatterer $D\cup B$, the inverse problem is to determine the location and shape of the boundary $\Gamma_D$.
\end{problem}

\section{Boundary integral equations}

In this section, we derive the boundary integral equations for the coupled boundary value problem \eqref{HelmholtzDec} and discuss their well-posedness.

\subsection{Jump relations}

We begin with investigating the jump relations for the surface potentials at the boundary $\Gamma_D$. 

For given vectors $a=(a_1,a_2)^\top\in\mathbb{R}^2$ and $b=(b_1,b_2)^\top\in\mathbb{R}^2$, denote
\[
\langle a,b\rangle=ab^\top, \quad \nabla a=(\nabla a_1, \nabla a_2)^\top,\quad \nabla a^\top=(\nabla a_1, \nabla a_2)=(\nabla a)^\top.
\]
For a given scalar function $f(x,y)$, define
\[
\nabla_y(\nabla_{x}f)=\langle\nabla_{x},\nabla_{y}\rangle f=\nabla_{x}\nabla_{y}^\top f=\left[  
\begin{array}{cc}  
\partial^2_{x_1y_1}f & \partial^2_{x_1y_2}f \\[2pt]
\partial^2_{x_2y_1}f & \partial^2_{x_2y_2}f \\ 
\end{array}
\right]
\]
and
\begin{align*}
(\nabla_{x}\nabla_{y}^\top f,\nu)=\left[  
\begin{array}{cc}  
\partial^2_{x_1y_1}f & \partial^2_{x_1y_2}f \\ [2pt]
\partial^2_{x_2y_1}f & \partial^2_{x_2y_2}f \\ 
\end{array}
\right]
\left[                
\begin{array}{c}  
\nu_1 \\ 
\nu_2 \\ 
\end{array}
\right],\quad 
(\boldsymbol{\rm curl}_{x}\nabla_{y}^\top f,\nu)=\left[  
\begin{array}{cc}  
 \partial^2_{x_2y_1}f  & \partial^2_{x_2y_2}f \\ [2pt]
-\partial^2_{x_1y_1}f  &-\partial^2_{x_1y_2}f \\ 
\end{array}
\right]
\left[                
\begin{array}{c}  
\nu_1 \\ 
\nu_2 \\ 
\end{array}
\right].
\end{align*}
Denote the fundamental solution of the two-dimensional Helmholtz equation by
$$
\Phi(x,y;\kappa)=\frac{\mathrm{i}}{4}H_0^{(1)}(\kappa|x-y|), \quad x\neq y,
$$
where $H_0^{(1)}$ is the Hankel function of the first kind of order zero. The single- and double-layer potentials with density $g$ are defined by
\begin{align*}
\omega(x)=\int_{\Gamma_D}\Phi(x,y;\kappa)g(y)\mathrm{d}s(y),\quad
\chi(x)=\int_{\Gamma_D}\frac{\partial\Phi(x,y;\kappa)}{\partial\nu(y)}g(y)\mathrm{d}s(y), \quad x\in\mathbb{R}^2\setminus\Gamma_D.
\end{align*}
In addition, we define the tangential-layer potential by
$$
\zeta(x)=\int_{\Gamma_D}\frac{\partial\Phi(x,y;\kappa)}{\partial\tau(y)}g(y)\mathrm{d}s(y), \quad x\in\mathbb{R}^2\setminus\Gamma_D.
$$

The jump relations can be found in \cite{DR-book1983} for the single- and double-layer potentials as $x\to\Gamma_D$. It is necessary to study the jump properties for the derivatives of those layer potentials in order to derive the boundary integral equations for the coupled boundary value problem \eqref{HelmholtzDec}.

\begin{lemma}\label{lemma1}
The first derivatives of the single-layer potential $\omega$ with density $g\in C^{0,\alpha}(\Gamma_D)$, $0<\alpha<1$, can be uniformly extended in a H\"{o}lder continuous fashion from $\mathbb{R}^2\setminus\overline{D}$ into $\mathbb{R}^2\setminus{D}$ and from $D$ into $\overline{D}$ with the limiting values
\begin{align}\label{single jump}
(\nabla\omega)_\pm(x)=\int_{\Gamma_D}\nabla_{x}\Phi(x,y;\kappa)g(y)\mathrm{d}s(y)\mp\frac{1}{2}\nu(x)g(x), \quad x\in\Gamma_D,
\end{align}
where 
$$
(\nabla\omega)_{\pm}(x) :=\lim_{h\to0^+}(\nabla\omega)(x\pm h\nu(x)).
$$
Moreover, for the single-layer potential $\omega$ with density $g\in C^{0,\alpha}(\Gamma_D)$, $0<\alpha<1$, we have
\begin{align}\label{curlsingle jump}
({\bf curl}~\omega)_{\pm}(x)= \int_{\Gamma_D}{\bf curl}_x\Phi(x,y;\kappa)g(y)\mathrm{d}s(y)\pm\frac{1}{2}\tau(x)g(x), \quad x\in\Gamma_D.
\end{align}
\end{lemma}

\begin{proof}
Noting 
\begin{align} \label{symmetry relation}
\nabla_{x}\Phi(x,y;\kappa)=-\nabla_{y}\Phi(x,y;\kappa)
\end{align}
and 
\begin{align} \label{vector identity}
\nabla_{y}\Phi(x,y;\kappa)=\nu(y)\frac{\partial\Phi(x,y;\kappa)}{\partial\nu(y)}+\tau(y)\frac{\partial\Phi(x,y;\kappa)}{\partial\tau(y)},
\end{align}
we may similarly show \eqref{single jump} by following the proof of Theorem 2.17 in \cite{DR-book1983}. It is clear to note \eqref{curlsingle jump} 
by combining the fact that ${\bf curl}~\omega=(\nabla\omega)_\perp$, $\nu_\perp=-\tau$ and the jump relation \eqref{single jump}.
\end{proof}

\begin{lemma}\label{lemma2}
The first derivatives of the double-layer potential $\chi$ with density $g\in C^{1,\alpha}(\Gamma_D)$, $0<\alpha<1$, can be uniformly extended in a H\"{o}lder continuous fashion from $\mathbb{R}^2\setminus\overline{D}$ into $\mathbb{R}^2\setminus{D}$ and from $D$ into $\overline{D}$ with the limiting values
\begin{align}\label{double jump}
(\nabla\chi)_\pm(x)=\kappa^2\int_{\Gamma_D}\Phi(x,y;\kappa)\nu(y)g(y)\mathrm{d}s(y)+\int_{\Gamma_D}{\bf curl}_x\Phi(x,y;\kappa)\frac{\partial g}{\partial\tau}(y)\mathrm{d}s(y)
\nonumber \\
\pm\frac{1}{2}\tau(x)\frac{\partial g}{\partial\tau}(x), \quad x\in\Gamma_D.
\end{align}
\end{lemma}

\begin{proof}
Using the jump relation \eqref{curlsingle jump} and the identities 
$$
\nabla\nabla\cdot b=\Delta b+{\bf curl}~{\rm curl}~b
$$ 
and
\[
-\Delta\Phi(x,y;\kappa)=\kappa^2\Phi(x,y;\kappa),\quad x\neq y,
\]
we may easily show \eqref{double jump} by following the proof of Theorem 7.32 in \cite{Kress-book2014} and Theorem 2.23 in \cite{DR-book1983}. 
\end{proof}

\begin{theorem}\label{Th3}
For the tangential-layer potential $\zeta$ with density $g\in C^{1,\alpha}(\Gamma_D)$, $0<\alpha<1$, we have
\begin{align}\label{double-type jump}
(\nabla\zeta)_\pm(x)=-\int_{\Gamma_D}\nabla_x\Phi(x,y;\kappa)\frac{\partial g}{\partial\tau}(y)\mathrm{d}s(y)
\pm\frac{1}{2}\nu(x)\frac{\partial g}{\partial\tau}(x), \quad x\in\Gamma_D.
\end{align}
\end{theorem}

\begin{proof}
Using the integration by parts, we have
\begin{align*}
(\nabla\zeta)(x)&=\nabla\int_{\Gamma_D}\frac{\partial\Phi(x,y;\kappa)}{\partial\tau(y)}g(y)\mathrm{d}s(y)=\nabla\int_{\Gamma_D}\tau(y)\cdot\nabla_y\Phi(x,y;\kappa)g(y)\mathrm{d}s(y)\\
&=-\nabla\int_{\Gamma_D}\Phi(x,y;\kappa)\frac{\partial g}{\partial\tau}(y)\mathrm{d}s(y),
\end{align*}		
which implies \eqref{double-type jump} by noting the jump relation \eqref{single jump}.
\end{proof}

\begin{theorem}\label{Th4}
The second derivatives of the single-layer potential $\omega$ with density $g\in C^{1,\alpha}(\Gamma_D)$, $0<\alpha<1$,  can be uniformly extended in a H\"{o}lder continuous fashion from $\mathbb{R}^2\setminus\overline{D}$ into $\mathbb{R}^2\setminus{D}$ and from $D$ into $\overline{D}$ with the limiting values
\begin{align}\label{derder-single1}
&(\nabla\nabla^\top\omega)_{\pm}(x)=-\kappa^2\int_{\Gamma_D}\Phi(x,y;\kappa)\big\langle\nu(y),\nu(y)\big\rangle g(y)\mathrm{d}s(y) -\int_{\Gamma_D}\big\langle\frac{\partial(g\nu)}{\partial\tau}(y),{\bf curl}_x\Phi(x,y;\kappa)\big\rangle\mathrm{d}s(y) \nonumber\\
&+\int_{\Gamma_D}\big\langle\frac{\partial(g\tau)}{\partial\tau}(y),\nabla_x\Phi(x,y;\kappa)\big\rangle\mathrm{d}s(y)\mp\frac{1}{2}\big\langle\frac{\partial(g\nu)}{\partial\tau}(x),\tau(x)\big\rangle\mp\frac{1}{2}\big\langle\frac{\partial(g\tau)}{\partial\tau}(x),\nu(x)\big\rangle, \quad x\in\Gamma_D.
\end{align}
and
\begin{align}\label{derder-single2}
&({\bf curl}\nabla^\top\omega)_{\pm}(x) =\kappa^2\int_{\Gamma_D}\Phi(x,y;\kappa)\big\langle\tau(y),\nu(y)\big\rangle g(y)\mathrm{d}s(y)+\int_{\Gamma_D}\big\langle\frac{\partial(g\tau)}{\partial\tau}(y),{\bf curl}_x\Phi(x,y;\kappa)\big\rangle\mathrm{d}s(y) \nonumber\\
&+\int_{\Gamma_D}\big\langle\frac{\partial(g\nu)}{\partial\tau}(y),\nabla_x\Phi(x,y;\kappa)\big\rangle\mathrm{d}s(y)\pm\frac{1}{2}\big\langle\frac{\partial(g\tau)}{\partial\tau}(x),\tau(x)\big\rangle\mp\frac{1}{2}\big\langle\frac{\partial(g\nu)}{\partial\tau}(x),\nu(x)\big\rangle, \quad x\in\Gamma_D.
\end{align}
\end{theorem}

\begin{proof} 
Using \eqref{symmetry relation}--\eqref{vector identity}, we have from taking the second derivatives of the single-layer potential that
\begin{align*}
(\nabla\nabla^\top\omega)(x)&=\int_{\Gamma_D}\nabla_{x}(\nabla_{x}\Phi(x,y;\kappa))g(y)\mathrm{d}s(y)=-\nabla_x\int_{\Gamma_D}\nabla_{y}\Phi(x,y;\kappa)g(y)\mathrm{d}s(y)\\
&=-\nabla_x\int_{\Gamma_D}\nu(y)\frac{\partial\Phi(x,y;\kappa)}{\partial\nu(y)}g(y)\mathrm{d}s(y)-\nabla_x\int_{\Gamma_D}\tau(y)\frac{\partial\Phi(x,y;\kappa)}{\partial\tau(y)}g(y)\mathrm{d}s(y).
\end{align*}
Combining the above equation and the jump relations \eqref{double jump}--\eqref{double-type jump} gives \eqref{derder-single1}.

Analogously, noting ${\bf curl}~\omega=(\nabla\omega)_\perp$, $\nu_\perp=-\tau$ and $\tau_\perp=\nu$, we have
\begin{align*}
({\bf curl}\nabla^\top\omega)(x)
&=\int_{\Gamma_D}\nabla_{x}(\nabla_{x}\Phi(x,y;\kappa))_\perp g(y)\mathrm{d}s(y)\nonumber\\
&=\nabla_x\int_{\Gamma_D}\tau(y)\frac{\partial\Phi(x,y;\kappa)}{\partial\nu(y)}g(y)\mathrm{d}s(y)-\nabla_x\int_{\Gamma_D}\nu(y)\frac{\partial\Phi(x,y;\kappa)}{\partial\tau(y)}g(y)\mathrm{d}s(y).
\end{align*}
Applying the jump relation \eqref{double jump}--\eqref{double-type jump} again yields \eqref{derder-single2}.
\end{proof}

In view of $\partial_\nu(\nabla\omega)=(\nabla\nabla^\top\omega,\nu)$, $\partial_\nu({\bf curl}~\omega) =({\bf curl}\nabla^\top\omega,\nu)$ and Theorem \ref{Th4}, we have the following result.

\begin{corollary}\label{corollary5}
For the single-layer potential $\omega$ with density $g\in C^{1,\alpha}(\Gamma_D)$, $0<\alpha<1$, we have on $\Gamma_D$ that 
\begin{align*}
\frac{\partial(\nabla\omega)_{\pm}}{\partial\nu}(x)=-\kappa^2\int_{\Gamma_D}\Phi(x,y;\kappa)&\langle\nu(y),\nu(y)\rangle\nu(x)g(y)\mathrm{d}s(y) -\int_{\Gamma_D}\frac{\partial\Phi(x,y;\kappa)}{\partial\tau(x)}\frac{\partial(g\nu)}{\partial\tau}(y)\mathrm{d}s(y) \nonumber\\
&+\int_{\Gamma_D}\frac{\partial\Phi(x,y;\kappa)}{\partial\nu(x)}\frac{\partial(g\tau)}{\partial\tau}(y)\mathrm{d}s(y)\mp\frac{1}{2}\frac{\partial(g\tau)}{\partial\tau}(x)
\end{align*}
and
\begin{align*}
\frac{\partial({\bf curl}~\omega)_{\pm}} {\partial\nu}(x)=\kappa^2\int_{\Gamma_D}\Phi(x,y;\kappa)&\big\langle\tau(y),\nu(y)\big\rangle\nu(x) g(y)\mathrm{d}s(y)+\int_{\Gamma_D}\frac{\partial\Phi(x,y;\kappa)}{\partial\tau(x)}\frac{\partial(g\tau)}{\partial\tau}(y)\mathrm{d}s(y) \nonumber\\
&+\int_{\Gamma_D}\frac{\partial\Phi(x,y;\kappa)}{\partial\nu(x)}\frac{\partial(g\nu)}{\partial\tau}(y)\mathrm{d}s(y)\mp\frac{1}{2}\frac{\partial(g\nu)}{\partial\tau}(x),
\end{align*}
where
$$
\frac{\partial(\nabla\omega)_\pm}{\partial\nu}(x)=\lim_{h\to0^+}\frac{\partial(\nabla\omega)(x\pm h\nu(x))}{\partial\nu(x)}.
$$
\end{corollary}

\subsection{Boundary integral equations}

We introduce the single-layer integral operator and the corresponding far-field integral operator
\begin{align*}
S_\kappa[g](x)&=2\int_{\Gamma_D} \Phi(x,y;\kappa)g(y)\mathrm{d}s(y), \quad x\in\Gamma_D,\\
S_\kappa^\infty [g](\hat{x})&=\gamma_\kappa\int_{\Gamma_D}\mathrm{e}^{-\mathrm{i}\kappa \hat{x}\cdot y}g(y)\mathrm{d}s(y), \quad\hat{x}\in\Omega,
\end{align*}
the normal derivative integral operator
$$
K_\kappa[g](x)=2\int_{\Gamma_D} \frac{\partial\Phi(x,y;\kappa)} {\partial\nu(x)}g(y)\mathrm{d}s(y), \quad x\in\Gamma_D,
$$
and the tangential derivative integral operator
$$
H_\kappa[g](x)=2\int_{\Gamma_D}\frac{\partial\Phi(x,y;\kappa)} {\partial\tau(x)}g(y)\mathrm{d}s(y), \quad x\in\Gamma_D.
$$

Let the solution of \eqref{HelmholtzDec} be given in the form of single-layer potentials, i.e.,
\begin{align}
\phi(x)&=\int_{\Gamma_D}\Phi(x,y;\kappa_{\rm p})g_1(y)\mathrm{d}s(y), \quad
x\in D, \label{singlelayer1}\\
\psi(x)&=\int_{\Gamma_D}\Phi(x,y;\kappa_{\rm s})g_2(y)\mathrm{d}s(y), \quad
x\in D, \label{singlelayer2}\\
u^{\rm s}(x)&=\int_{\Gamma_D}\Phi(x,y;\kappa_{\rm a})g_3(y)\mathrm{d}s(y), \quad x\in\mathbb{R}^2\setminus\overline{D}, \label{singlelayer3}
\end{align} 
where the densities $g_1\in C^{1,\alpha}(\Gamma_D)$, $g_2\in C^{1,\alpha}(\Gamma_D)$, and $g_3\in C^{1,\alpha}(\Gamma_D)$.

Letting $x\in{D}$ tend to boundary $\Gamma_D$ in \eqref{singlelayer1}--\eqref{singlelayer2} and $x\in\mathbb{R}^2\setminus\overline{D}$ tend to boundary $\Gamma_D$ in \eqref{singlelayer3}, using the jump relations of the single-layer potentials, Lemmas \ref{lemma1}--\ref{lemma2}, Corollary \ref{corollary5}, and the boundary conditions of \eqref{HelmholtzDec}, we obtain on $\Gamma_D$ that 
\begin{align}
\begin{split}\label{boundaryIE}
2f_1(x)=&-\mu\kappa^2_{\rm p}\nu^\top S_{\kappa_{\rm p}}\big[\langle\nu,\nu\rangle g_1\big]\nu +\mu\nu^\top K_{\kappa_{\rm p}}\big[\tau\partial_\tau g_1+g_1\partial_\tau\tau\big]-\mu\nu^\top H_{\kappa_{\rm p}}\big[\nu\partial_\tau g_1+g_1\partial_\tau\nu\big] 
\\
&+\mu\kappa^2_{\rm s}\nu^\top S_{\kappa_{\rm s}}\big[\langle\tau,\nu\rangle g_2\big]\nu +\mu\nu^\top K_{\kappa_{\rm s}}\big[\nu\partial_\tau g_2+g_2\partial_\tau\nu\big]+\mu\nu^\top H_{\kappa_{\rm s}}\big[\tau\partial_\tau g_2+g_2\partial_\tau\tau\big] \\
&-(\lambda+\mu)\kappa_{\rm p}^2 S_{\kappa_{\rm p}}[g_1]+S_{\kappa_{\rm a}}[g_3]
+\mu(\nu\cdot\partial_\tau\tau)g_1+\mu(\nu\cdot\partial_\tau\nu)g_2+\mu\partial_\tau g_2,  
\\ 
2f_2(x)=&-\kappa^2_{\rm p}\tau^\top S_{\kappa_{\rm p}}\big[\langle\nu,\nu\rangle g_1\big]\nu+\tau^\top K_{\kappa_{\rm p}}\big[\tau\partial_\tau g_1+g_1\partial_\tau\tau\big]-\tau^\top H_{\kappa_{\rm p}}\big[\nu\partial_\tau g_1+g_1\partial_\tau\nu\big] 
\\
&+\kappa^2_{\rm s}\tau^\top S_{\kappa_{\rm s}}\big[\langle\tau,\nu\rangle g_2\big]\nu+\tau^\top K_{\kappa_{\rm s}}\big[\nu\partial_\tau g_2+g_2\partial_\tau\nu\big]+\tau^\top H_{\kappa_{\rm s}}\big[\tau\partial_\tau g_2+g_2\partial_\tau\tau\big] 
\\
&+(\tau\cdot\partial_\tau\tau)g_1+\partial_\tau g_1 +(\tau\cdot\partial_\tau\nu)g_2,  
\\
2f_3(x)=&K_{\kappa_{\rm p}}[g_1]+H_{\kappa_{\rm s}}[g_2]-K_{\kappa_{\rm a}}[g_3]/(\omega^2\rho_a)+g_1+g_3/(\omega^2\rho_{\rm a}).
\end{split}
\end{align}
We point out that $\nu$ and $\tau$ inside of $[\cdot]$ or $\langle\cdot\rangle$ are with respect to the variable $y$; otherwise $\nu$ and $\tau$ are taken with respect to the variable $x$. For brevity, we shall adopt the same notations in the rest of the paper but they should be clear from the context. The far-field pattern is 
\begin{align} \label{singlelayer_far}
u_\infty(\hat{x})=\gamma_{\rm a}\int_{\Gamma_D}e^{-\mathrm{i}\kappa_{\rm a}\hat{x}\cdot y} g_3(y){\rm d}s(y),  \quad\hat{x}\in\Omega, 
\end{align}
where $\gamma_{\rm a}= e^{\mathrm{i}\pi/4}/{\sqrt{8\kappa_{\rm a}\pi}}$.

Now we discuss the uniqueness and existence of the solution for the integral equations \eqref{boundaryIE}.

\begin{theorem}\label{uniquethmBIE}
There exists at most one solution to the boundary integral equations \eqref{boundaryIE} if $\kappa_{\rm a}$ is not the eigenvalue of the interior Dirichlet problem of the Helmholtz equation in $D$.
\end{theorem}

\begin{proof}
It suffices to show that $g_1=g_2=g_3=0$ if $f_1=f_2=f_3=0$ for equations in \eqref{boundaryIE}. For $x\in\mathbb{R}^2\setminus\Gamma_D$, we define single-layer potentials
	\begin{align*}
	\phi(x)&=\int_{\Gamma_D}\Phi(x,y;\kappa_{\rm p})g_1(y)\mathrm{d}s(y),\\
	\psi(x)&=\int_{\Gamma_D}\Phi(x,y;\kappa_{\rm s})g_2(y)\mathrm{d}s(y), \\
	u^{\rm s}(x)&=\int_{\Gamma_D}\Phi(x,y;\kappa_{\rm a})g_3(y)\mathrm{d}s(y).
	\end{align*} 
Let
	\begin{eqnarray*}
	\phi(x)=\begin{cases}
	\phi_i, &  x\in D\\
	\phi_e, &  x\in \mathbb{R}^2\backslash \overline{D}
	\end{cases},\quad 
	\psi(x)=\begin{cases}
	\psi_i, &  x\in D\\
	\psi_e, &  x\in \mathbb{R}^2\backslash \overline{D}
	\end{cases},
	\end{eqnarray*}
	and 
	\begin{eqnarray*}
	u^{\rm s}(x)=\begin{cases}
	u^{\rm s}_i, &  x\in D\\
	u^{\rm s}_e, &  x\in \mathbb{R}^2\backslash \overline{D}
	\end{cases}.
	\end{eqnarray*}
	Since $\phi_i$, $\psi_i$ and $u^{\rm s}_e$ satisfy the boundary value problem \eqref{HelmholtzDec}, they are identically zero by Theorem \ref{uniquethm}. Using the jump condition of single layer potentials, we have on $\partial D$ that 
	\begin{eqnarray}
	&\phi_e-\phi_i=0, \quad \psi_e-\psi_i=0, \label{equ_phi_psi}\\
	&\partial_\nu \phi_e - \partial_\nu \phi_i = -g_1, \quad \partial_\nu \psi_e - \partial_\nu \psi_i = -g_2. \label{equ_dphi_dpsi}
	\end{eqnarray}
	Combining \eqref{equ_phi_psi} and the fact $\phi_i=\psi_i=0$ in $D$, we derive that $\phi_e$ and $\psi_e$ satisfy the zero boundary condition on $\partial D$. By the uniqueness of the exterior problem for the Helmholtz equation, it holds that $\phi_e=\psi_e=0$ in $\mathbb{R}^2\backslash \overline{D}$. We conclude that $g_1=g_2=0$ by \eqref{equ_dphi_dpsi}. Similarly, we have on $\partial D$ that    
	\begin{eqnarray}
	&u^{\rm s}_e-u^{\rm s}_i=0, \label{equ_u}\\
	&\partial_\nu u^{\rm s}_e - \partial_\nu u^{\rm s}_i = -g_3. \label{equ_du}
	\end{eqnarray}
	By \eqref{equ_u}, we see that $u^{\rm s}_i$ satisfies the zero Dirichlet boundary condition. Since $\kappa_{\rm a}$ is not the eigenvalue of the interior Dirichlet problem, we conclude that $u^{\rm s}_i$ is identically zero in $D$, which implies $g_3=0$ by \eqref{equ_du}.
\end{proof}

\begin{theorem}\label{existence}
There exists a unique solution to the boundary integral equations \eqref{boundaryIE} if none of $\kappa_{\rm p}$, $\kappa_{\rm s}$ and $\kappa_{\rm a}$ is the eigenvalue of the interior Dirichlet problem of the Helmholtz equation in $D$.
\end{theorem}

\begin{proof}
Since the original coupled equations \eqref{Navier equation}--\eqref{radiation} admit a unique solution $(u,\boldsymbol{U})$, by the Helmholtz decomposition $\boldsymbol{U} = \nabla\phi+\mathbf{curl}~\psi$, we have on $\partial D$ that 
\[
	\partial_\nu \phi +\partial_\tau \psi  = \boldsymbol{U}\cdot \nu,\quad 
	\partial_\tau \psi -\partial_\nu \psi  = \boldsymbol{U}\cdot \tau. 
\]
Plugging the single layer representations \eqref{singlelayer1}--\eqref{singlelayer2} for $\phi$ and $\psi$ into the above equations and using the jump property of boundary integral operators, we have 
	\begin{eqnarray}\label{single_coupled}
	\begin{cases}
	(I+K_{\kappa_{\rm p}})g_1+H_{\kappa_{\rm s}} g_2 = 2\boldsymbol{U}\cdot \nu \\
	H_{\kappa_{\rm p}}g_1-(I+K_{\kappa_{\rm s}}) g_2 = 2\boldsymbol{U}\cdot \tau
	\end{cases}
	\end{eqnarray}
	where $I$ is the identity operator. 
	Following the idea in \cite{LaiLi}, we can show that the boundary integral equation \eqref{single_coupled} admits a solution $(g_1,g_2)$ when neither of $\kappa_{\rm p}$ or $\kappa_{\rm s}$ is the eigenvalue of the interior Dirichlet problem of the Helmholtz equation in $D$. For $g_3$, since $\kappa_{\rm a}$ is not an interior Dirichlet eigenvalue either, the single layer operator $S_{\kappa_{\rm a}}$ is invertible. Therefore
	\begin{align*}
	g_3 = S_{\kappa_{\rm a}}^{-1}u^s.
	\end{align*}
	Based on the construction, one can easily see that $g_1$, $g_2$ and $g_3$ satisfy the boundary integral equations \eqref{boundaryIE}. 
\end{proof}

\section{translation invariance and a uniqueness result}

In this section, we prove the translation invariance of the phaseless far-field pattern and present a uniqueness result for the phaseless IAEIP.
 
\begin{theorem}\label{transinvarance}
Under the assumption of Theorem \ref{existence}, let $u_\infty$ be the far field pattern of the scattered waves $u^{\rm s}$ with the incident plane wave ${u}^{\rm inc}(x)=e^{{\rm i} \kappa_{\rm a}d\cdot x}$. For the shifted domain $D_h:=\{x+h: x\in D\}$ with a fixed vector $h\in\mathbb{R}^2$, the far-field pattern $u_\infty^h$ satisfies
\begin{align*}
u_\infty^h(\hat{x})=e^{\mathrm{i}\kappa_{\rm a}(d-\hat{x})\cdot h}u_\infty(\hat{x}).
\end{align*}	 
\end{theorem}

\begin{proof}
We assume that the densities $g_1^h, g_2^h$ and $g_3^h$ solve the boundary integral equations \eqref{boundaryIE} with $\Gamma_D$ replaced by $\Gamma_{D_h}$. We claim that if $g_1, g_2$ and $g_3$ solve the equations \eqref{boundaryIE}, then
\begin{equation}  \label{density_invar}
g_1^h(x)=e^{\mathrm{i}\kappa_{\rm a}d\cdot h}g_1(x-h), \quad
g_2^h(x)=e^{\mathrm{i}\kappa_{\rm a}d\cdot h}g_2(x-h), \quad
g_3^h(x)=e^{\mathrm{i}\kappa_{\rm a}d\cdot h}g_3(x-h).
\end{equation} 
In fact, by substituting the above equations into the right hand side of 
\eqref{boundaryIE} with $\Gamma_D$ replaced by $\Gamma_{D_h}$ and noting that
\[
\frac{\partial\Phi(x,y;\kappa)}{\partial\nu(x)}=\frac{\partial\Phi(\tilde{x},\tilde{y};\kappa)}{\partial\nu(\tilde{x})}, \qquad \frac{\partial\Phi(x,y;\kappa)}{\partial\tau(x)}=\frac{\partial\Phi(\tilde{x},\tilde{y};\kappa)}{\partial\tau(\tilde{x})},
\]
where $\tilde{x}=x-h$, $\tilde{y}=y-h$, the first equation in \eqref{boundaryIE} becomes
\begin{align*}
&-\mu\kappa^2_{\rm p}S^h_{\kappa_{\rm p}}\big[(\nu(x)\cdot\nu(y))^2g_1^h(y)\big] +\mu K^h_{\kappa_{\rm p}}\big[\nu(x)\cdot\tau(y)\partial_\tau g_1^h(y) +g_1^h(y)\nu(x)\cdot\partial_\tau\tau(y)\big]
\nonumber\\
&-\mu H^h_{\kappa_{\rm p}}\big[\nu(x)\cdot\nu(y)\partial_\tau g_1^h(y) +g_1^h(y)\nu(x)\cdot\partial_\tau\nu(y)\big]+\mu\kappa^2_{\rm s}S^h_{\kappa_{\rm s}}[\nu(x)\cdot\tau(y)\nu(y)\cdot\nu(x)g_2^h(y)]
\nonumber\\
&+\mu K^h_{\kappa_{\rm s}}\big[\nu(x)\cdot\nu(y)\partial_\tau g_2^h(y) +g_2^h(y)\nu(x)\cdot\partial_\tau\nu(y)\big]-(\lambda+\mu)\kappa_{\rm p}^2 S^h_{\kappa_{\rm p}}\big[g_1^h(y)\big]
\nonumber\\
&+\mu H^h_{\kappa_{\rm s}}\big[\nu(x)\cdot\tau(y)\partial_\tau g_2^h(y) +g_2^h(y)\nu(x)\cdot\partial_\tau\tau(y)\big] 
+S^h_{\kappa_{\rm a}}\big[g_3^h(y)\big]
\nonumber\\
&+\mu g_1^h(x)\nu(x)\cdot\partial_\tau\tau(x) +\mu g_2^h(x)\nu(x)\cdot\partial_\tau\nu(x)+\mu\partial_\tau g_2^h(x) 
\nonumber\\
=&e^{\mathrm{i}\kappa_{\rm a}d\cdot h} \big(-\mu\kappa^2_{\rm p}S_{\kappa_{\rm p}}\big[(\nu(\tilde{x})\cdot\nu(\tilde{y}))^2g_1(\tilde{y})\big]+\mu K_{\kappa_{\rm p}}\big[\nu(\tilde{x})\cdot\tau(\tilde{y})\partial_\tau g_1(\tilde{y}) +g_1(\tilde{y})\nu(\tilde{x})\cdot\partial_\tau\tau(\tilde{y})\big]
\nonumber\\
&-\mu H_{\kappa_{\rm p}}\big[\nu(\tilde{x})\cdot\nu(\tilde{y})\partial_\tau g_1(\tilde{y})+g_1(\tilde{y})\nu(\tilde{x})\cdot\partial_\tau\nu(\tilde{y})\big]+\mu\kappa^2_{\rm s}S_{\kappa_{\rm s}}[\nu(\tilde{x})\cdot\tau(\tilde{y})\nu(\tilde{y})\cdot\nu(\tilde{x})g_2(\tilde{y})]
\nonumber\\
&+\mu K_{\kappa_{\rm s}}\big[\nu(\tilde{x})\cdot\nu(\tilde{y})\partial_\tau g_2(\tilde{y})+g_2(\tilde{y})\nu(\tilde{x})\cdot\partial_\tau\nu(\tilde{y})\big]-(\lambda+\mu)\kappa_{\rm p}^2 S_{\kappa_{\rm p}}\big[g_1(\tilde{y})\big]
\nonumber\\
&+\mu H_{\kappa_{\rm s}}\big[\nu(\tilde{x})\cdot\tau(\tilde{y})\partial_\tau g_2(\tilde{y})+g_2(\tilde{y})\nu(\tilde{x})\cdot\partial_\tau\tau(\tilde{y})\big] 
+S_{\kappa_{\rm a}}\big[g_3(\tilde{y})\big]
\nonumber\\
&+\mu g_1(\tilde{x})\nu(\tilde{x}) \cdot\partial_\tau\tau(\tilde{x}) +\mu g_2(\tilde{x})\nu(\tilde{x})\cdot\partial_\tau\nu(\tilde{x})+\mu\partial_\tau g_2(\tilde{x})\big)
\nonumber\\
=&e^{\mathrm{i}\kappa_{\rm a}d\cdot h}  \Big(2u^{\rm inc}(\tilde{x})\Big)=2f_1(x),\quad x\in\partial D_h,
\end{align*}
where the boundary integral operators defined on $\Gamma_{D_h}$ are denoted by $S^h, K^h$ and $H^h$. Similarly, the second and third equations of \eqref{boundaryIE} can be verified in the same way. Thus, \eqref{density_invar} follows from the fact that the system of boundary integral equations \eqref{boundaryIE} for $D_h$ has a unique solution (cf. Theorems \ref{uniquethmBIE} and \ref{existence}).

Combining \eqref{singlelayer_far} and \eqref{density_invar}, we obtain 
\begin{align*}
u_\infty^h(\hat{x})&=\gamma_{\rm a}\int_{\Gamma_{D_h}}e^{-\mathrm{i}\kappa_{\rm a} \hat{x}\cdot y}g_3^h(y)\mathrm{d}s(y) \\
&=\gamma_{\rm a}\int_{\Gamma_{D_h}} e^{-\mathrm{i}\kappa_{\rm a} \hat{x}\cdot(y-h)} e^{-\mathrm{i}\kappa_{\rm a}\hat{x}\cdot h}e^{\mathrm{i}\kappa_{\rm a} d\cdot h} g_3(y-h)\mathrm{d}s(y) \\
&=e^{\mathrm{i}\kappa_{\rm a} (d-\hat{x})\cdot h}u_\infty(\hat{x}),
\end{align*}
which completes the proof. 
\end{proof}

\begin{remark}
The assumption on the wavenumber $\kappa_{\rm p}$, $\kappa_{\rm s}$ and $\kappa_{\rm a}$ in Theorem \ref{transinvarance} can be removed by using combined boundary layer potentials to represent the solution. We refer to \cite{LaiLi} for a related discussion. 
\end{remark}

Theorem \ref{transinvarance} implies that the location of the obstacle can not be uniquely recovered by the modules of far-field pattern when the plane wave is used as an incident field. To overcome this difficulty, motivated by \cite{ZG18}, we may introduce an elastic reference ball $B=B(x_0,R)=\{x\in\mathbb{R}^2: |x-x_0|<R\}$ to the scattering system in order to break the translation invariance.

Assume that $P$ is a disk (with positive radius) such that $P\subset\mathbb{R}^2\setminus(\overline{D}\cup\overline{B})$ and $\kappa_{\rm a}^2$ is not a Dirichlet eigenvalue of $-\Delta$ in $P$. Denote the boundary of $P$ by $\partial P$. Consider that the incident wave is given by a plane  wave $u^{\rm inc}(x,d)$ and a point source $v^{\rm inc}(x,z)$, i.e., $u^{\rm inc}(x,d)=e^{{\rm i}\kappa_{\rm a}x\cdot d}$ and $v^{\rm inc}(x,z)=\Phi(x,z;\kappa_{\rm a})$, where $z\in\partial P$ is the source location. Assume further that $\{u_{D\cup B}^{\rm s}(x)(x,d), u_{D\cup B}^\infty(\hat{x},z)\}$ and $\{v_{D\cup B}^{\rm s}(x,z), v_{D\cup B}^\infty(\hat{x},z)\}$ are the scattered field and the far-field pattern generated by $D\cup B$ corresponding to the incident field $u^{\rm inc}(x,d)$ and $v^{\rm inc}(x,z)$, respectively.

Now we present a uniqueness result for the phaseless inverse scattering problem. A similar uniqueness result may be found in \cite[Theorem 4.1]{ZG18} for the phaseless inverse medium scattering problem. 

\begin{theorem}\label{unithmphaseless}
Let $D_1$ and $D_2$ be two elastic obstacles with $\mathcal{C}^2$ boundaries, and $\omega$ is not a Jones frequency either for $D_1$ or $D_2$.
Suppose that the far-field patterns satisfy the following conditions:
\begin{align}
|u_{D_1\cup B}^\infty(\hat{x},d_0)|&=|u_{D_2\cup B}^\infty(\hat{x},d_0)|, \qquad&&\forall\hat{x}\in\Omega, \label{unicondi1}\\
|v_{D_1\cup B}^\infty(\hat{x},z)|&=|v_{D_2\cup B}^\infty(\hat{x},z)|, \qquad&&\forall(\hat{x},z)\in\Omega\times\partial{P}, \label{unicondi2}\\
|u_{D_1\cup B}^\infty(\hat{x},d_0)+v_{D_1\cup B}^\infty(\hat{x},z)|&=|u_{D_2\cup B}^\infty(\hat{x},d_0)+v_{D_2\cup B}^\infty(\hat{x},z)|, \qquad&&\forall(\hat{x},z)\in\Omega\times\partial{P} \label{unicondi3}
\end{align}
for a fixed $d_0\in\Omega$, then $D_1=D_2$.
\end{theorem}

\begin{proof} 	
By \eqref{unicondi1}--\eqref{unicondi3}, we have
$$
\Re{\rm e}\Big\{u_{D_1\cup B}^\infty(\hat{x},d_0)\overline{v_{D_1\cup B}^\infty(\hat{x},z)}\Big\}=\Re{\rm e}\Big\{u_{D_2\cup B}^\infty(\hat{x},d_0)\overline{v_{D_2\cup B}^\infty(\hat{x},z)}\Big\}, \quad \forall\hat{x}\in\Omega, z\in\partial{P}.
$$
In view of \eqref{unicondi1} and \eqref{unicondi2}, we assume that 
$$
u_{D_j\cup B}^\infty(\hat{x},d_0)=r(\hat{x},d_0)e^{{\rm i}\alpha_j(\hat{x},d_0)}, \quad v_{D_j\cup B}^\infty(\hat{x},z)=s(\hat{x},z)e^{{\rm i}\beta_j(\hat{x},z)}, \quad j=1, 2,
$$
where $r(\hat{x},d_0)=|u_{D_j\cup B}^\infty(\hat{x},d_0)|$, $s(\hat{x},z)=|v_{D_j\cup B}^\infty(\hat{x},z)|$, $\alpha_j(\hat{x},d_0)$ and $\beta_j(\hat{x},z)$ are real-valued functions, $j=1, 2$. Following the proof of \cite[Theorem 3.16]{DR-book2013}, we can show the mixed reciprocity relation
\begin{align}\label{mixedreci}
\gamma_{\rm a}v^\infty(\hat{x},z)=u^{\rm s}(z,-\hat{x}).
\end{align}
Using \eqref{mixedreci} and similar arguments in \cite[Theorem 3.1]{ZG18}, we obtain that
\begin{align}\label{argu}
u_{D_1\cup B}^{\rm s}(x,d)=e^{{\rm i}\gamma(-d)}u_{D_2\cup B}^{\rm s}(x,d), \quad \forall x\in\mathbb{R}^2\setminus(D_1\cup D_2\cup B),\quad -d\in S,
\end{align}
where $\gamma(\hat{x}):=\alpha_1(\hat{x},d_0)-\alpha_2(\hat{x},d_0)-2m\pi$, $\hat{x}\in S$, $m\in\mathbb{Z}$, and $S\subset\Omega$ is an open arc. Furthermore, for $x\in\Gamma_B$, $-d\in S$, we get
$$
u_{D_1\cup B}^{\rm s}(x,d)=e^{{\rm i}\gamma(-d)}u_{D_2\cup B}^{\rm s}(x,d), \quad \frac{\partial u_{D_1\cup B}^{\rm s}(x,d)}{\partial\nu}=e^{{\rm i}\gamma(-d)}\frac{\partial u_{D_2\cup B}^{\rm s}(x,d)}{\partial\nu}.
$$
Noting that the total fields $(u_{D_1\cup B},\boldsymbol{U}_{D_1\cup B})$ and $(u_{D_2\cup B},\boldsymbol{U}_{D_2\cup B})$ are the solutions of \eqref{Navier equation}--\eqref{radiation} corresponding to the scatterers $D_1\cup B$ and $D_2\cup B$, respectively, we find that
$$
\tilde{u}(x,d):=u_{D_1\cup B}(x,d)-e^{{\rm i}\gamma(-d)}u_{D_2\cup B}(x,d),\quad\widetilde{\boldsymbol{U}}(x,d):=\boldsymbol{U}_{D_1\cup B}(x,d)-e^{{\rm i}\gamma(-d)}\boldsymbol{U}_{D_2\cup B}(x,d)
$$
satisfy the Navier equation and the Helmholtz equation
\begin{align*}
\mu\Delta\widetilde{\boldsymbol{U}}+(\lambda+\mu)\nabla\nabla\cdot\widetilde{\boldsymbol{U}}+\omega^2\rho_{\rm e}\widetilde{\boldsymbol{U}}&=0 \quad{\rm in~} {B}, \\
\Delta \tilde{u}+\kappa_{\rm a}^2\tilde{u}&=0\quad {\rm in~}\mathbb{R}^2\setminus\overline{D_1\cup D_2\cup B}, 
\end{align*}
and the transmission conditions on $\Gamma_B$
\begin{equation}\label{transmission34}
T(\widetilde{\boldsymbol{U}})=-\tilde{u}\nu, \quad 
\widetilde{\boldsymbol{U}}\cdot\nu=\frac{1}{\omega^2\rho_{\rm a}}\partial_\nu\tilde{u}.
\end{equation}

Suppose that $\left(w(x,d_0),\boldsymbol{W}(x,d_0)\right)$ is the solution of \eqref{Navier equation}--\eqref{radiation} corresponding to the single reference ball $B$ with incident plane wave $u^{\rm inc}(x,d_0)$, then the far-field $w^\infty(\hat{x},d_0)\not\equiv0$, $\hat{x}\in\Omega$. Using the Betti formula and the transmission condition \eqref{transmission34}, and noting the identity
$$
\tilde{u}(x,d)=\left(1-e^{{\rm i}\gamma(-d)}\right)u^{\rm inc}(x,d), \quad\frac{\partial\tilde{u}}{\partial\nu}(x,d)=\left(1-e^{{\rm i}\gamma(-d)}\right)\frac{\partial u^{\rm inc}}{\partial\nu}\left(x,d\right), \forall x \in \Gamma_B, -d\in S,
$$
we have 
\begin{align*}
0&=\int_{\Gamma_B}\Big\{T(\boldsymbol{W})(y,d_0)\cdot\widetilde{\boldsymbol{U}}(y,d)-T(\widetilde{\boldsymbol{U}})(y,d)\cdot\boldsymbol{W}(y,d_0)\Big\}\mathrm{d}s(y) \\
&=\frac{-1}{\omega^2\rho_{\rm a}}\int_{\Gamma_B}\Big\{w(y,d_0)\frac{\partial\tilde{u}(y,d)}{\partial\nu}-\tilde{u}(y,d)\frac{\partial w(y,d_0)}{\partial\nu}\Big\}\mathrm{d}s(y) \\
&=\frac{-(1-e^{{\rm i}\gamma(-d)})}{\omega^2\rho_{\rm a}} \int_{\Gamma_B}\Big\{w(y,d_0)\frac{\partial u^{\rm inc}(y,d)}{\partial\nu} -u^{\rm inc}(y,d)\frac{\partial w(y,d_0)}{\partial\nu}\Big\}\mathrm{d}s(y) \\
&=\frac{-(1-e^{{\rm i}\gamma(-d)})}{\omega^2\rho_{\rm a}}w^\infty(-d,d_0), \quad \forall -d\in S. 
\end{align*}

We claim $|w^\infty(-d,d_0)|\not\equiv0$, $\forall-d\in S$. Otherwise, we obtain by using the analytic continuation that  $w^\infty(\hat{x},d_0)=0$, $\forall \hat{x}\in\Omega$. This is a contradiction. By continuity, there exists an open curve $\tilde{S}\subset S$, such that $|w^\infty(-d,d_0)|\not =0$, $\forall-d\in \tilde{S}$, which implies that $e^{{\rm i}\gamma(-d)}=1$ for $-d\in \tilde{S}$. From \eqref{argu}, we have
$$
u_{D_1\cup B}^\infty(\hat{x},d)=u_{D_2\cup B}^\infty(\hat{x},d),\quad\forall(\hat{x},-d)\in\Omega\times \tilde{S}.
$$
Again, using the reciprocity relation and the analyticity of $u_{D_j\cup B}^\infty(\hat{x},d)$ for $j=1, 2$, we obtain that the far-field patterns
$u_{D_1\cup B}^\infty$ and $u_{D_2\cup B}^\infty$ coincide for all observation and incident directions $\hat{x},d\in\Omega$. We conclude from \cite[Theorem 4.1]{MS2009} that $D_1=D_2$.
\end{proof}

\begin{remark}
	In view of the proof of Theorem 4.3, we can also assume that the scatterers $D_1, D_2$ and the reference ball $B$ possess different mass densities and Lam\'{e} parameters.
\end{remark}

\begin{remark}
	Theorem \ref{unithmphaseless} only gives a sufficient condition to uniquely reconstruct the the unknown obstacle $D$ with phaseless data. We expect the uniqueness result also holds with much less data. In particular, for our numerical practice, we do not use the phaseless far-field data generated by the point source $v^{\rm inc}(x,z)$ with $z\in \partial P$, although in this case, the uniqueness result is still under investigation. 
\end{remark}

\section{Nystr\"{o}m-type discretization for boundary integral equations}

In this section, we introduce a Nystr\"{o}m-type discretization for the boundary integral equations and present some effective numerical quadrature to handle the singular integrals. 

\subsection{Parametrization}

For simplicity, the boundary $\Gamma_D$ is assumed to be a starlike curve with the parametrized form
\begin{equation*}
\Gamma_D=\{p(\hat{x})=c+r(\hat{x})\hat{x}; ~c=(c_1,c_2)^\top,\ \hat{x}\in\Omega\},
\end{equation*} 
where $\Omega=\{\hat{x}(t)=(\cos t, \sin t)^\top; ~0\leq t< 2\pi\}$.
We introduce the parametrized integral operators which are still represented by $S_\kappa$, $S_\kappa^\infty$, $K_\kappa$ and $H_\kappa$ for convenience, i.e.,
\begin{align*}
\big(S_\kappa[\vartheta;p]\big)(t)&=\int_0^{2\pi}\widetilde M(t,\varsigma;\kappa) \vartheta(\varsigma)\mathrm{d}\varsigma,\\
\big(S^\infty_\kappa[\vartheta;p]\big)(t)&=\gamma_{\kappa}\int_0^{2\pi}e^{-\mathrm{i}\kappa\hat{x}(t)\cdot p(\varsigma)}\vartheta(\varsigma)\mathrm{d}\varsigma,
\\
\big(K_\kappa[\vartheta;p]\big)(t)&=\frac{1}{G(t)}\int_0^{2\pi}\widetilde K(t,\varsigma;\kappa)\vartheta(\varsigma)\mathrm{d}\varsigma, \\
\big(H_\kappa[\vartheta;p]\big)(t)&=\frac{1}{G(t)}\int_0^{2\pi}  \widetilde
H(t,\varsigma;\kappa)\vartheta(\varsigma)\mathrm{d}\varsigma,
\end{align*}
where $\vartheta(\varsigma)=G(\varsigma)g(p(\varsigma))$, $G(\varsigma):=|p'(\varsigma)|=\sqrt{(r'(\varsigma))^2+r^2(\varsigma)}$ is the Jacobian of the transformation,
\begin{align*}
\widetilde
M(t,\varsigma;\kappa)&=\frac{\mathrm{i}}{2}H_0^{(1)}(\kappa|p(t)-p(\varsigma)|),
\\
\widetilde K(t,\varsigma;\kappa)&=\frac{\mathrm{i}\kappa}{2}\mathsf{n}(t)\cdot[p(\varsigma)-p(t)]
\frac{H_1^{(1)}(\kappa|p(t)-p(\varsigma)|)}{|p(t)-p(\varsigma)|},
\\
\widetilde H(t,\varsigma;\kappa)&=\frac{\mathrm{i}\kappa}{2}\mathsf{n}(t)^\perp\cdot[p(\varsigma)-p(t)]
\frac{H_1^{(1)}(\kappa|p(t)-p(\varsigma)|)}{|p(t)-p(\varsigma)|},
\end{align*}
and
\begin{align*}
\mathsf{n}(t)&:=\tilde{\nu}(t)|p'(t)|=\Big(p'_2(t), -p'_1(t)\Big)^\top, \quad\tilde{\nu}=\nu\circ p,\\
\mathsf{n}(t)^\perp&:=\tilde{\tau}(t)|p'(t)|=\Big(p'_1(t), p'_2(t)\Big)^\top, \quad\tilde{\tau}=\tau\circ p.
\end{align*}
Thus, \eqref{boundaryIE} can be reformulated as the parametrized integral equations
\begin{align}
\begin{split} \label{paraboundaryIE}
w_1=&-\mu\kappa^2_{\rm p}\tilde{\nu}^\top S_{\kappa_{\rm p}}\big[\langle\tilde{\nu},\tilde{\nu}\rangle\varphi_1G;p\big]\tilde{\nu} +\mu\tilde{\nu}^\top K_{\kappa_{\rm p}}\big[\tilde{\tau}\varphi'_1+\tilde{\tau}'\varphi_1;p\big]
-\mu\tilde{\nu}^\top H_{\kappa_{\rm p}}\big[\tilde{\nu}\varphi'_1+\tilde{\nu}'\varphi_1;p\big] \\
&+\mu\kappa^2_{\rm s}\tilde{\nu}^\top S_{\kappa_{\rm s}}\big[\langle\tilde{\tau},\tilde{\nu}\rangle \varphi_2G;p\big]\tilde{\nu} +\mu\tilde{\nu}^\top K_{\kappa_{\rm s}}\big[\tilde{\nu}\varphi'_2+\tilde{\nu}'\varphi_2;p\big]+\mu\tilde{\nu}^\top H_{\kappa_{\rm s}}\big[\tilde{\tau}\varphi'_2+\tilde{\tau}'\varphi_2;p\big] \\
&-(\lambda+\mu)\kappa_{\rm p}^2 S_{\kappa_{\rm p}}[\varphi_1G;p]+S_{\kappa_{\rm a}}[\varphi_3G;p]
+\mu(\tilde{\nu}\cdot\tilde{\tau}')\varphi_1/G+\mu(\tilde{\nu}\cdot\tilde{\nu}')\varphi_2/G+\mu\varphi'_2/G,  \\ 
w_2=&-\kappa^2_{\rm p}\tilde{\tau}^\top S_{\kappa_{\rm p}}\big[\langle\tilde{\nu},\tilde{\nu}\rangle\varphi_1G;p\big]\tilde{\nu}+\tilde{\tau}^\top K_{\kappa_{\rm p}}\big[\tilde{\tau}\varphi'_1+\tilde{\tau}'\varphi_1;p\big]-\tilde{\tau}^\top H_{\kappa_{\rm p}}\big[\tilde{\nu}\varphi'_1+\tilde{\nu}'\varphi_1;p\big] \\
&+\kappa^2_{\rm s}\tilde{\tau}^\top S_{\kappa_{\rm s}}\big[\langle\tilde{\tau},\tilde{\nu}\rangle \varphi_2G;p\big]\tilde{\nu}+\tilde{\tau}^\top K_{\kappa_{\rm s}}\big[\tilde{\nu}\varphi'_2+\tilde{\nu}'\varphi_2;p\big]+\tilde{\tau}^\top H_{\kappa_{\rm s}}\big[\tilde{\tau}\varphi'_2+\tilde{\tau}'\varphi_2;p\big] \\
&+(\tilde{\tau}\cdot\tilde{\tau}')\varphi_1/G+\varphi'_1/G+(\tilde{\tau}\cdot\tilde{\nu}')\varphi_2/G, \\ 
w_3=&K_{\kappa_{\rm p}}[\varphi_1G;p]+H_{\kappa_{\rm s}}[\varphi_2G;p]-K_{\kappa_{\rm a}}[\varphi_3G;p]/(\omega^2\rho_{\rm a})+\varphi_1+\varphi_3/(\omega^2\rho_{\rm a}), 
\end{split}
\end{align}
where $w_j=2(f_j\circ p)$, $\varphi_j=(g_j\circ p)$, ${\varphi}'_j=(g_j\circ p)'$, $j=1, 2, 3$, and $\tilde{\tau}':=(\tilde{\tau}'_1,\tilde{\tau}'_2)^\top$, $\tilde{\nu}':=(\tilde{\nu}'_1,\tilde{\nu}'_2)^\top$.

To avoid calculating the derivative of the Jacobi $G$ in numerical discretization, we transform the parametrized integral equations \eqref{paraboundaryIE} to  
\begin{align}
\begin{split}\label{pboundaryIE}
w_1=&-\mu\kappa^2_{\rm p}\tilde{\nu}^\top S_{\kappa_{\rm p}}\big[\langle \mathsf{n},\mathsf{n}\rangle\tilde{\varphi}_1;p\big]\tilde{\nu} +\mu\tilde{\nu}^\top K_{\kappa_{\rm p}}\big[\mathsf{n}^\perp{\tilde{\varphi}_1}'+{\mathsf{n}^\perp}'\tilde{\varphi}_1;p\big]
-\mu\tilde{\nu}^\top H_{\kappa_{\rm p}}\big[\mathsf{n}{\tilde{\varphi}_1}'+\mathsf{n}'\tilde{\varphi}_1;p\big] \\
&+\mu\kappa^2_{\rm s}\tilde{\nu}^\top S_{\kappa_{\rm s}}\big[\langle \mathsf{n}^\perp,\mathsf{n}\rangle\tilde{\varphi}_2;p\big]\tilde{\nu} +\mu\tilde{\nu}^\top K_{\kappa_{\rm s}}\big[\mathsf{n}\tilde{\varphi}'_2+\mathsf{n}'\tilde{\varphi}_2;p\big]+\mu\tilde{\nu}^\top H_{\kappa_{\rm s}}\big[\mathsf{n}^\perp{\tilde{\varphi}_2}'+{\mathsf{n}^\perp}'\tilde{\varphi}_2;p\big] \\
&-(\lambda+\mu)\kappa_{\rm p}^2 S_{\kappa_{\rm p}}[\tilde{\varphi}_1G^2;p]+S_{\kappa_{\rm a}}[\tilde{\varphi}_3G^2;p]
+\mu(\tilde{\nu}\cdot{\mathsf{n}^\perp}')\tilde{\varphi}_1/G+\mu(\tilde{\nu}\cdot \mathsf{n}')\tilde{\varphi}_2/G+\mu\tilde{\varphi}'_2, \\ 
w_2=&-\kappa^2_{\rm p}\tilde{\tau}^\top S_{\kappa_{\rm p}}\big[\langle\mathsf{n},\mathsf{n}\rangle\tilde{\varphi}_1;p\big]\tilde{\nu}+\tilde{\tau}^\top K_{\kappa_{\rm p}}\big[\mathsf{n}^\perp{\tilde{\varphi}_1}'+{\mathsf{n}^\perp}'\tilde{\varphi}_1;p\big]-\tilde{\tau}^\top H_{\kappa_{\rm p}}\big[\mathsf{n}{\tilde{\varphi}_1}'+\mathsf{n}'\tilde{\varphi}_1;p\big] \\
&+\kappa^2_{\rm s}\tilde{\tau}^\top S_{\kappa_{\rm s}}\big[\langle\mathsf{n}^\perp,\mathsf{n}\rangle\tilde{\varphi}_2;p\big]\tilde{\nu}+\tilde{\tau}^\top K_{\kappa_{\rm s}}\big[\mathsf{n}\tilde{\varphi}'_2+\mathsf{n}'\tilde{\varphi}_2;p\big]+\tilde{\tau}^\top H_{\kappa_{\rm s}}\big[\mathsf{n}^\perp{\tilde{\varphi}_2}'+{\mathsf{n}^\perp}'\tilde{\varphi}_2;p\big] \\
&+(\tilde{\tau}\cdot{\mathsf{n}^\perp}')\tilde{\varphi}_1/G+\tilde{\varphi}'_1+(\tilde{\tau}\cdot\mathsf{n}')\tilde{\varphi}_2/G,  \\ 
w_3=&K_{\kappa_{\rm p}}[\tilde{\varphi}_1G^2;p]+H_{\kappa_{\rm s}}[\tilde{\varphi}_2G^2;p]-K_{\kappa_{\rm a}}[\tilde{\varphi}_3G^2;p]/(\omega^2\rho_{\rm a})+\tilde{\varphi}_1G+\tilde{\varphi}_3G/(\omega^2\rho_{\rm a}),
\end{split}
\end{align}
where $\tilde{\varphi}_l=\varphi_l/G$, $l=1,2,3$, $\mathsf{n}'=(p_2'',-p_1'')^\top$, and ${\mathsf{n}^\perp}'=(p_1'',p_2'')^\top$.

\subsection{Discretization}

The kernel $\widetilde{M}$ and $\widetilde{K}$ of the parametrized single-layer and normal derivative integral operators can be written in form of
\begin{align*}
&\widetilde{M}(t,\varsigma;\kappa)=\widetilde{M}_1(t,\varsigma;\kappa)\ln\bigg(4\sin^2\frac{t-\varsigma}{2}\bigg)+\widetilde{M}_2(t,\varsigma;\kappa),\\
&\widetilde{K}(t,\varsigma;\kappa)=\widetilde{K}_1(t,\varsigma;\kappa)\ln\bigg(4\sin^2\frac{t-\varsigma}{2}\bigg)+\widetilde{K}_2(t,\varsigma;\kappa),
\end{align*}
where
\begin{align*}
\widetilde M_1(t,\varsigma;\kappa)&= -\frac{1}{2\pi}J_0(\kappa|p(t)-p(\varsigma)|), \\
\widetilde M_2(t,\varsigma;\kappa)&=\widetilde M(t,\varsigma;\kappa)-\widetilde M_1(t,\varsigma;\kappa)\ln\bigg(4\sin^2\frac{t-\varsigma}{2}\bigg),\\
\widetilde K_1(t,\varsigma;\kappa)&= \frac{\kappa}{2\pi}\mathsf{n}(t)\cdot\big[p(t)-p(\varsigma)\big]\frac{J_1(\kappa|p(t)-p(\varsigma)|)}{|p(t)-p(\varsigma)|}, \\
\widetilde K_2(t,\varsigma;\kappa)&=\widetilde K(t,\varsigma;\kappa)-\widetilde K_1(t,\varsigma;\kappa)\ln\bigg(4\sin^2\frac{t-\varsigma}{2}\bigg),
\end{align*}
and the diagonal terms are given as
\begin{align*}
&\widetilde M_1(t,t;\kappa)=-\frac{1}{2\pi},  &&\widetilde M_2(t,t;\kappa)=\frac{\rm i}{2}-\frac{E_{\rm c}}{\pi}-\frac{1}{\pi}\ln\Big(\frac{\kappa}{2}G(t)\Big), \\
&\widetilde K_1(t,t;\kappa)=0,  
&&\widetilde K_2(t,t;\kappa)= \frac{1}{2\pi}\frac{\mathsf{n}(t)\cdot p''(t)}{|p'(t)|^2},
\end{align*}
with the Euler constant $E_{\rm c}=0.57721\cdots$.

For the kernel $\widetilde{H}$ of parametrized tangential derivative integral operator, analogously to \cite{DLL2019}, we split the kernel in the form  
$$
\widetilde{H}(t,\varsigma;\kappa)=\widetilde{H}_1(t,\varsigma;\kappa)\frac{1}{\sin(\varsigma-t)}+\widetilde{H}_2(t,\varsigma;\kappa)\ln\bigg(4\sin^2\frac{t-\varsigma}{2}\bigg)+\widetilde{H}_3(t,\varsigma;\kappa),
$$
where
\begin{align*}
\widetilde H_1(t,\varsigma;\kappa)&= \frac{1}{\pi}\mathsf{n}(t)^\perp\cdot\big[p(\varsigma)-p(t)\big]\frac{\sin(\varsigma-t)}{|p(t)-p(\varsigma)|^2}, \\
\widetilde H_2(t,\varsigma;\kappa)&= \frac{\kappa}{2\pi}\mathsf{n}(t)^\perp\cdot\big[p(t)-p(\varsigma)\big]\frac{J_1(\kappa|p(t)-p(\varsigma)|)}{|p(t)-p(\varsigma)|},\\
\widetilde H_3(t,\varsigma;\kappa)&=\widetilde H(t,\varsigma;\kappa)-\widetilde H_1(t,\varsigma;\kappa)\frac{1}{\sin(\varsigma-t)}-\widetilde H_2(t,\varsigma;\kappa)\ln\bigg(4\sin^2\frac{t-\varsigma}{2}\bigg)
\end{align*}
turn out to be analytic with the diagonal terms
$$
\widetilde H_1(t,t;\kappa)=\frac{1}{\pi}, \quad \widetilde H_2(t,t;\kappa)=0, \quad \widetilde H_3(t,t;\kappa)=0.
$$

Let $\varsigma_j^{(n)}:=\pi j/n$, $j=0,\cdots,2n-1$ be an equidistant set of quadrature nodes. For the integral of weakly singular part, by making use of quadrature rule in our previous work \cite[eqn. (4.6)]{DLL2019}, we employ the following quadrature rules
\begin{align}
&\int_{0}^{2\pi}\ln\bigg(4\sin^2\frac{t-\varsigma}{2}\bigg)Q(t,\varsigma)f(\varsigma)\mathrm{d}\varsigma\approx\sum_{j=0}^{2n-1}R_j^{(n)}(t)Q(t,\varsigma_j^{(n)})f(\varsigma_j^{(n)}), \label{quadrature1} \\
&\int_{0}^{2\pi}\frac{1}{\sin(\varsigma-t)}Q(t,\varsigma)f(\varsigma)\mathrm{d}\varsigma\approx\sum_{j=0}^{2n-1}T_j^{(n)}(t)Q(t,\varsigma_j^{(n)})f(\varsigma_j^{(n)}), \label{quadrature2}
\end{align}
where the function $Q$ is required to be continuous, and the quadrature weights are given by
\begin{align*}
&R_j^{(n)}(t)=-\frac{2\pi}{n}\sum_{m=1}^{n-1}\frac{1}{m}\cos\Big[m(t-\varsigma_j^{(n)})\Big] -\frac{\pi}{n^2}\cos\Big[n(t-\varsigma_j^{(n)})\Big] 
\\
&T_j^{(n)}(t)=
\begin{cases} 
\displaystyle 
-\frac{2\pi}{n}\sum_{m=0}^{(n-3)/2}\sin\Big[ (2m+1)(t-\varsigma_j^{(n)})\Big] -\frac{\pi}{n}\sin\Big[n(t-\varsigma_j^{(n)})\Big], \quad &n=1,3,5,\cdots, \\ \displaystyle
-\frac{2\pi}{n}\sum_{m=0}^{n/2-1}\sin\Big[ (2m+1)(t-\varsigma_j^{(n)})\Big], \quad &n=2,4,6,\cdots.
\end{cases}
\end{align*}
We also refer to \cite{Kress-book2014} for details of \eqref{quadrature1}. On the other hand, with the help of trapezoidal rule 
\begin{align} \label{traperule}
\int_{0}^{2\pi}f(\varsigma)\mathrm{d}\varsigma\approx\frac{\pi}{n}\sum_{j=0}^{2n-1}f(\varsigma_j^{(n)})
\end{align}
and Lagrange bases 
$$
\mathcal{L}_m(\varsigma)=\frac{1}{2n}\left\{1+2\sum_{k=1}^{n-1}\cos k(\varsigma-\varsigma_m^{(n)})+\cos n(\varsigma-\varsigma_m^{(n)})\right\}
$$
for the trigonometric interpolation, we derive the following quadrature rules for the integration with derivative involved 
\begin{align}
&\int_{0}^{2\pi}Q(t,\varsigma)f'(\varsigma)\mathrm{d}\varsigma\approx\frac{\pi}{n}\sum_{j=0}^{2n-1}\sum_{m=0}^{2n-1}d_{m-j}^{(n)}Q(t,\varsigma_m^{(n)})f(\varsigma_j^{(n)}),\\
&\int_{0}^{2\pi}\ln\bigg(4\sin^2\frac{t-\varsigma}{2}\bigg)Q(t,\varsigma)f'(\varsigma)\mathrm{d}\varsigma\approx\sum_{j=0}^{2n-1}\sum_{m=0}^{2n-1}d_{m-j}^{(n)}R_m^{(n)}(t)Q(t,\varsigma_m^{(n)})f(\varsigma_j^{(n)}), \label{quadrature1_der} \\
&\int_{0}^{2\pi}\frac{1}{\sin(\varsigma-t)}Q(t,\varsigma)f'(\varsigma)\mathrm{d}\varsigma\approx\sum_{j=0}^{2n-1}\sum_{m=0}^{2n-1}d_{m-j}^{(n)}T_m^{(n)}(t)Q(t,\varsigma_m^{(n)})f(\varsigma_j^{(n)}),  \label{quadrature2_der}
\end{align}
where we have set $d_{m-j}^{(n)}=\mathcal{L}'_j(\varsigma_m^{(n)})$, and the quadrature weights can be given by
\begin{align*}
&d_{j}^{(n)}=
\begin{cases} \displaystyle
\frac{(-1)^j}{2}\cot\frac{j\pi}{2n}, \quad &j=\pm1,\cdots,\pm2n-1,\\  \displaystyle
0,\quad & j=0.
\end{cases}
\end{align*}

Now, in view of the quadrature rules \eqref{quadrature1}--\eqref{quadrature2_der}, we employ following quadrature operators
\begin{align*}
S_\kappa(Q,\vartheta)(t)&=\sum_{j=0}^{2n-1}\left(R_{j}^{(n)}(t)\widetilde M_1(t,\varsigma_j^{(n)};\kappa)+ \frac{\pi}{n}\widetilde M_2(t,\varsigma_j^{(n)};\kappa)\right)Q(t,\varsigma_j^{(n)})\vartheta(\varsigma_j^{(n)}) \\
K_\kappa(Q,\vartheta)(t)&=\sum_{j=0}^{2n-1}\left(R_{j}^{(n)}(t)\widetilde K_1(t,\varsigma_j^{(n)};\kappa)+ \frac{\pi}{n}\widetilde K_2(t,\varsigma_j^{(n)};\kappa)\right)Q(t,\varsigma_j^{(n)})\vartheta(\varsigma_j^{(n)})/G(t)\\
\mathcal{K}_\kappa(Q,\vartheta')(t)&=\sum_{j=0}^{2n-1}\sum_{m=0}^{2n-1}\left(R_{m}^{(n)}(t)\widetilde K_1(t,\varsigma_m^{(n)};\kappa)+ \frac{\pi}{n}\widetilde K_2(t,\varsigma_m^{(n)};\kappa)\right)d_{m-j}^{(n)}Q(t,\varsigma_m^{(n)})\vartheta(\varsigma_j^{(n)})/G(t)
\end{align*}
\begin{align*}
 &H_\kappa(Q,\vartheta)(t)\\
=&\sum_{j=0}^{2n-1}\left(T_{j}^{(n)}(t)\widetilde H_1(t,\varsigma_j^{(n)};\kappa) + R_{j}^{(n)}(t)\widetilde H_2(t,\varsigma_j^{(n)};\kappa)+ \frac{\pi}{n}\widetilde H_3(t,\varsigma_j^{(n)};\kappa)\right)Q(t,\varsigma_j^{(n)})\vartheta(\varsigma_j^{(n)})/G(t)\\
 &\mathcal{H}_\kappa(Q,\vartheta')(t)\\
=&\sum_{j=0}^{2n-1}\sum_{m=0}^{2n-1}\left(T_{m}^{(n)}(t)
\widetilde H_1(t,\varsigma_m^{(n)};\kappa) + R_{m}^{(n)}(t)\widetilde H_2(t,\varsigma_m^{(n)};\kappa)+ \frac{\pi}{n}\widetilde H_3(t,\varsigma_m^{(n)};\kappa)\right)d_{m-j}^{(n)}Q(t,\varsigma_m^{(n)})\vartheta(\varsigma_j^{(n)})/G(t)
\end{align*}
as the approximation of integral operators $S_\kappa[Q\vartheta;p]$, $K_\kappa[Q\vartheta;p]$, $K_\kappa[Q\vartheta';p]$, $H_\kappa[Q\vartheta;p]$ and $H_\kappa[Q\vartheta';p]$.

To obtain a Nystr\"{o}m-type of discretization, we express the following combination
$$
\tilde{\varphi}_l^{(n)}(\varsigma)=\sum_{j=0}^{2n-1}\varUpsilon^{(l)}_j\mathcal{L}_j(\varsigma)
$$
with unknowns $\varUpsilon^{(l)}_j:=\tilde{\varphi}_l(\varsigma_j^{(n)})$ as finite dimensional approximation of the densities $\tilde{\varphi}_l$, $l=1,2,3$. Then, the derivative $\tilde{\varphi}'_l$ can be approximate by
$$
\tilde{\varphi}_l^{'(n)}(\varsigma)=\sum_{j=0}^{2n-1}\varUpsilon^{(l)}_j\mathcal{L}'_j(\varsigma).
$$
Hence, the full discretization of \eqref{pboundaryIE} can be deduced as the form
\begin{align}
\begin{cases}
\displaystyle
E^X_{1,i}\varUpsilon^{(1)}_i+\sum_{j=0}^{2n-1}X_{i,j}^{(1)}\varUpsilon^{(1)}_j+E^X_{2,i}\varUpsilon^{(2)}_i+\sum_{j=0}^{2n-1}\left(\mu d_{i-j}^{(n)} +X_{i,j}^{(2)}\right)\varUpsilon^{(2)}_j+\sum_{j=0}^{2n-1}X_{i,j}^{(3)}\varUpsilon^{(3)}_j=w_{1,i}^{(n)} \\
\displaystyle
E^Y_{1,i}\varUpsilon^{(1)}_i+\sum_{j=0}^{2n-1}\left(d_{i-j}^{(n)} +Y_{i,j}^{(1)}\right) \varUpsilon^{(1)}_j +E^Y_{2,i}\varUpsilon^{(2)}_i +\sum_{j=0}^{2n-1}Y_{i,j}^{(2)}\varUpsilon^{(2)}_j=w_{2,i}^{(n)} \\
\displaystyle
E^Z_{1,i}\varUpsilon^{(1)}_i+\sum_{j=0}^{2n-1}Z_{i,j}^{(1)}\varUpsilon^{(1)}_j+\sum_{j=0}^{2n-1}Z_{i,j}^{(2)}\varUpsilon^{(2)}_j+E^Z_{3,i}\varUpsilon^{(3)}_i+\sum_{j=0}^{2n-1}Z_{i,j}^{(3)}\varUpsilon^{(3)}_j=w_{3,i}^{(n)} \\
\end{cases}
\end{align}
where
\begin{align*}
&E^X_{1,i}=\mu\tilde{\nu}(\varsigma_i^{(n)})\cdot{\mathsf{n}^\perp}'(\varsigma_i^{(n)})/G(\varsigma_i^{(n)}),
&&E^X_{2,i}=\mu\tilde{\nu}(\varsigma_i^{(n)})\cdot\mathsf{n}'(\varsigma_i^{(n)})/G(\bar{\varsigma}_i^{(\bar{n})}),\\
&E^Y_{1,i}=\tilde{\tau}(\varsigma_i^{(n)})\cdot{\mathsf{n}^\perp}'(\varsigma_i^{(n)})/G(\varsigma_i^{(n)}),
&&E^Y_{2,i}=\tilde{\tau}(\varsigma_i^{(n)})\cdot\mathsf{n}'(\varsigma_i^{(n)})/G(\bar{\varsigma}_i^{(\bar{n})}),\\
&E^Z_{1,i}=G(\varsigma_i^{(n)}),
&&E^Z_{3,i}=G(\varsigma_i^{(n)})/(\omega^2\rho_{\rm a}),
\end{align*}
\begin{align*}
X^{(1)}\varUpsilon^{(1)}=&-\mu\kappa^2_{\rm p}S_{\kappa_{\rm p}}(Q^X_{11},\tilde{\varphi}_1)(\varsigma_i^{(n)})+\mu \mathcal{K}_{\kappa_{\rm p}}(Q^X_{12},\tilde{\varphi}'_1)(\varsigma_i^{(n)})+\mu K_{\kappa_{\rm p}}(Q^X_{13},\tilde{\varphi}_1)(\varsigma_i^{(n)}) \\
& -\mu \mathcal{H}_{\kappa_{\rm p}}(Q^X_{14},\tilde{\varphi}'_1)(\varsigma_i^{(n)})-\mu H_{\kappa_{\rm p}}(Q^X_{15},\tilde{\varphi}_1)(\varsigma_i^{(n)})-(\lambda+\mu)\kappa^2_{\rm p}S_{\kappa_{\rm p}}(Q^X_{16},\tilde{\varphi}_1)(\varsigma_i^{(n)})\\
X^{(2)}\varUpsilon^{(2)}=&\mu\kappa^2_{\rm s}S_{\kappa_{\rm s}}(Q^X_{21},\tilde{\varphi}_2)(\varsigma_i^{(n)})+\mu \mathcal{K}_{\kappa_{\rm s}}(Q^X_{22},\tilde{\varphi}'_2)(\varsigma_i^{(n)})+\mu K_{\kappa_{\rm s}}(Q^X_{23},\tilde{\varphi}_2)(\varsigma_i^{(n)}) \\
&+\mu \mathcal{H}_{\kappa_{\rm s}}(Q^X_{24},\tilde{\varphi}'_2)(\varsigma_i^{(n)})+ \mu H_{\kappa_{\rm s}}(Q^X_{25},\tilde{\varphi}_2)(\varsigma_i^{(n)})\\
X^{(3)}\varUpsilon^{(3)}=&S_{\kappa_{\rm a}}(Q^X_{31},\tilde{\varphi}_3)(\varsigma_i^{(n)})\\
Y^{(1)}\varUpsilon^{(1)}=&-\kappa^2_{\rm p}S_{\kappa_{\rm p}}(Q^Y_{11},\tilde{\varphi}_1)(\varsigma_i^{(n)})+\mathcal{K}_{\kappa_{\rm p}}(Q^Y_{12},\tilde{\varphi}'_1)(\varsigma_i^{(n)})+ K_{\kappa_{\rm p}}(Q^Y_{13},\tilde{\varphi}_1)(\varsigma_i^{(n)}) \\
&-\mathcal{H}_{\kappa_{\rm p}}(Q^Y_{14},\tilde{\varphi}'_1)(\varsigma_i^{(n)})-H_{\kappa_{\rm p}}(Q^Y_{15},\tilde{\varphi}_1)(\varsigma_i^{(n)})\\
Y^{(2)}\varUpsilon^{(2)}=&\kappa^2_{\rm s}S_{\kappa_{\rm s}}(Q^Y_{21},\tilde{\varphi}_2)(\varsigma_i^{(n)})+\mathcal{K}_{\kappa_{\rm s}}(Q^Y_{22},\tilde{\varphi}'_2)(\varsigma_i^{(n)})+ K_{\kappa_{\rm s}}(Q^Y_{23},\tilde{\varphi}_2)(\varsigma_i^{(n)}) \\
&+\mathcal{H}_{\kappa_{\rm s}}(Q^Y_{24},\tilde{\varphi}'_2)(\varsigma_i^{(n)})+  H_{\kappa_{\rm s}}(Q^Y_{25},\tilde{\varphi}_2)(\varsigma_i^{(n)})\\
Z^{(1)}\varUpsilon^{(1)}=&K_{\kappa_{\rm p}}(Q^Z_{11},\tilde{\varphi}_1)(\varsigma_i^{(n)}), \quad
Z^{(2)}\varUpsilon^{(2)}=H_{\kappa_{\rm s}}(Q^Z_{21},\tilde{\varphi}_2)(\varsigma_i^{(n)}),
\\
Z^{(3)}\varUpsilon^{(3)}=&-K_{\kappa_{\rm a}}(Q^Z_{31},\tilde{\varphi}_3)(\varsigma_i^{(n)})/(\omega^2\rho_{\rm a})
\end{align*}
with
\begin{align*}
&Q^X_{11}(t,\varsigma)=\tilde{\nu}^\top(t)\langle\mathsf{n}(\varsigma),\mathsf{n}(\varsigma)\rangle\tilde{\nu}(t),
&&Q^X_{12}(t,\varsigma)=\tilde{\nu}^\top(t)\mathsf{n}^\perp(\varsigma), \quad &&Q^X_{13}(t,\varsigma)=\tilde{\nu}^\top(t){\mathsf{n}^\perp}'(\varsigma), \\
&Q^X_{14}(t,\varsigma)=\tilde{\nu}^\top(t)\mathsf{n}(\varsigma),  &&Q^X_{15}(t,\varsigma)=\tilde{\nu}^\top(t)\mathsf{n}'(\varsigma), \quad
&&Q^X_{16}(t,\varsigma)=G^2(\varsigma), \\
&Q^X_{21}(t,\varsigma)=\tilde{\nu}^\top(t)\langle\mathsf{n}^\perp(\varsigma),\mathsf{n}(\varsigma)\rangle\tilde{\nu}(t), 
&&Q^X_{22}(t,\varsigma)=Q^X_{14}(t,\varsigma), 
&&Q^X_{23}(t,\varsigma)=Q^X_{15}(t,\varsigma), \\
&Q^X_{24}(t,\varsigma)=Q^X_{12}(t,\varsigma), 
&&Q^X_{25}(t,\varsigma)=Q^X_{13}(t,\varsigma), 
&&Q^X_{31}(t,\varsigma)=G^2(\varsigma),\\
&Q^Y_{11}(t,\varsigma)=\tilde{\tau}^\top(t)\langle\mathsf{n}(\varsigma),\mathsf{n}(\varsigma)\rangle\tilde{\nu}(t), 
&&Q^Y_{12}(t,\varsigma)=\tilde{\tau}^\top(t)\mathsf{n}^\perp(\varsigma), 
&&Q^Y_{13}(t,\varsigma)=\tilde{\tau}^\top(t){\mathsf{n}^\perp}'(\varsigma), \\
&Q^Y_{14}(t,\varsigma)=\tilde{\tau}^\top(t)\mathsf{n}(\varsigma), 
&&Q^Y_{15}(t,\varsigma)=\tilde{\tau}^\top(t)\mathsf{n}'(\varsigma), \\
&Q^Y_{21}(t,\varsigma)=\tilde{\tau}^\top(t)\langle\mathsf{n}^\perp(\varsigma),\mathsf{n}(\varsigma)\rangle\tilde{\nu}(t), 
&&Q^Y_{22}(t,\varsigma)=Q^Y_{14}(t,\varsigma), 
&&Q^Y_{23}(t,\varsigma)=Q^Y_{15}(t,\varsigma), \\
&Q^Y_{24}(t,\varsigma)=Q^Y_{12}(t,\varsigma), 
&&Q^Y_{25}(t,\varsigma)=Q^Y_{13}(t,\varsigma), \\
&Q^Z_{11}(t,\varsigma)=Q^Z_{21}(t,\varsigma)=Q^Z_{31}(t,\varsigma)=G^2(\varsigma),
\end{align*}
and $w_{l,i}^{(n)}=w_l(\varsigma_i^{(n)})$ for $i=0,\cdots,2n-1$, $l=1, 2, 3$.

\section{Reconstruction methods}

In this section, we introduce the iterative methods and the algorithms for the phased and phaseless IAEIP.

\subsection{Iterative method for the phased IAEIP}

We assume that the field equations are \eqref{boundaryIE},
and the data equation is given by
\begin{align*} 
S^\infty_{\kappa_{\rm a}}[g_3]=u_\infty.
\end{align*}
Thus, the field equations and data equation can be reformulated as the parametrized integral equations \eqref{paraboundaryIE} and
\begin{align}
S^\infty_{\kappa_{\rm a}}[\varphi_3G;p]=w_\infty, \label{paradata eqn}
\end{align}
where $w_\infty=u_\infty\circ p$.

In the reconstruction process, when an approximation of the boundary $\Gamma_D$
is available, the field equations \eqref{paraboundaryIE}
are solved for the densities $\varphi_l$, $l=1,2,3$. Once the approximated densities
$\varphi_l$ are computed, the update of the boundary $\Gamma_D$ can be
obtained by solving the linearized data equation \eqref{paradata eqn} with
respect to $\Gamma_D$.

\subsubsection{Iterative scheme}

The linearization of \eqref{paradata eqn} with respect to a given $p$ requires the Fr\'{e}chet derivative of the parameterized integral operator $S_\kappa^\infty$, which can be easily computed and is given by 
\begin{align} 
\bigg({S_\kappa^\infty}'[\vartheta;p]q\bigg)(t)=&-\mathrm{i}\kappa\gamma_\kappa
\int_0^{2\pi}e^{-\mathrm{i}\kappa\hat{x}(t)\cdot p(\varsigma)}\hat{x}(t)\cdot q(\varsigma)\vartheta(\varsigma)\mathrm{d}\varsigma \nonumber \\
=&-\mathrm{i}\kappa\gamma_\kappa\int_0^{2\pi}\exp\bigg(-\mathrm{i}\kappa\big(c_1\cos t+c_2\sin t+r(\varsigma)\cos(t-\varsigma)\big)\bigg) \nonumber \\
&\qquad\quad\cdot\bigg(\Delta c_1\cos t+\Delta c_2\sin t +\Delta r(\varsigma) \cos(t-\varsigma)\bigg)\vartheta(\varsigma)\,\mathrm{d}\varsigma, \label{FreSinfty}
\end{align}
where $q(\varsigma)=(\Delta c_1, \Delta c_2)+\Delta r(\varsigma) (\cos\varsigma,\sin\varsigma)$ denoted as the update of the boundary $\Gamma_D$. Then, the linearization of \eqref{paradata eqn} leads to
\begin{equation}\label{linear paradata eqn}
{S_{\kappa_{\rm a}}^\infty}'[\varphi_3G;p]q=w,
\end{equation}
where
\begin{align*}
w:=w_\infty-S^\infty_{\kappa_{\rm a}}[\varphi_3G;p].
\end{align*}
As usual for iterative algorithms, the stopping criteria is necessary to justify the convergence numerically. With regard to our iterative procedure, the relative error estimator is chosen as follows
\begin{equation}
E_k:=\frac{\left\|w_\infty-S^\infty_{\kappa_{\rm a}}[\varphi_3G;p^{(k)}]\right\|_{L^2}} {\big\|w_\infty\big\|_{L^2}}\leq\epsilon \label{relativeerror}
\end{equation}
for some sufficiently small parameter $\epsilon>0$ depending on the noise level, where $p^{(k)}$ is the $k$th approximation of the boundary $\Gamma_D$.

We are now in a position to present the iterative algorithm for the inverse
obstacle scattering problem with phased far-field data as {\bf Algorithm I}.
\begin{table}[ht]
	\begin{tabular}{cp{.8\textwidth}}
		\toprule
		\multicolumn{2}{l}{{\bf Algorithm I:}\quad Iterative algorithm for the phased IAEIP} \\
		\midrule
		Step 1 & Send an incident plane wave ${u}^{\rm
			inc}$ with a fixed wave number $\kappa_{\rm a}$ and a fixed incident direction
		$d\in\Omega$, and then collect the corresponding far-field data
		$u_\infty$ for the scatterer $D$; \\
		Step 2 & Select an initial star-like curve $\Gamma^{(0)}$
		for the boundary $\Gamma_D$ and the error tolerance $\epsilon$. Set $k=0$; \\
		Step 3 & For the curve $\Gamma^{(k)}$, compute the densities
		$\varphi_l$, $l=1,2,3$ from \eqref{paraboundaryIE};
		\\
		Step 4 & Solve \eqref{linear paradata eqn} to obtain the
		updated approximation $\Gamma^{(k+1)}:=\Gamma^{(k)}+q$ and evaluate the error
		$E_{k+1}$ defined in \eqref{relativeerror}; \\
		Step 5 & If $E_{k+1}\geq\epsilon$, then set $k=k+1$ and go
		to Step 3. Otherwise, the current approximation $\Gamma^{(k+1)}$ is taken to be
		the final reconstruction of $\Gamma_D$. \\
		\bottomrule
	\end{tabular}
\end{table}

\subsubsection{Discretization} 

We use the Nystr\"{o}m-type method which is described in Section 5 for the full discretizations of \eqref{paraboundaryIE}. Now we discuss the discretization of the linearized equation \eqref{linear paradata	eqn} and obtain the update by using the least squares with Tikhonov regularization
\cite{Kress-IP2003}. As for a finite dimensional space to approximate the
radial function $r$ and its update $\Delta r$, we choose the space of
trigonometric polynomials of the form
\begin{equation*}
\Delta r(\varsigma)=\sum_{m=0}^M\alpha_m\cos{m\varsigma} +\sum_{m=1}^M\beta_m\sin{m\varsigma}. 
\end{equation*}
where the integer $M>1$ is the truncation number. For simplicity, we reformulate the equation \eqref{linear paradata eqn} by introducing the following definitions
\begin{align*}
&L_1(t,\varsigma;\varphi):=-\mathrm{i}\kappa_{\rm a}\gamma_{\kappa_{\rm a}} \exp\left\{-\mathrm{i}\kappa_{\rm a}\Big(c_1\cos t+c_2\sin t+r(\varsigma)\cos(t-\varsigma)\Big)\right\}\cos t~\varphi(\varsigma),\\
&L_2(t,\varsigma;\varphi):=-\mathrm{i}\kappa_{\rm a}\gamma_{\kappa_{\rm a}} \exp\left\{-\mathrm{i} \kappa_{\rm a}\Big(c_1\cos t+c_2\sin t +r(\varsigma)\cos(t-\varsigma)\Big)\right\}\sin t~ \varphi(\varsigma),\\
&L_{3,m}(t,\varsigma;\varphi):=-\mathrm{i}\kappa_{\rm a}\gamma_{\kappa_{\rm a}} \exp\left\{-\mathrm{i}\kappa_{\rm a}\Big(c_1\cos t+c_2\sin t+r(\varsigma) \cos(t-\varsigma)\Big)\right\} \cos(t-\varsigma)\cos m\varsigma ~\varphi(\varsigma),\\
&L_{4,m}(t,\varsigma;\varphi):=-\mathrm{i}\kappa_{\rm a}\gamma_{\kappa_{\rm a}} \exp\left\{-\mathrm{i}\kappa_{\rm a}\Big(c_1\cos t+c_2\sin t+ r(\varsigma)\cos(t-\varsigma)\Big)\right\} \cos(t-\varsigma)\sin m\varsigma ~\varphi(\varsigma).
\end{align*}
Then, by combining \eqref{FreSinfty} and \eqref{linear paradata eqn} together and using trapezoidal rule \eqref{traperule}, we get the discretized linear system
\begin{align}\label{discrelinear}
\sum_{l=1}^2 B_{l}^c(\varsigma_i^{(\tilde{n})})\Delta c_l+\sum_{m=0}^M\alpha_mB_{1,m}^r(\varsigma_i^{(\tilde{n})})+\sum_{m=1}^M\beta_mB_{2,m}^r(\varsigma_i^{(\tilde{n})})= w(\varsigma_i^{(\tilde{n})})
\end{align}
to determine the real coefficients $\Delta c_1$, $\Delta c_2$, $\alpha_m$ and $\beta_m$, where $\varsigma_i^{(\tilde{n})}:=\pi i/\tilde{n},~i = 0,1,\cdots, 2\tilde{n}-1$ are the far-field observation points in $[0, 2\pi]$,
\begin{align*}
B_l^c(\varsigma_i^{(\tilde{n})})=\frac{\pi}{n}\sum_{j=0}^{2n-1}L_l(\varsigma_i^{(\tilde{n})},\varsigma_j^{(n)};\varphi_3G)
\end{align*}
for $l=1, 2$, and
$$
B_{1,m}^r(\varsigma_i^{(\tilde{n})})=\frac{\pi}{n}\sum_{j=0}^{2n-1}L_{3,m}(\varsigma_i^{(\tilde{n})},\varsigma_j^{(n)};\varphi_3G), \qquad 
B_{2,m}^r(\varsigma_i^{(\tilde{n})})=\frac{\pi}{n}\sum_{j=0}^{2n-1}L_{4,m}(\varsigma_i^{(\tilde{n})},\varsigma_j^{(n)};\varphi_3G).
$$

In general,  $2M+1\ll 2\tilde{n}$, and due to the ill-posedness, the overdetermined
system \eqref{discrelinear} is solved via the Tikhonov regularization. Hence the
linear system \eqref{discrelinear} is reformulated into minimizing the following
function

\begin{align}
\sum_{i=0}^{2\tilde{n}-1}\Bigg|\sum_{l=1}^2 B_{l}^c(\varsigma_i^{(\tilde{n})})\Delta c_l&+\sum_{m=0}^M\alpha_mB_{1,m}^r(\varsigma_i^{(\tilde{n})})+\sum_{m=1}^M\beta_mB_{2,m}^r(\varsigma_i^{(\tilde{n})})-w(\varsigma_i^{(\tilde{n})})\Bigg|^2 \nonumber \\ 
&+\lambda\bigg(|\Delta c_1|^2+|\Delta c_2|^2+2\pi\Big[\alpha_0^2+\frac{1}{2}\sum_{m=1}^M(1+m^2)^2(\alpha_m^2+\beta_m^2)\Big]\bigg) \label{RLHuygens3}
\end{align}
with $H^2$ penalty term, where $\lambda>0$ is a regularization parameter. It is easy to show that the minimizer of \eqref{RLHuygens3} is the solution of the system
\begin{align}\label{EqualRLHuygens3}
(\lambda\widetilde{I}+ \Re(\widetilde{B}^*\widetilde{B}))\xi=\Re(\widetilde{B}^*\widetilde{w}),
\end{align}
where
\begin{align*}
\noindent
\widetilde{B}&=\Big(B_1^c, B_2^c, B_{1,0}^r, \cdots, B_{1,M}^r,B_{2,1}^r,\cdots,B_{2,M}^r\Big)_{(2\tilde{n})\times(2M+3)},\\
\xi&=(\Delta c_1, \Delta c_2, \alpha_0,\cdots, \alpha_M,\beta_1,\cdots, \beta_M)^\top, \\
\widetilde{I}&=\mathrm{diag}\{1, 1, 2\pi, \pi(1+1^2)^2, \cdots, \pi(1+M^2)^2,
\pi(1+1^2)^2, \cdots, \pi(1+M^2)^2\}, \\
\widetilde{w}&=(w(\varsigma_0^{(\tilde{n})}),\cdots, w(\varsigma_{2\tilde{n}-1}^{(\tilde{n})}))^\top.
\end{align*}
Thus, we obtain the new approximation 
$$
p^{\rm new}(\hat{x})=(c+\Delta c)+\Big(r(\hat{x})+\Delta r(\hat{x})\Big)\hat{x}.
$$

\subsection{Iterative method for the phaseless IAEIP}

To incorporate the reference ball, we find the solution of
\eqref{HelmholtzDec} with $D$ replaced by $D\cup B$ in the form of single-layer
potentials with densities $g_{1,\sigma}$,  $g_{2,\sigma}$ and $g_{3,\sigma}$:
\begin{align}
&\phi(x)=\sum_\sigma\int_{\Gamma_\sigma}\Phi(x,y;\kappa_{\rm	p})g_{1,\sigma}(y)\mathrm{d}s(y), \label{singlelayerDB1}\\
&\psi(x)=\sum_\sigma\int_{\Gamma_\sigma}\Phi(x,y;\kappa_{\rm	s})g_{2,\sigma}(y)\mathrm{d}s(y),  \label{singlelayerDB2}\\
&u^s(x)=\sum_\sigma\int_{\Gamma_\sigma}\Phi(x,y;\kappa_{\rm	a})g_{3,\sigma}(y)\mathrm{d}s(y),  \label{singlelayerDB3}
\end{align}
for $x\in\mathbb{R}^2\setminus\Gamma_{D\cup B}$, where $\sigma=D, B$.

\begin{table} \label{partable}
	\caption{Parametrization of the exact boundary curves.}
	\begin{tabular}{lll}
		\toprule  
		Type           &Parametrization\\
		\midrule  
		Apple-shaped   & $p_D(t)=\displaystyle\frac{0.55(1+0.9\cos{t}+0.1\sin{2t})}{1+0.75\cos{t}}(\cos{t},\sin{t}), \quad t\in [0,2\pi]$ \\
		~\\
		Peanut-shaped  & 
		$p_D(t)=0.65\sqrt{0.25\cos^2{t}+\sin^2{t}}(\cos{t},\sin{t}), \quad
		t\in[0,2\pi]$\\   
		\bottomrule 
	\end{tabular}
\end{table}

\begin{figure}
	\centering 
	\subfigure[Reconstruction with $1\%$ noise, $\epsilon=0.2$ ]
	{\includegraphics[width=0.4\textwidth]{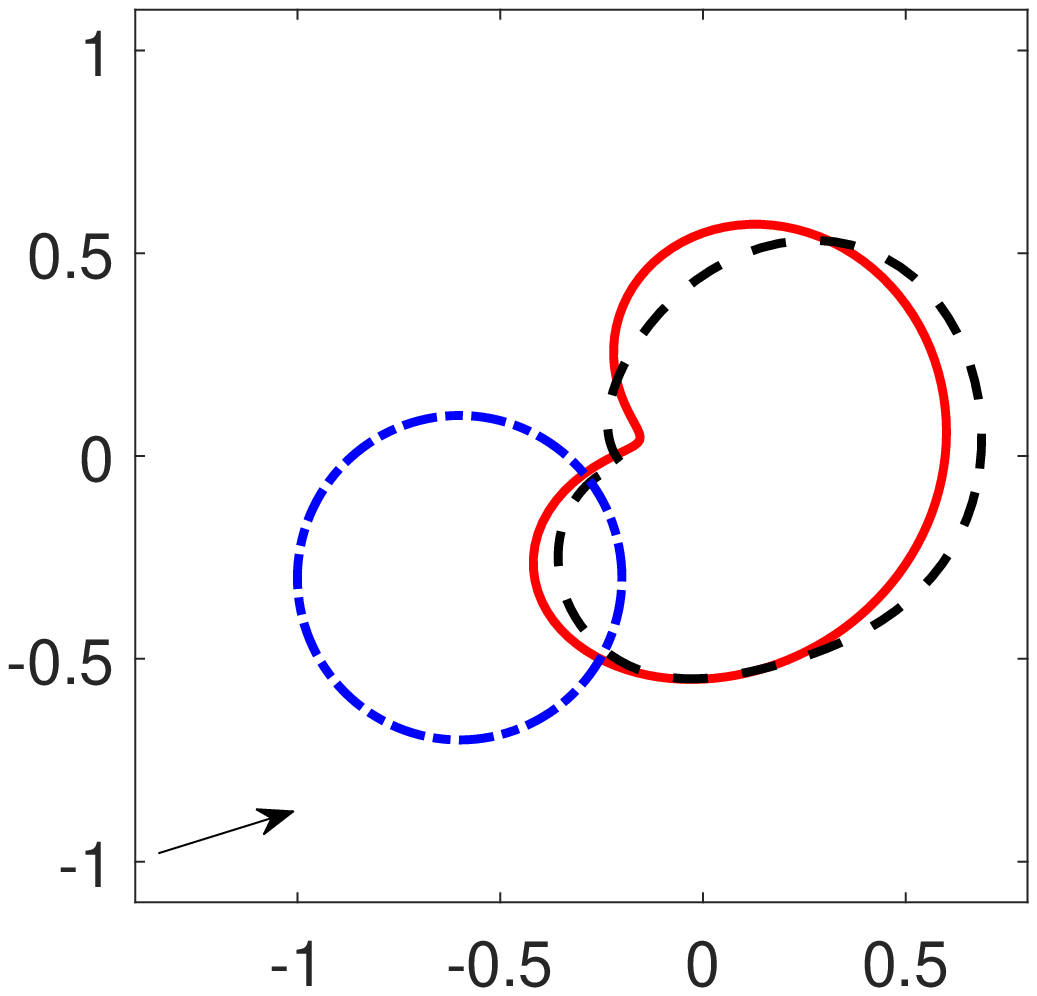}}
	\subfigure[Relative error with $1\%$ noise]
	{\includegraphics[width=0.4\textwidth]{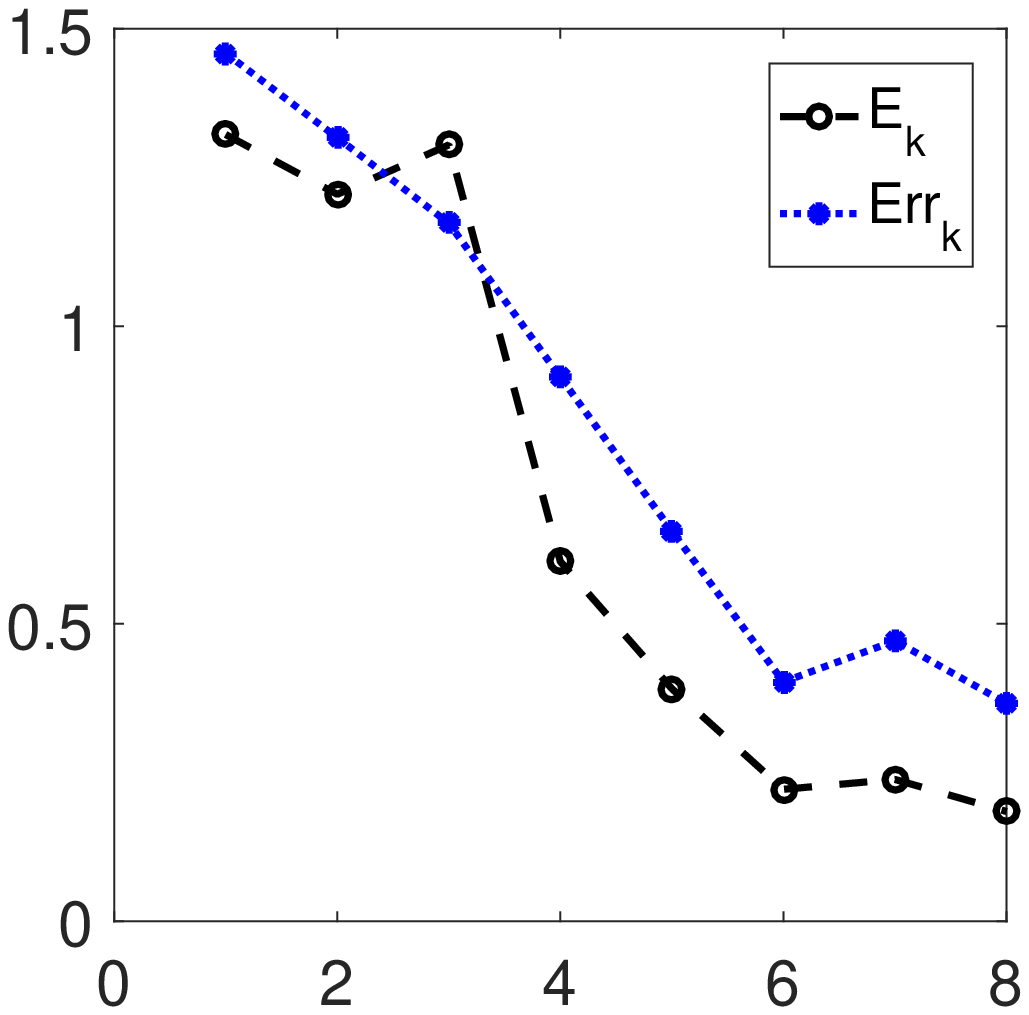}} 
	\subfigure[Reconstruction with $5\%$ noise, $\epsilon=0.2$]
	{\includegraphics[width=0.4\textwidth]{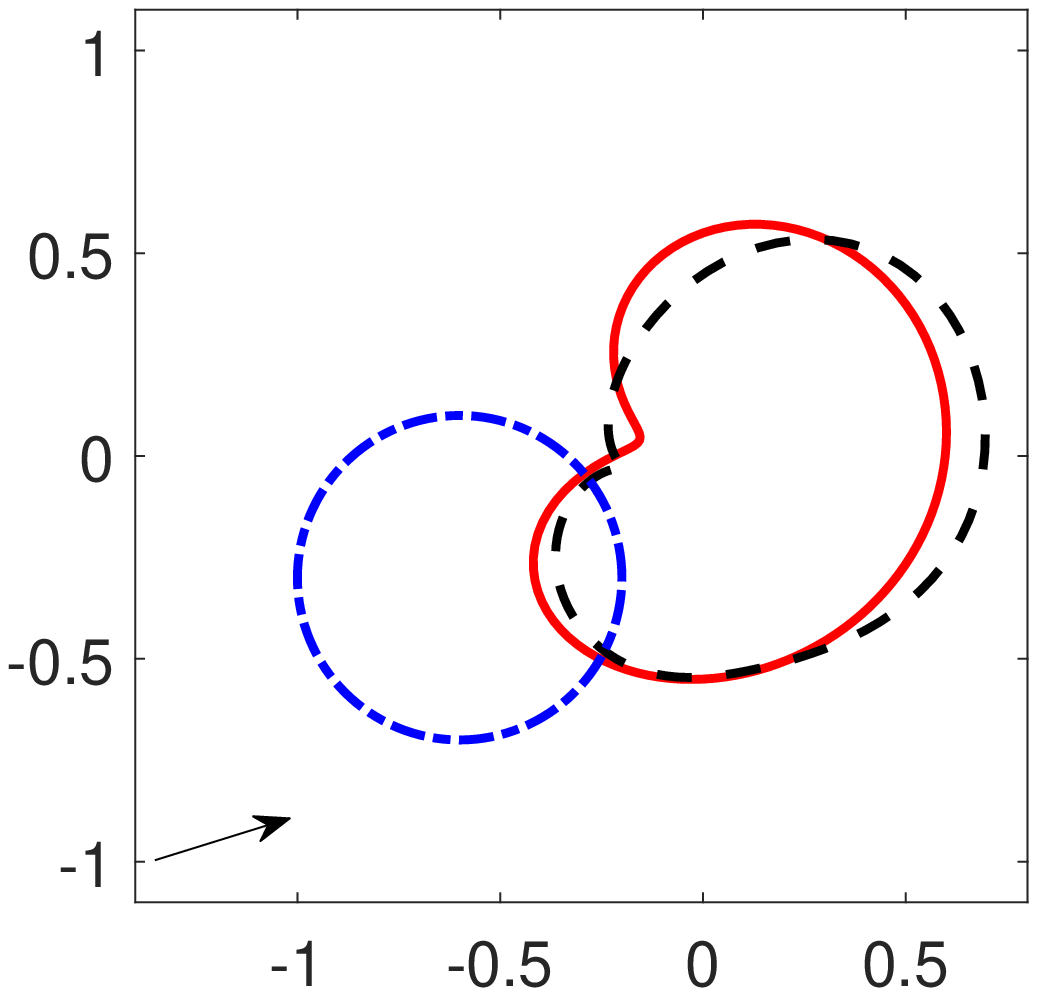}}
	\subfigure[Relative error with $5\%$ noise]
	{\includegraphics[width=0.4\textwidth]{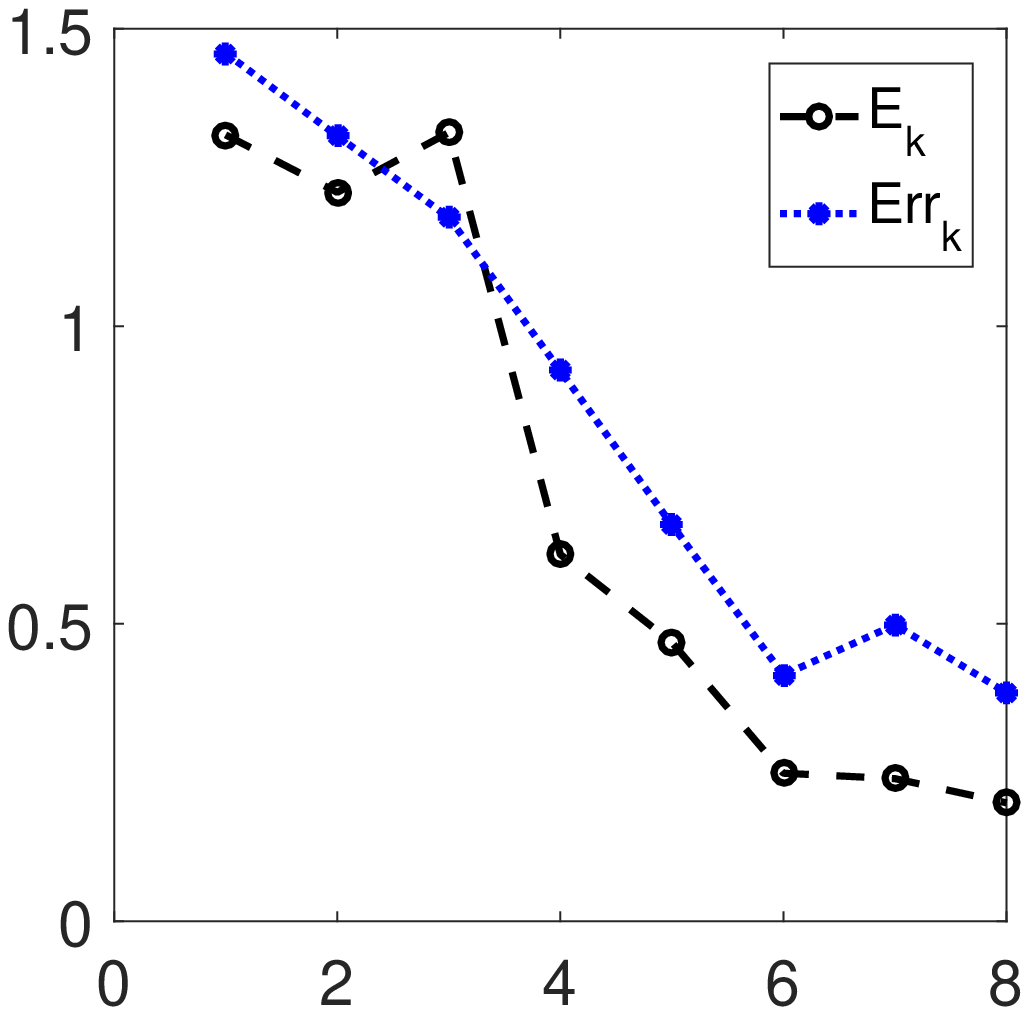}}
	\caption{Reconstructions of an apple-shaped obstacle with phased data at different
		levels of noise (see example 1). The initial guess is given by $(c_1^{(0)},c_2^{(0)})=(-0.6,
		-0.3), r^{(0)}=0.4$ and the incident angle $\theta=\pi/8$.}\label{IOSP-2}
\end{figure}

\begin{figure}
	\centering 
	\subfigure[Reconstruction with $1\%$ noise, $\epsilon=0.25$ ]
	{\includegraphics[width=0.4\textwidth]{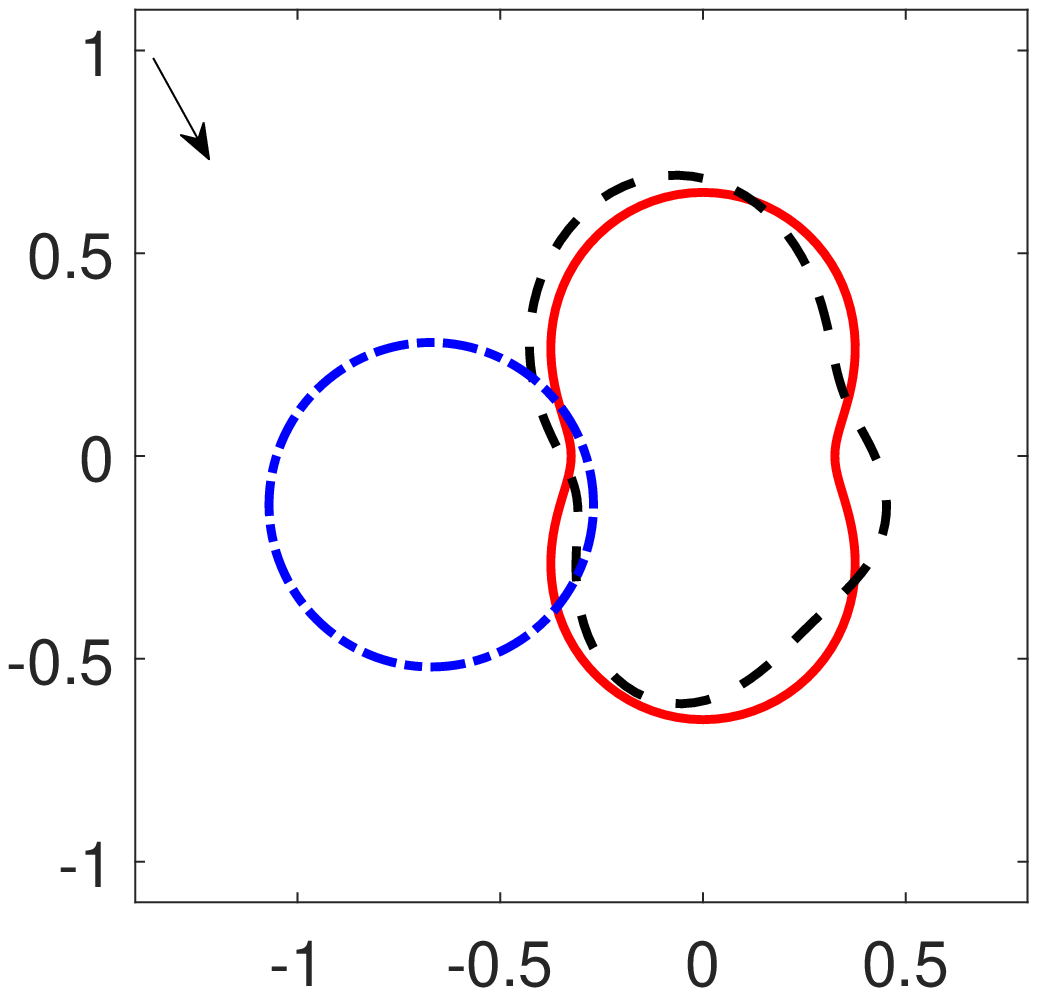}}
	\subfigure[Relative error with $1\%$ noise]
	{\includegraphics[width=0.4\textwidth]{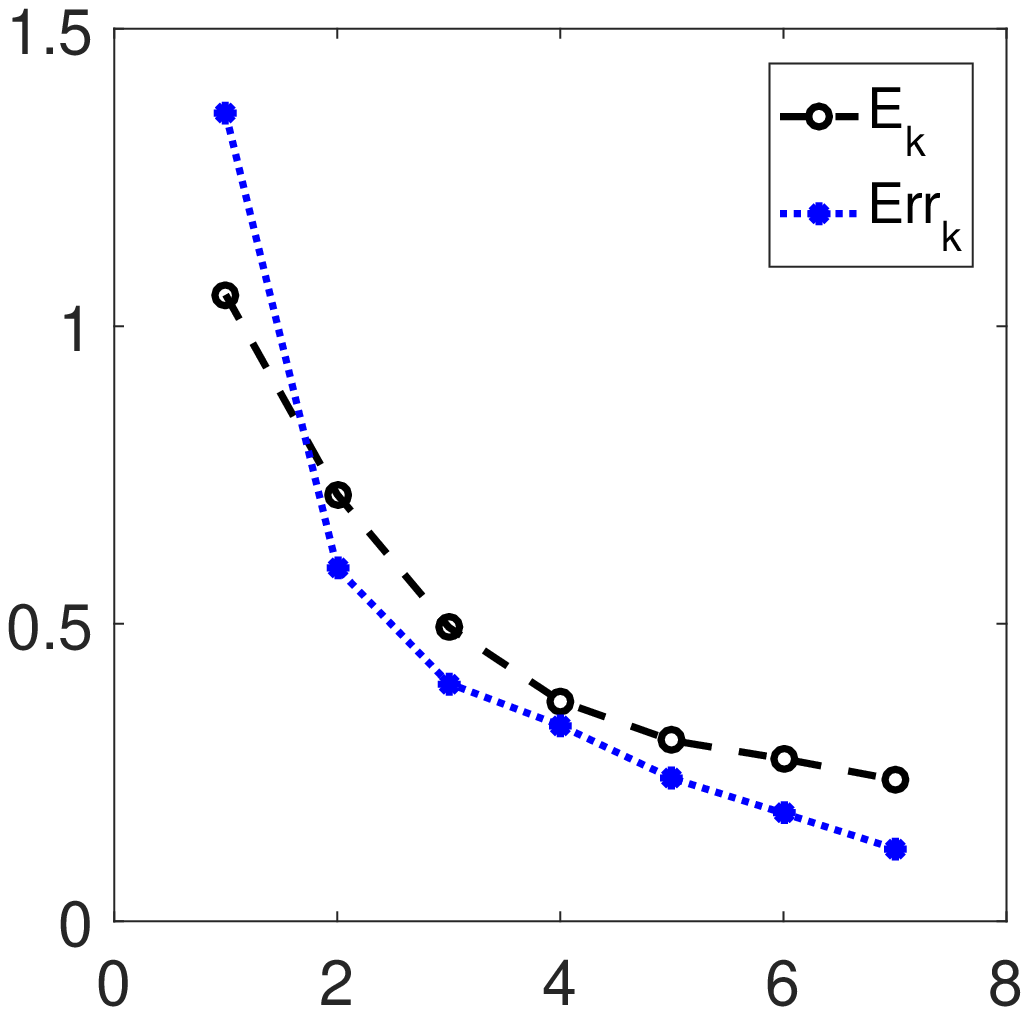}} 
	\subfigure[Reconstruction with $5\%$ noise, $\epsilon=0.25$]
	{\includegraphics[width=0.4\textwidth]{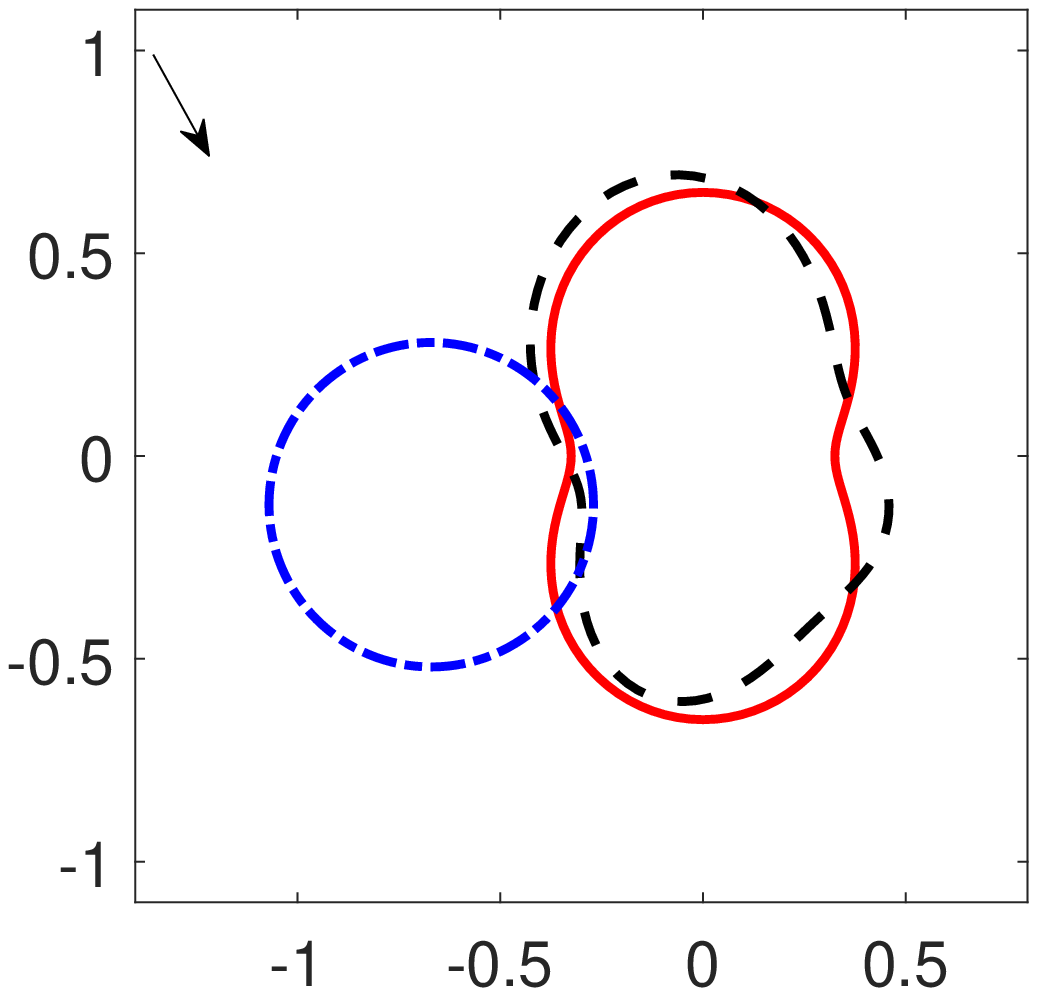}}
	\subfigure[Relative error with $5\%$ noise]
	{\includegraphics[width=0.4\textwidth]{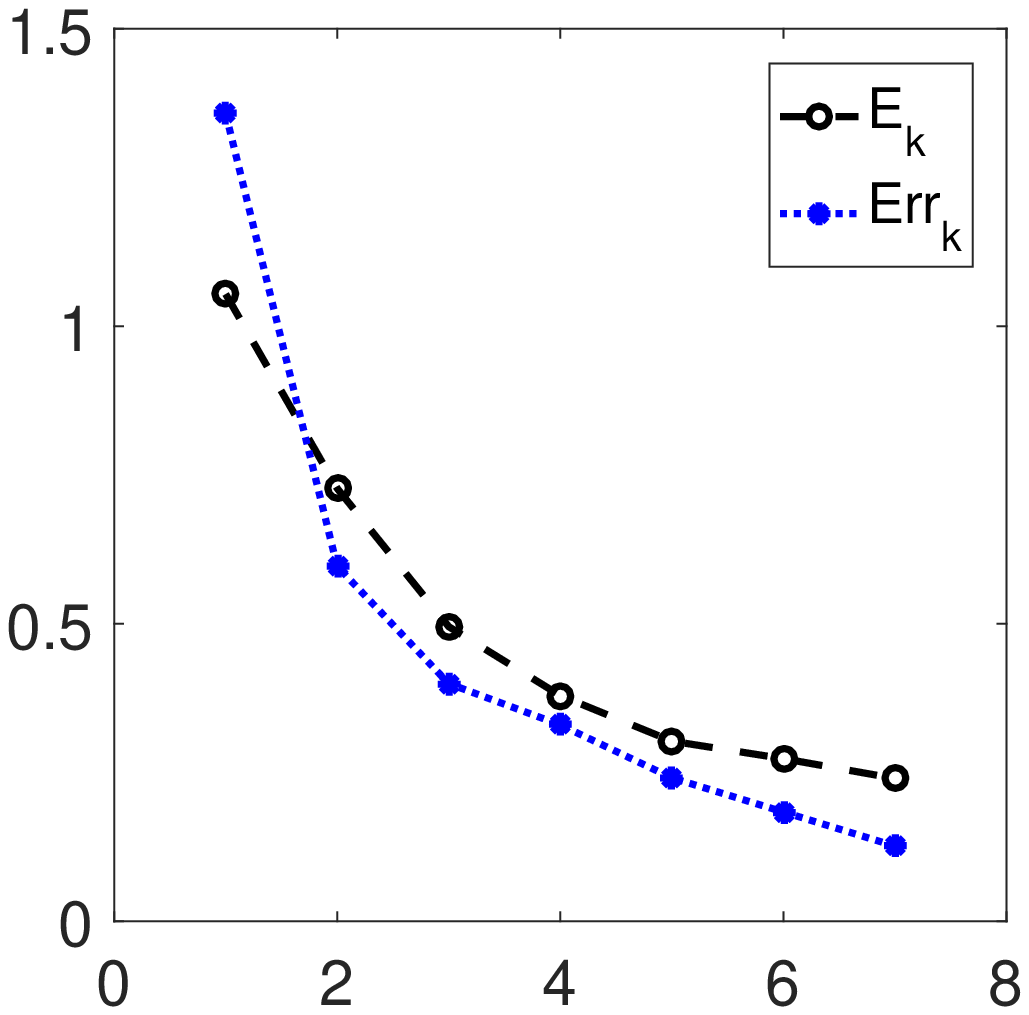}}
	\caption{Reconstructions of a peanut-shaped obstacle with phased data at different
		levels of noise (see example 1). The initial guess is given by $(c_1^{(0)},c_2^{(0)})=(-0.67,
		-0.12), r^{(0)}=0.4$ and the incident angle $\theta=13\pi/8$.}\label{IOSP-5}
\end{figure}

\begin{figure}
	\centering 
	\subfigure[]
	{\includegraphics[width=0.4\textwidth]{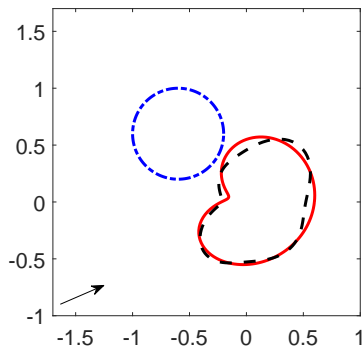}}
	\subfigure[]
	{\includegraphics[width=0.4\textwidth]{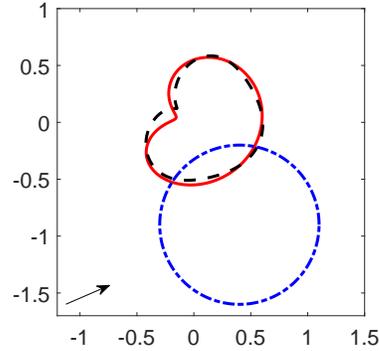}}
	\caption{Reconstructions of an apple-shaped obstacle with different
		initial guesses, where $1\%$ noise is added and the incident angle
		$\theta=\pi/6$. (a) $(c_1^{(0)},c_2^{(0)})=(-0.6, 0.6)$, $r^{(0)}=0.4$, $\epsilon=0.2$; (b) $(c_1^{(0)},c_2^{(0)})=(0.4, -0.9)$, $r^{(0)}=0.7$, $\epsilon=0.15$ (see example 1).}\label{IOSP-3}
\end{figure}

Furthermore, we introduce following integral operators, i.e.,
\begin{align*}
&S_\kappa^{\sigma,\varrho}[g](x)=2\int_{\Gamma_\sigma} \Phi(x,y;\kappa)g(y)\mathrm{d}s(y), \quad x\in\Gamma_\varrho, \\
&S_{\kappa,\sigma}^\infty [g](\hat{x})=\gamma_\kappa\int_{\Gamma_\sigma}e^{-\mathrm{i}\kappa \hat{x}\cdot y}g(y)\mathrm{d}s(y), \qquad\hat{x}\in\Omega,\\
&K_\kappa^{\sigma,\varrho} [g](x)=2\int_{\Gamma_\sigma}\frac{\partial\Phi(x,y;\kappa)} {\partial\nu(x)}g(y)\mathrm{d}s(y), \qquad x\in\Gamma_\varrho,\\
&H_\kappa^{\sigma,\varrho} [g](x)=2\int_{\Gamma_\sigma}\frac{\partial\Phi(x,y;\kappa)} {\partial\tau(x)}g(y)\mathrm{d}s(y), \qquad x\in\Gamma_\varrho,
\end{align*}
where $\varrho=D, B$. Then, letting $x\in\mathbb{R}^2\setminus\overline{D\cup B}$ tend to boundaries $\Gamma_D$ and $\Gamma_B$ respectively in \eqref{singlelayerDB1}--\eqref{singlelayerDB3}, and making use of the jump relation of the single-layer potentials and the boundary condition of \eqref{HelmholtzDec} for $\Gamma_{D\cup B}$, we can readily deduce the following field equations in the operator form on $\Gamma_D$:
\begin{align}
2f_1=&-\mu\kappa^2_{\rm p}\nu_D^\top S^{D,D}_{\kappa_{\rm p}}\big[\langle\nu_D,\nu_D\rangle g_{1,D}\big]\nu_D +\mu\nu_D^\top K^{D,D}_{\kappa_{\rm p}}\big[\tau_D\partial_{\tau_D} g_{1,D}+g_{1,D}\partial_{\tau_D}\tau_D\big] \nonumber\\
&-\mu\nu_D^\top H^{D,D}_{\kappa_{\rm p}}\big[\nu_D\partial_{\tau_D} g_{1,D}+g_{1,D}\partial_{\tau_D}\nu_D\big]+\mu\kappa^2_{\rm s}\nu_D^\top S^{D,D}_{\kappa_{\rm s}}\big[\langle\tau_D,\nu_D\rangle g_{2,D}\big]\nu_D
\nonumber\\
&+\mu\nu_D^\top K^{D,D}_{\kappa_{\rm s}}\big[\nu_D\partial_{\tau_D} g_{2,D}+g_{2,D}\partial_{\tau_D}\nu_D\big]+\mu\nu_D^\top H^{D,D}_{\kappa_{\rm s}}\big[\tau_D\partial_{\tau_D} g_{2,D}+g_{2,D}\partial_{\tau_D}\tau_D\big] \nonumber\\
&+\mu\nu_D\cdot\partial_{\nu_D}\nabla S^{B,D}_{\kappa_{\rm p}}[g_{1,B}]+\mu\nu_D\cdot\partial_{\nu_D}\boldsymbol{\rm curl}~S^{B,D}_{\kappa_{\rm s}}[g_{2,B}] -(\lambda+\mu)\kappa_{\rm p}^2\big(S^{D,D}_{\kappa_{\rm p}}[g_{1,D}]+S^{B,D}_{\kappa_{\rm p}}[g_{1,B}]\big) \nonumber\\
&+S^{D,D}_{\kappa_{\rm a}}[g_{3,D}]+S^{B,D}_{\kappa_{\rm a}}[g_{3,B}]
+\mu(\nu_D\cdot\partial_{\tau_D}\tau_D)g_{1,D}+\mu(\nu_D\cdot\partial_{\tau_D}\nu_D)g_{2,D}+\mu\partial_{\tau_D} g_{2,D}, 
\label{boundaryIE1}\\
2f_2=&-\kappa^2_{\rm p}\tau_D^\top S^{D,D}_{\kappa_{\rm p}}\big[\langle\nu_D,\nu_D\rangle g_{1,D}\big]\nu_D+\tau_D^\top K^{D,D}_{\kappa_{\rm p}}\big[\tau_D\partial_{\tau_D} g_{1,D}+g_{1,D}\partial_{\tau_D}\tau_D\big]  \nonumber\\
&-\tau_D^\top H^{D,D}_{\kappa_{\rm p}}\big[\nu_D\partial_{\tau_D} g_{1,D}+g_{1,D}\partial_{\tau_D}\nu_D\big]+\kappa^2_{\rm s}\tau_D^\top S^{D,D}_{\kappa_{\rm s}}\big[\langle\tau_D,\nu_D\rangle g_{2,D}\big]\nu_D
\nonumber\\
&+\tau_D^\top K^{D,D}_{\kappa_{\rm s}}\big[\nu_D\partial_{\tau_D} g_{2,D}+g_{2,D}\partial_{\tau_D}\nu_D\big]+\tau_D^\top H^{D,D}_{\kappa_{\rm s}}\big[\tau_D\partial_{\tau_D} g_{2,D}+g_{2,D}\partial_{\tau_D}\tau_D\big] 
\nonumber\\
&+\tau_D\cdot\partial_{\nu_D}\nabla S^{B,D}_{\kappa_{\rm p}}[g_{1,B}]+\tau_D\cdot\partial_{\nu_D}\boldsymbol{\rm curl}~S^{B,D}_{\kappa_{\rm s}}[g_{2,B}] \nonumber\\ &+(\tau_D\cdot\partial_{\tau_D}\tau_D)g_{1,D}+\partial_{\tau_D} g_{1,D} +(\tau_D\cdot\partial_{\tau_D}\nu_D)g_{2,D}, 
\label{boundaryIE2}\\
2f_3=&K^{D,D}_{\kappa_{\rm p}}[g_{1,D}]+K^{B,D}_{\kappa_{\rm p}}[g_{1,B}]+H^{D,D}_{\kappa_{\rm s}}[g_{2,D}]+H^{B,D}_{\kappa_{\rm s}}[g_{2,B}]-\big(K^{D,D}_{\kappa_{\rm a}}[g_{3,D}]+K^{B,D}_{\kappa_{\rm a}}[g_{3,B}]\big)/(\omega^2\rho_{\rm a}) \nonumber\\
&+g_{1,D}+g_{3,D}/(\omega^2\rho_{\rm a}),
\label{boundaryIE3}
\end{align}
and on $\Gamma_B$ the field equations are the same as above with superscript/subscript $D$ and $B$ interchanged. The phaseless data equation is given by
\begin{align}\label{phadata eqn}
\bigg|\sum_\sigma S^\infty_{\kappa_{\rm a},\sigma}[g_{3,\sigma}]\bigg|^2=|u_\infty|^2.
\end{align}
In the reconstruction process, the field equations are solved for $g_{1,\sigma}, g_{2,\sigma}$ and $g_{3,\sigma}$ with an approximation of the boundary $\Gamma_D$. Then, by keeping $g_{1,\sigma}, g_{2,\sigma}$ and $g_{3,\sigma}$ fixed, the update of the boundary $\Gamma_D$ can be obtained by linearizing \eqref{phadata eqn} with respect to $\Gamma_D$.

\subsubsection{Parametrization and iterative scheme}

For simplicity, the boundary $\Gamma_D$ and $\Gamma_B$ are assumed to be starlike curves with the parametrized form
\begin{align*}
&\Gamma_D=\{p_D(\hat{x})=c+r(\hat{x})\hat{x}; ~c=(c_1,c_2)^\top,\ \hat{x}\in\Omega\}, \\
&\Gamma_B=\{p_B(\hat{x})=b+R\hat{x}; ~b=(b_1,b_2)^\top,\ \hat{x}\in\Omega\},
\end{align*} 
where $\Omega=\{\hat{x}(t)=(\cos t, \sin t)^\top; ~0\leq t< 2\pi\}$. We assume $G_D(\varsigma):=|p'(\varsigma)|=\sqrt{(r'(\varsigma))^2+r^2(\varsigma)}$ and $G_B=R$ denoted as the Jacobian of the transformation. 

Now, we reformulate the phaseless data equation \eqref{phadata eqn} as the parametrized integral equations
\begin{align}
\Big|\sum_{\sigma}S^\infty_{\kappa_{\rm a},\sigma}[\varphi_{3,\sigma}G_\sigma;p_\sigma]\Big|^2=|u_\infty|^2    \label{phaparadata eqn}
\end{align}
where $\varphi_{3,\sigma}=g_3\circ p_\sigma$, $\sigma=D, B$. By recalling the Fr\'{e}chet derivative operator ${S_\kappa^\infty}'[p;\varphi]q$ in \eqref{FreSinfty}, the linearization of \eqref{phaparadata eqn} leads to
\begin{equation}\label{linear phaparadata eqn}
2\Re\Big(\overline{\sum_\sigma S^\infty_{\kappa_{\rm a},\sigma} [\varphi_{3,\sigma}G_\sigma;p_\sigma]}{S_{\kappa_{\rm a}}^\infty}'[\varphi_{3,D}G_D;p_D]q\Big) =\breve{w},
\end{equation}
where
\begin{align*}
\breve{w}:=|u_\infty|^2-\bigg|\sum_{\sigma}S^\infty_{\kappa_{\rm a},\sigma}[\varphi_{3,\sigma} G_\sigma;p_\sigma]\bigg|^2.
\end{align*}

Again, with regard to our iterative procedure, the relative error estimator is chosen as following
\begin{align}
E_k:=\frac{\displaystyle \left\||u_\infty|^2-\Big|S^\infty_{\kappa_{\rm a},D}[\varphi_{3,D}G_D;p_D^{(k)}]+S^\infty_{\kappa_{\rm a},B} [\varphi_{3,B}G_B;p_B]\Big|^2\right\|_{L^2}} {\Big\||u_\infty|^2\Big\|_{L^2}}\leq\epsilon \label{reference relativeerror}
\end{align}
for some sufficiently small parameter $\epsilon>0$ depending on the noise level, where $p_D^{(k)}$ is the $k$th approximation of the boundary $\Gamma_D$. 

The iterative algorithm for the phaseless IAEIP is given by {\bf Algorithm II}. 

\begin{table}[ht]
	\begin{tabular}{cp{.8\textwidth}}
		\toprule
		\multicolumn{2}{l}{{\bf Algorithm II:}\quad Iterative algorithm for the phaseless IAEIP} \\
		\midrule
		Step 1 & Send an incident plane wave ${u}^{\rm
			inc}$ with a fixed wave number $\kappa_{\rm a}$ and a fixed incident direction
		$d\in\Omega$, and then collect the corresponding far-field data
		$u_\infty$ for the scatterer $D\cup B$; \\
		Step 2 & Select an initial star-like curve $\Gamma^{(0)}$
		for the boundary $\Gamma_D$ and the error tolerance $\epsilon$. Set $k=0$; \\
		Step 3 & For the curve $\Gamma^{(k)}$, compute the densities
		$\varphi_{1,\sigma}$, $\varphi_{2,\sigma}$ and $\varphi_{3,\sigma}$ from field equations; \\
		Step 4 & Solve \eqref{linear phaparadata eqn} to obtain the
		updated approximation $\Gamma^{(k+1)}:=\Gamma^{(k)}+q$ and evaluate the error
		$E_{k+1}$ defined in \eqref{reference relativeerror}; \\
		Step 5 & If $E_{k+1}\geq\epsilon$, then set $k=k+1$ and go
		to Step 3. Otherwise, the current approximation $\Gamma^{(k+1)}$ is served as
		the final reconstruction of $\Gamma_D$. \\
		\bottomrule
	\end{tabular}
\end{table}

\subsubsection{Discretization}

Noting that the kernels of $S^{\sigma,\varrho}_\kappa$, $K^{\sigma,\varrho}_\kappa$ and $ H^{\sigma,\varrho}_\kappa$ are weakly singular when $\sigma=\varrho$. With the help of quadrature rules \eqref{quadrature1}--\eqref{quadrature2_der}, the full discretization of \eqref{boundaryIE1}--\eqref{boundaryIE3} can be handled the same as those described in Section 5.

In addition, we introduce the following definition
\begin{align*}
&M_D(t,\varsigma;\varphi):=\gamma_{\kappa_{\rm a}}\exp\left\{-\mathrm{i}\kappa_{\rm a}\Big(c_1\cos t+c_2\sin t+r(\varsigma) \cos(t-\varsigma)\Big)\right\}\varphi(\varsigma), \\
&M_B(t,\varsigma;\varphi):=\gamma_{\kappa_{\rm a}}\exp\left\{-\mathrm{i}\kappa_{\rm a}\Big(c_1\cos t+c_2\sin t+R \cos(t-\varsigma)\Big)\right\}\varphi(\varsigma),\\
&\sum_{\sigma}S^\infty_{\kappa_{\rm a},\sigma}[\varphi_{3,\sigma}G_\sigma;p_\sigma](\varsigma_i^{(\bar{n})}) =\frac{\pi}{n}\sum_{j=0}^{2n-1}\bigg(M_D(\varsigma_i^{(\bar{n})},\varsigma_j^{(n)};\varphi_{3,D}G_D)+M_B(\varsigma_i^{(\bar{n})},\varsigma_j^{(n)};\varphi_{3,D}G_D)\bigg),
\end{align*}
Then, we get the discretized linear system
\begin{align}\label{phaseless discrelinear}
\sum_{l=1}^2 A_{l}^c(\varsigma_i^{(\bar{n})}) \Delta c_l+ \sum_{m=0}^M\alpha_mA_{1,m}^r(\varsigma_i^{(\bar{n})}) +\sum_{m=1}^M\beta_mA_{2,m}^r(\varsigma_i^{(\bar{n})})
= \breve{w}(\varsigma_i^{(\bar{n})})
\end{align}
to determine the real coefficients $\Delta c_1$, $\Delta c_2$, $\alpha_m$ and $\beta_m$, where
\begin{align*}
A_l^c(\varsigma_i^{(\bar{n})})=2\Re\Big\{\frac{\pi}{n}\overline{\sum_{\sigma}S^\infty_{\kappa_{\rm a},\sigma}[\varphi_{3,\sigma}G_\sigma;p_\sigma](\varsigma_i^{(\bar{n})})}\sum_{j=0}^{2n-1}L_l(\varsigma_i^{(\bar{n})},\varsigma_j^{(n)};\varphi_{3,D}G_D)\Big\}
\end{align*}
for $l=1, 2$, and
\begin{align*}
&A_{1,m}^r(\varsigma_i^{(\bar{n})})=2\Re\Big\{\frac{\pi}{n}\overline{\sum_{\sigma}S^\infty_{\kappa_{\rm a},\sigma}[\varphi_{3,\sigma}G_\sigma;p_\sigma](\varsigma_i^{(\bar{n})})}\sum_{j=0}^{2n-1}L_{3,m}(\varsigma_i^{(\bar{n})},\varsigma_j^{(n)};\varphi_{3,D}G_D)\Big\}, \\
&A_{2,m}^r(\varsigma_i^{(\bar{n})})=2\Re\Big\{\frac{\pi}{n}\overline{\sum_{\sigma}S^\infty_{\kappa_{\rm a},\sigma}[\varphi_{3,\sigma}G_\sigma;p_\sigma](\varsigma_i^{(\bar{n})})}\sum_{j=0}^{2n-1}L_{4,m}(\varsigma_i^{(\bar{n})},\varsigma_j^{(n)};\varphi_{3,D}G_D)\Big\}.
\end{align*}

Similarly, the overdetermined system \eqref{phaseless discrelinear} is also solved via the Tikhonov regularization with $H^2$ penalty term which is introduced in Section 6.1.2.

\begin{figure}
	\centering 
	\subfigure[]
	{\includegraphics[width=0.4\textwidth]{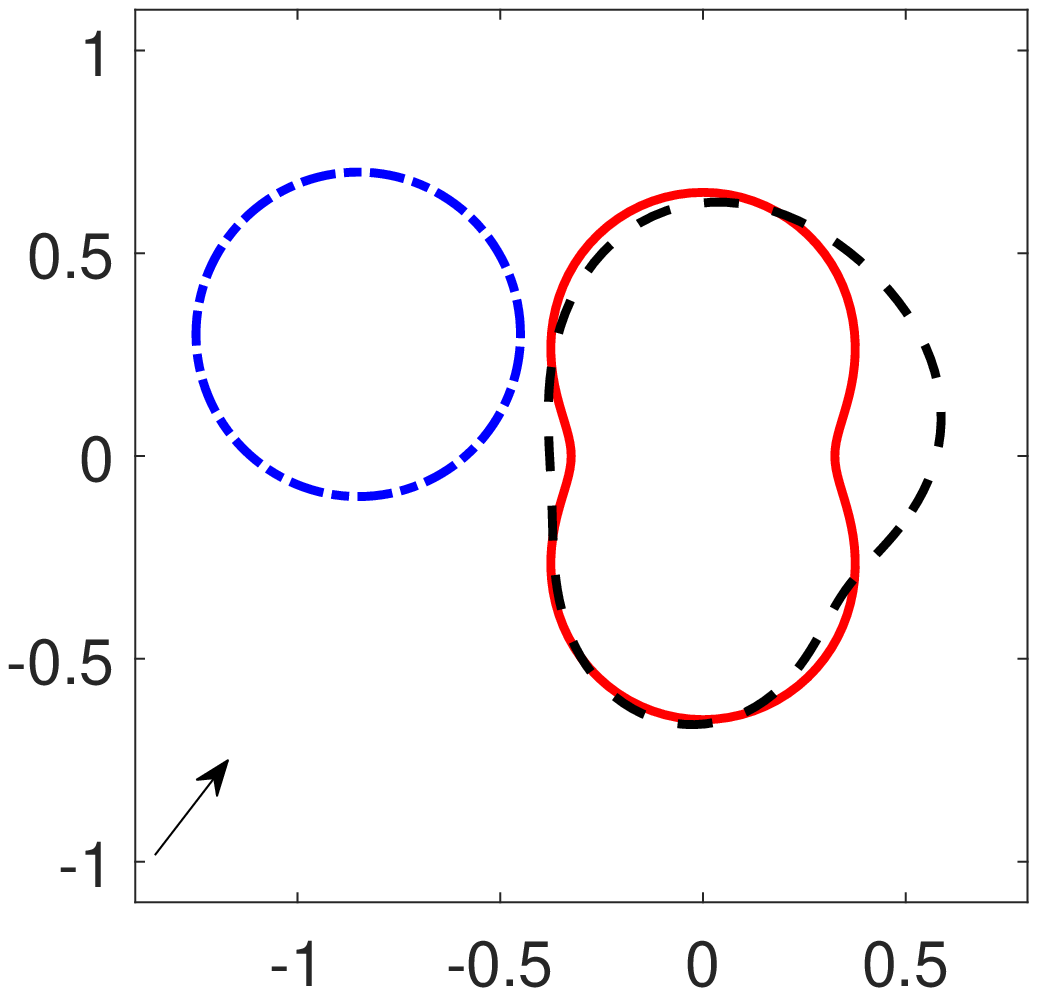}}
	\subfigure[]
	{\includegraphics[width=0.4\textwidth]{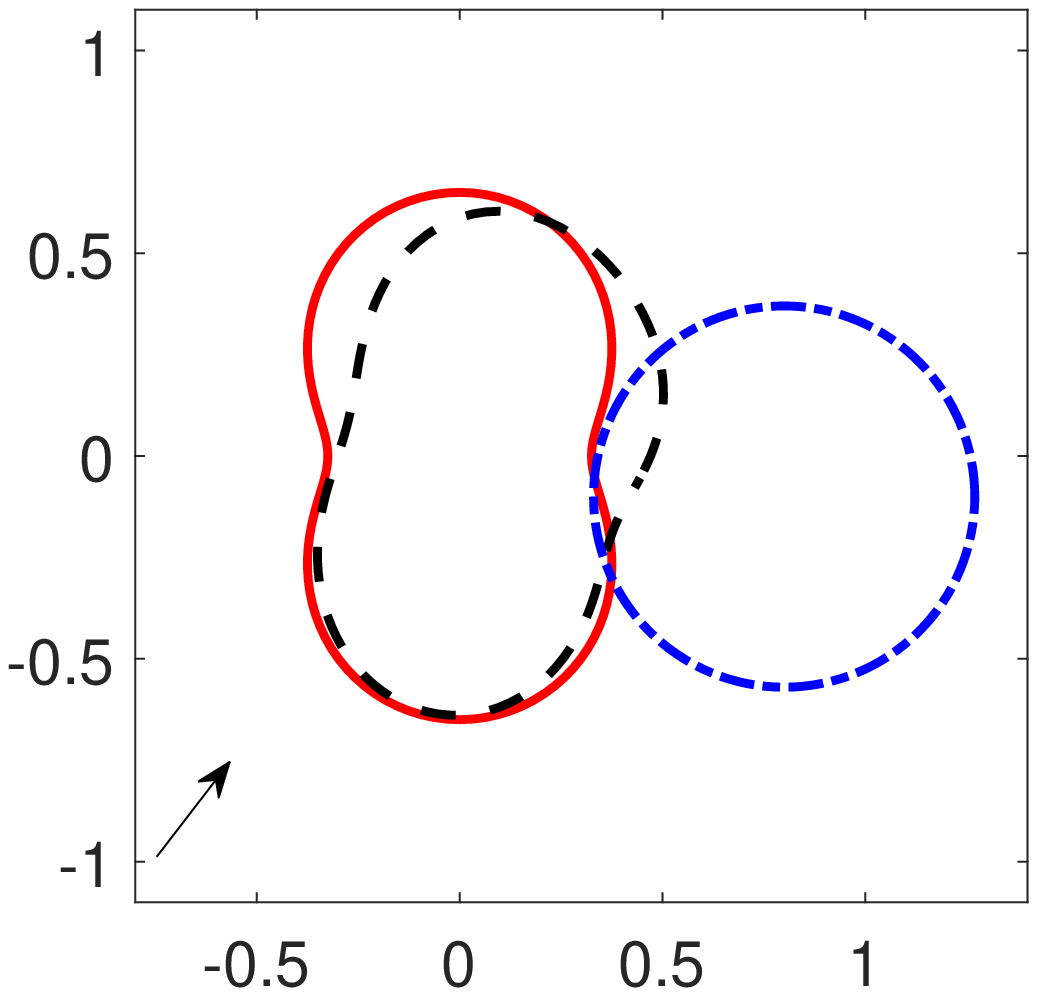}}
	\caption{Reconstructions of a peanut-shaped obstacle with different
		initial guesses, where $1\%$ noise is added and the incident angle
		$\theta=\pi/3$ (see example 1). (a) $(c_1^{(0)},c_2^{(0)})=(-0.85, 0.3)$, $r^{(0)}=0.4$, $\epsilon=0.25$; (b) $(c_1^{(0)},c_2^{(0)})=(0.8, -0.1)$, $r^{(0)}=0.47$, $\epsilon=0.3$.}\label{IOSP-6}
\end{figure}

\begin{figure}
	\centering 
	\subfigure[]
	{\includegraphics[width=0.4\textwidth]{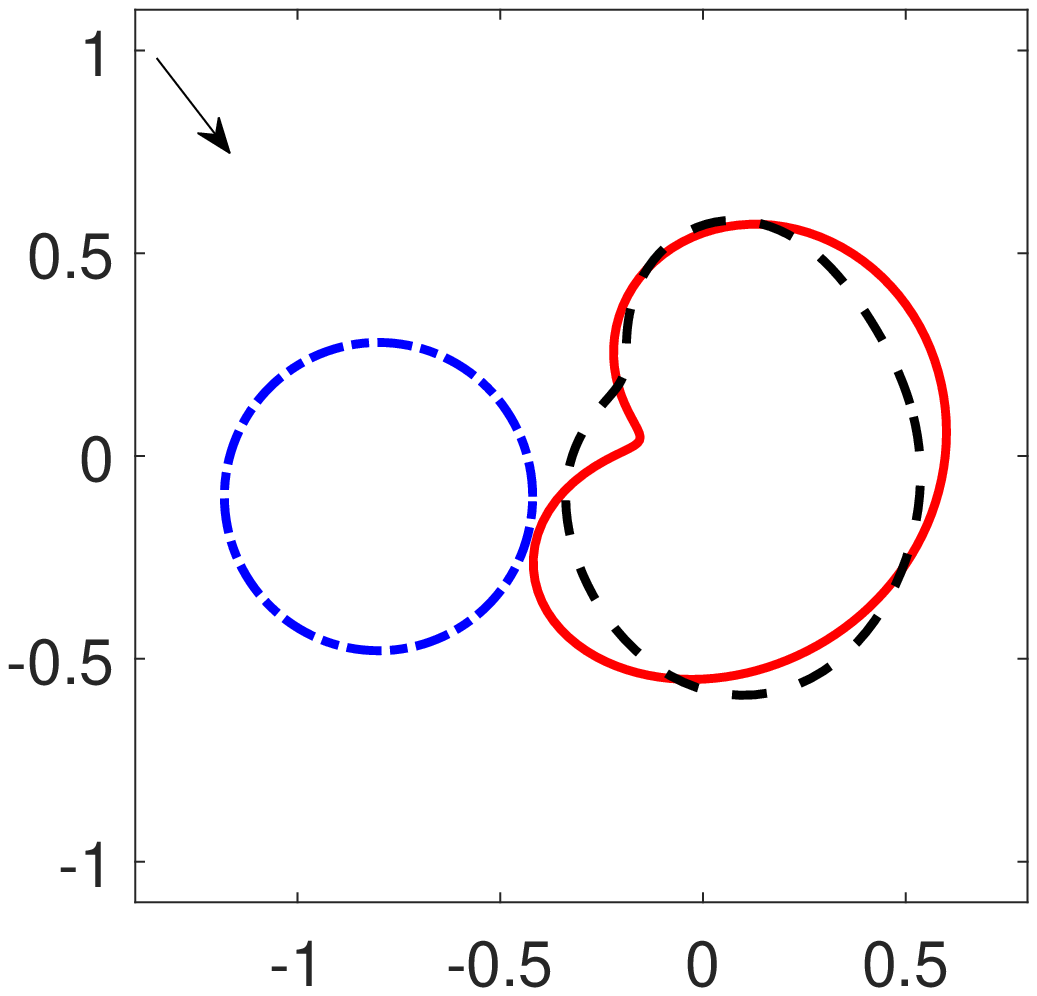}}
	\subfigure[]
	{\includegraphics[width=0.4\textwidth]{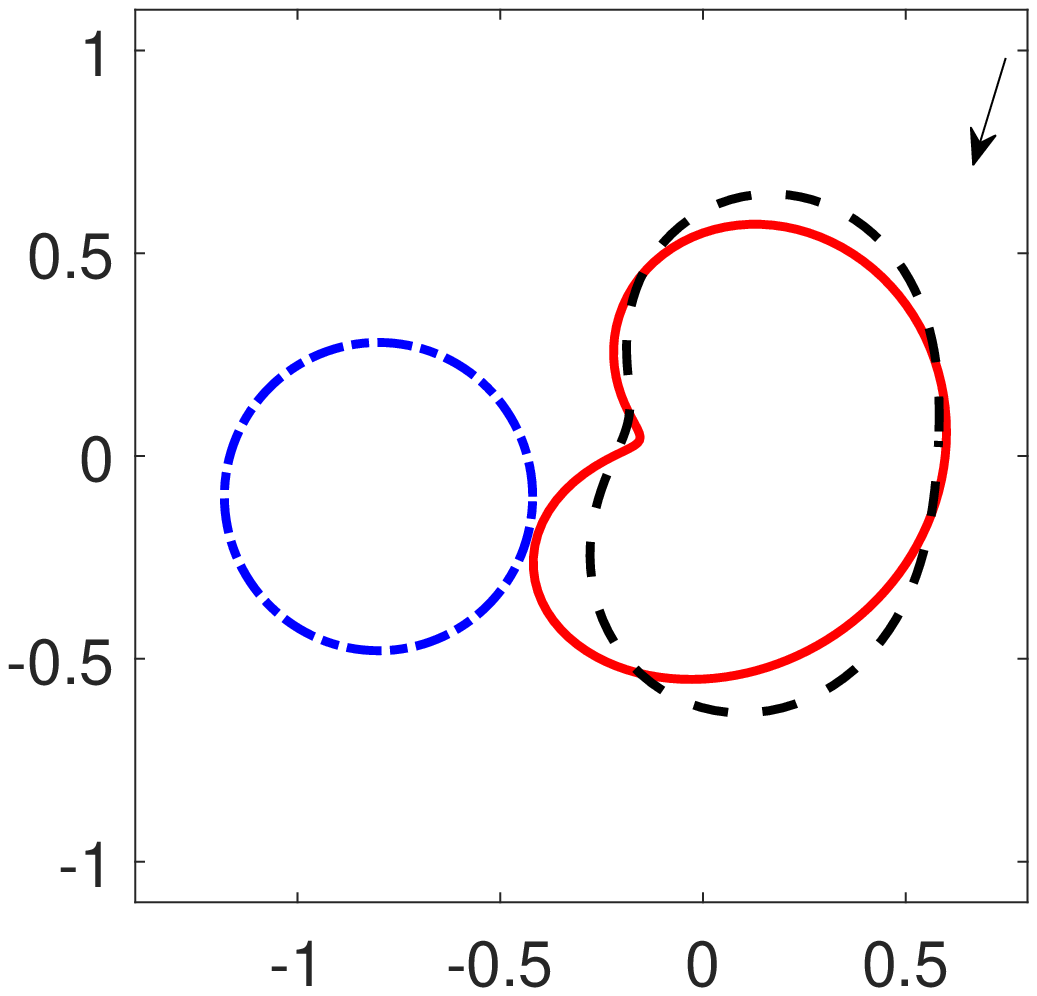}}
	\caption{Reconstructions of an apple-shaped obstacle with different incident
		directions, where $1\%$ noise is added and the initial guess
		is given by $(c_1^{(0)},c_2^{(0)})=(-0.8, -0.1)$, $r^{(0)}=0.38$ (see example 1). (a) incident
		angle $\theta=5\pi/3$, $\epsilon=0.2$; (b) incident angle $\theta=10\pi/7$,
		$\epsilon=0.2$.}\label{IOSP-4}
\end{figure}

\begin{figure}
	\centering 
	\subfigure[]
	{\includegraphics[width=0.4\textwidth]{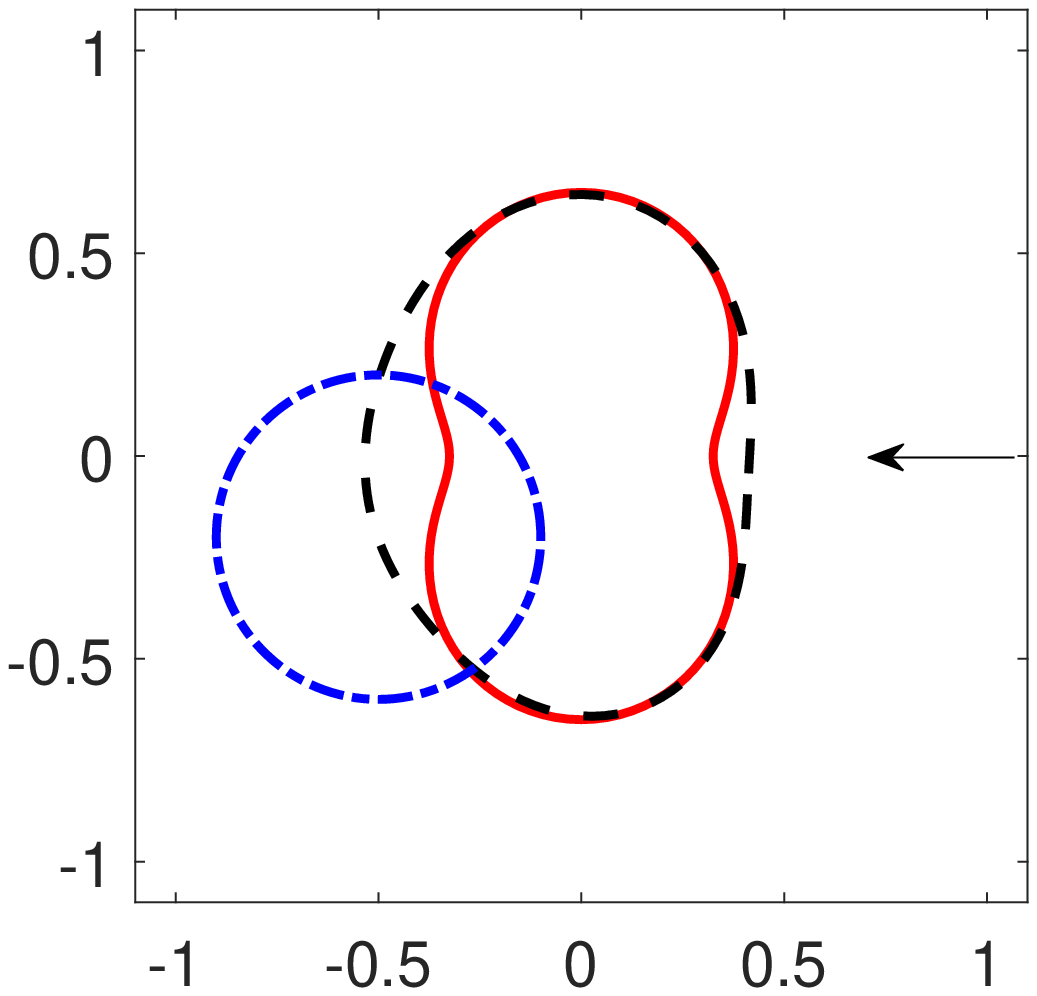}}
	\subfigure[]
	{\includegraphics[width=0.4\textwidth]{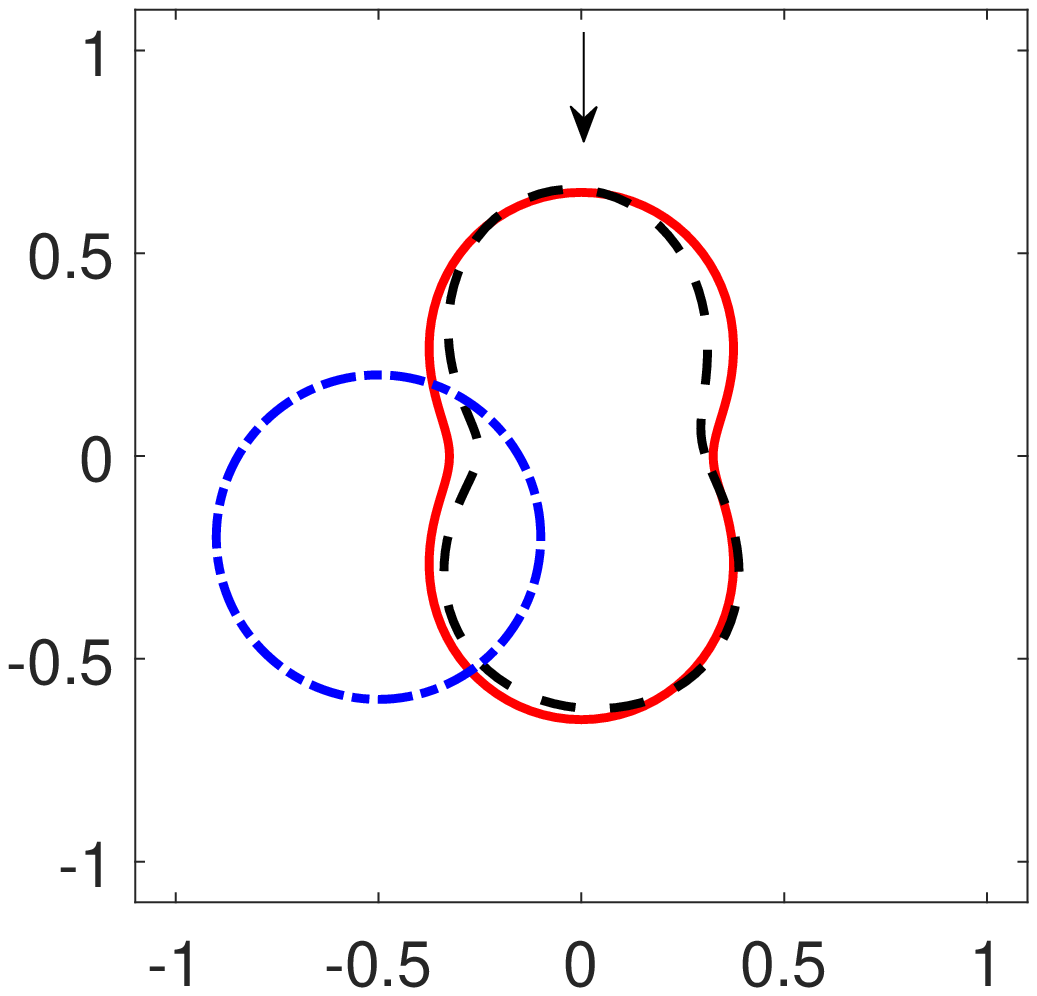}}
	\caption{Reconstructions of a peanut-shaped obstacle with different incident
		directions, where $1\%$ noise is added and the initial guess
		is given by $(c_1^{(0)},c_2^{(0)})=(-0.5, -0.2)$, $r^{(0)}=0.4$ (see example 1). (a) incident angle $\theta=\pi$, $\epsilon=0.2$; (b) incident angle $\theta=3\pi/2$,
		$\epsilon=0.2$.}\label{IOSP-7}
\end{figure}

\begin{figure}
	\centering 
	\subfigure[Recnstruction with $1\%$ noise, $\epsilon=0.05$ ]
	{\includegraphics[width=0.4\textwidth]{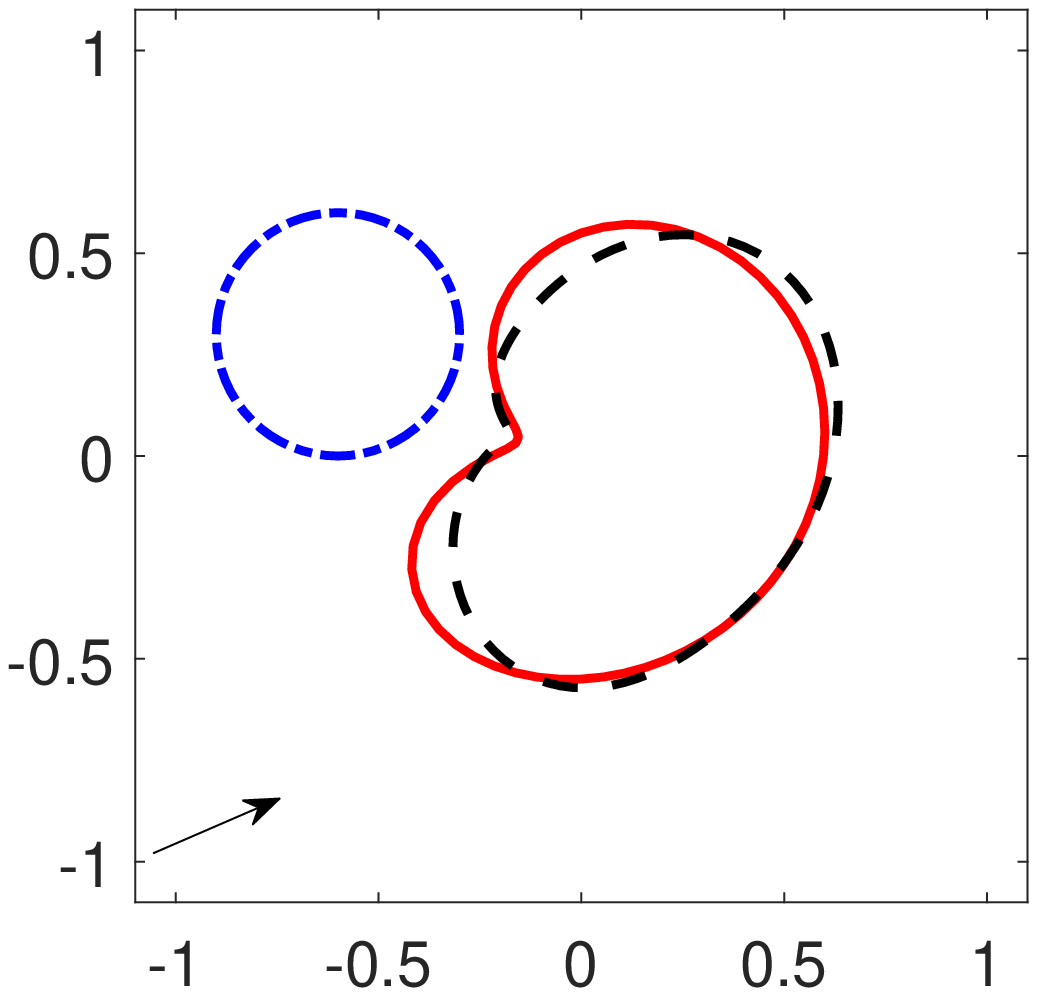}}
	\subfigure[Relative error with $1\%$ noise]
	{\includegraphics[width=0.4\textwidth]{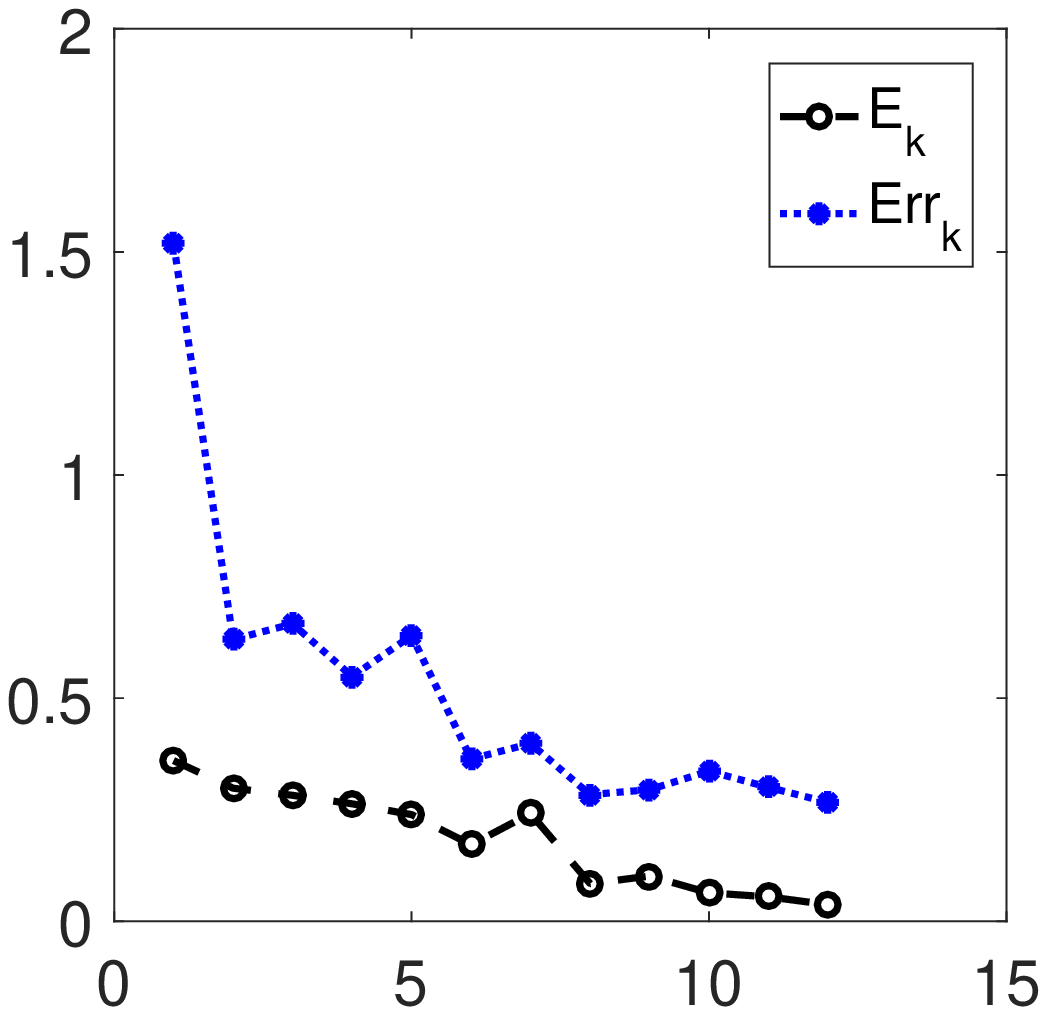}} 
	\subfigure[Reconstruction with $5\%$ noise, $\epsilon=0.1$]
	{\includegraphics[width=0.4\textwidth]{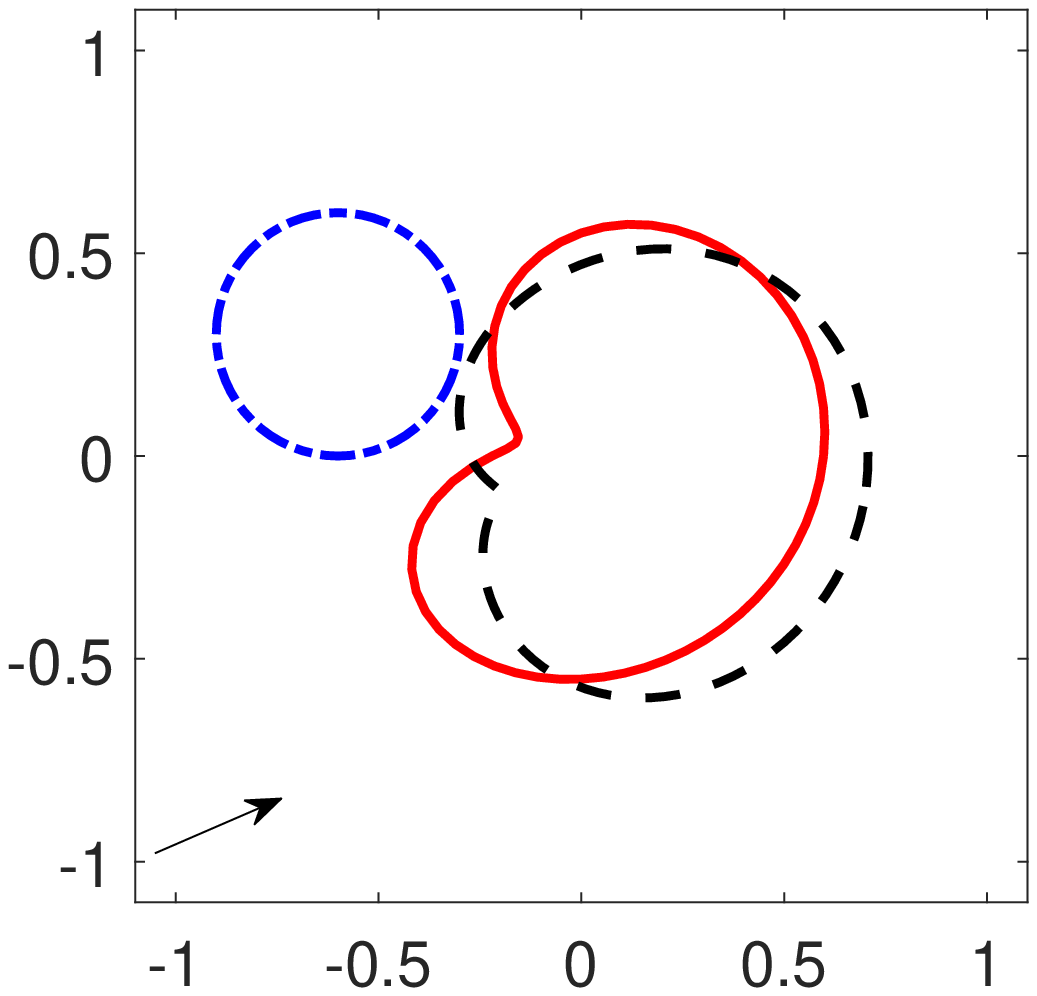}}
	\subfigure[Relative error with $5\%$ noise]
	{\includegraphics[width=0.4\textwidth]{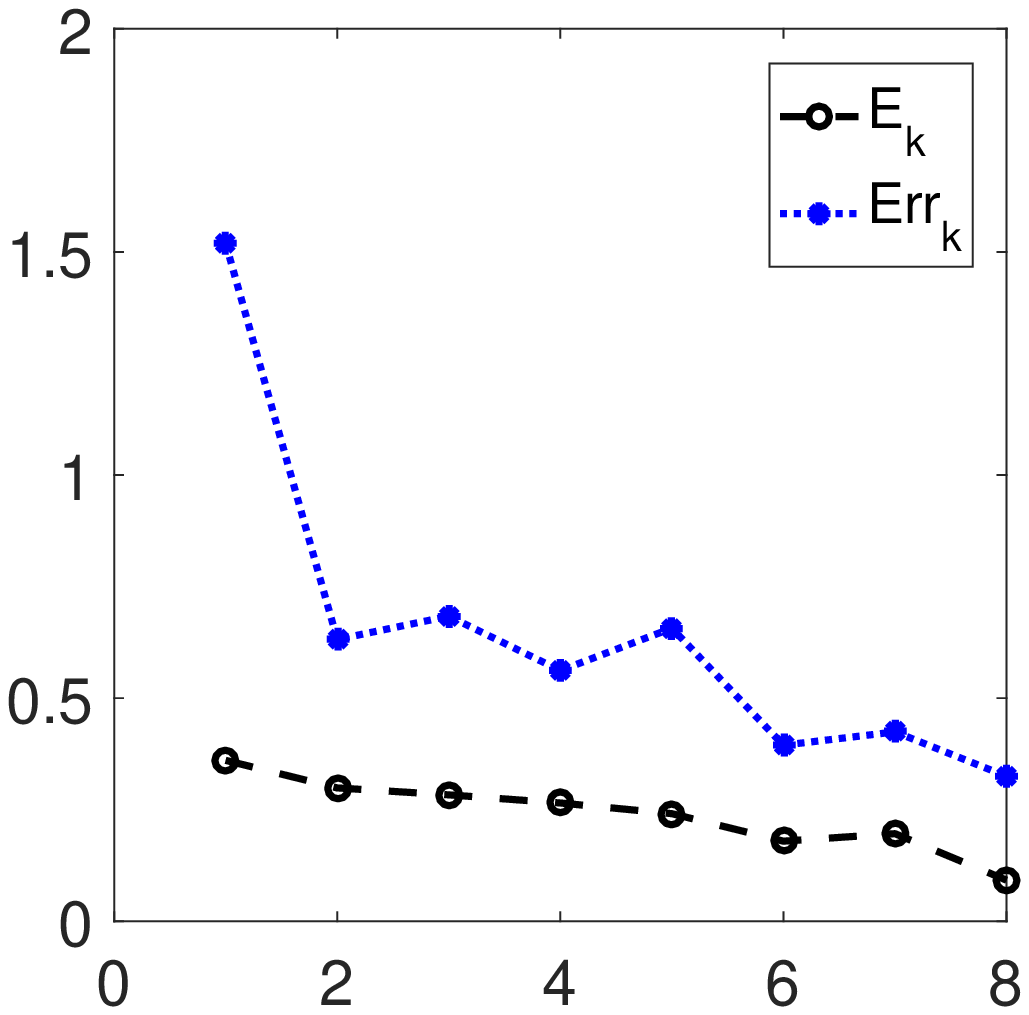}}
	\caption{Reconstructions of an apple-shaped obstacle with different levels of noise by using phaseless data and a reference ball (see example 2). The initial guess is given by 
		$(c_1^{(0)},c_2^{(0)})=(-0.6, 0.3), r^{(0)}=0.3$, the incident angle 
		$\theta=\pi/6$, and the reference ball is $(b_1,b_2)=(6.2, 0),
		R=0.74$.}\label{PhaselessIOSP-8}
\end{figure}

\begin{figure}
	\centering 
	\subfigure[Reconstruction with $1\%$ noise, $\epsilon=0.1$ ]
	{\includegraphics[width=0.4\textwidth]{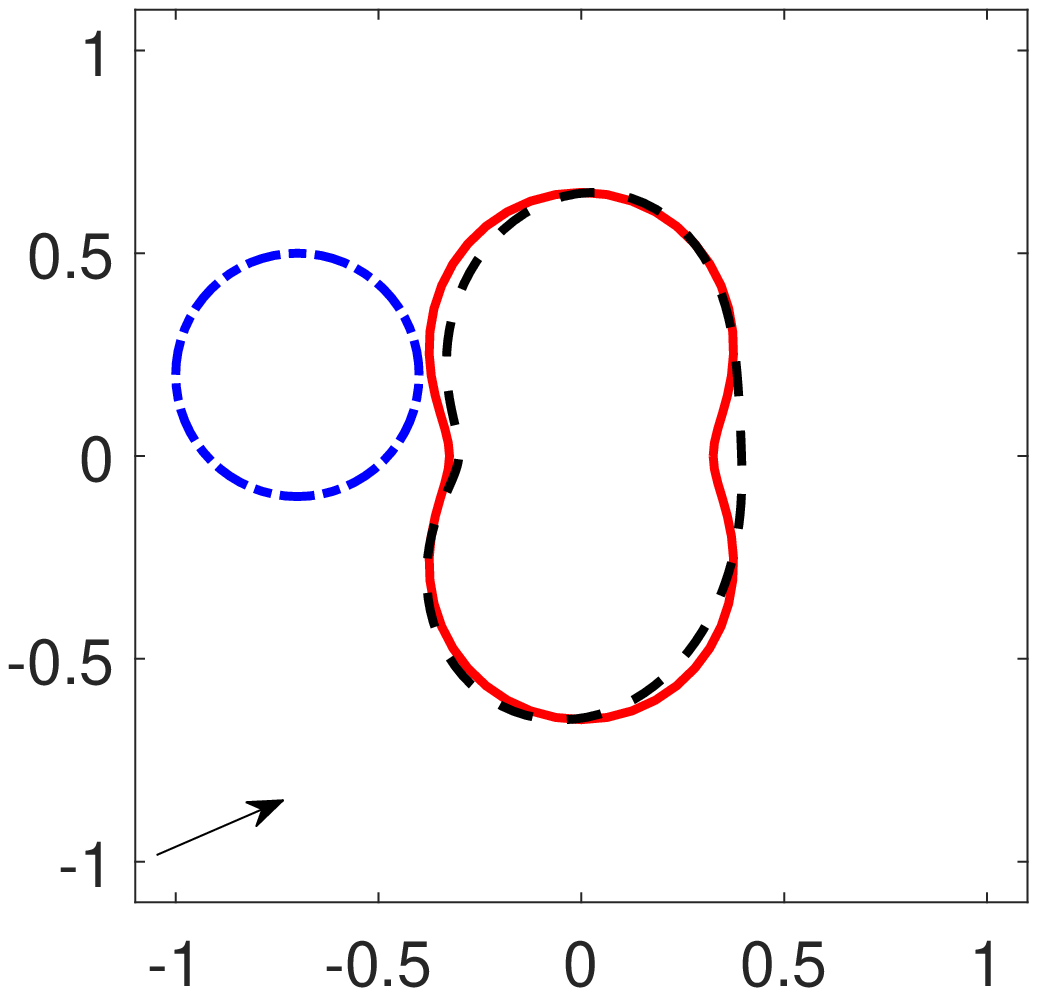}}
	\subfigure[Relative error with $1\%$ noise]
	{\includegraphics[width=0.4\textwidth]{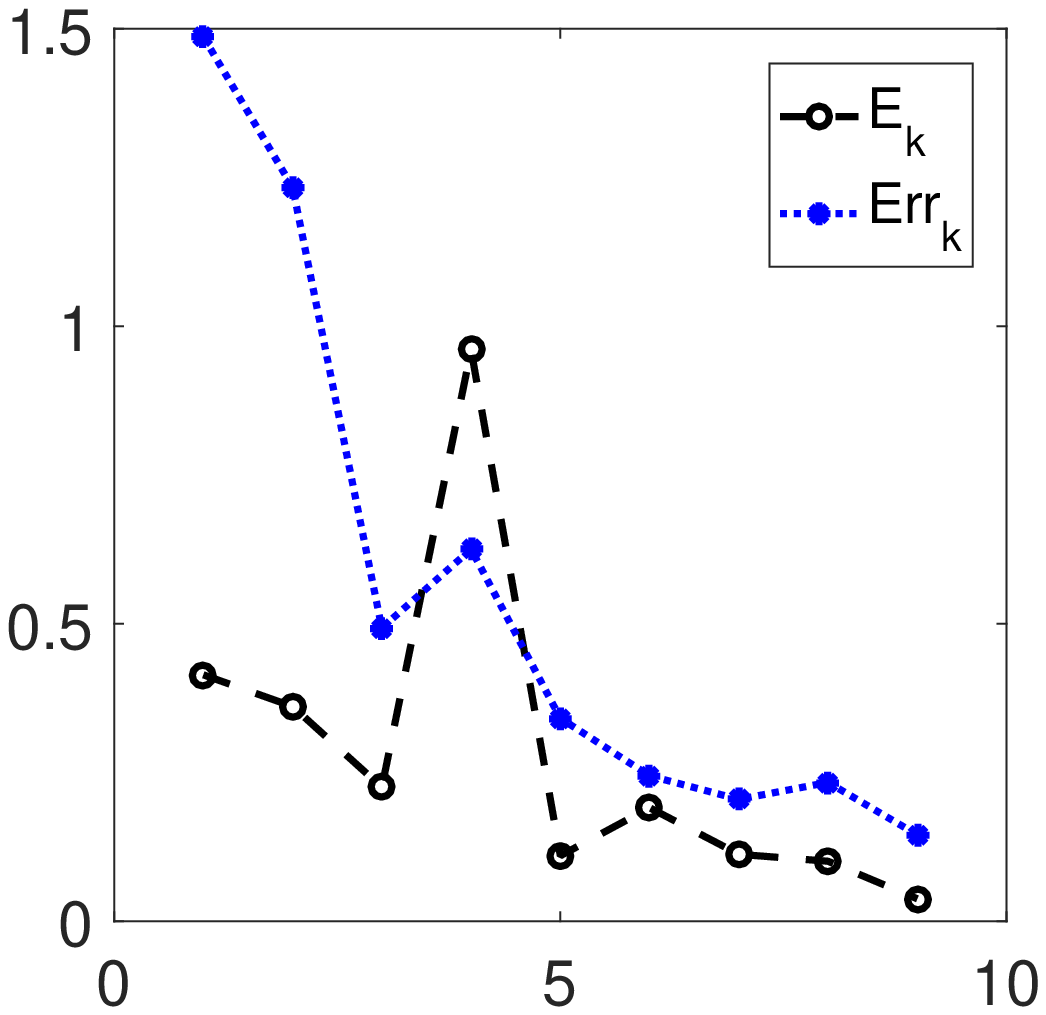}}
	\subfigure[Reconstruction with $5\%$ noise, $\epsilon=0.1$]
	{\includegraphics[width=0.4\textwidth]{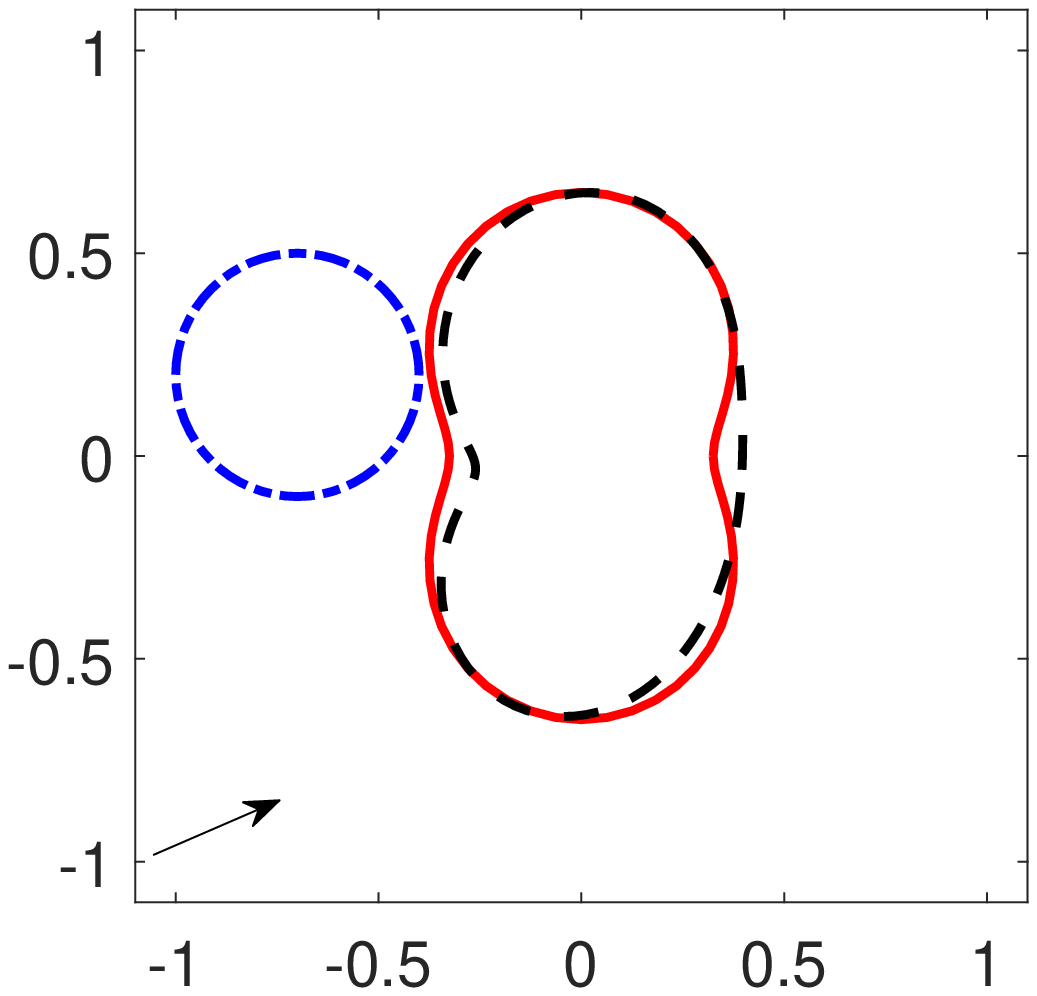}}
	\subfigure[Relative error with $5\%$ noise]
	{\includegraphics[width=0.4\textwidth]{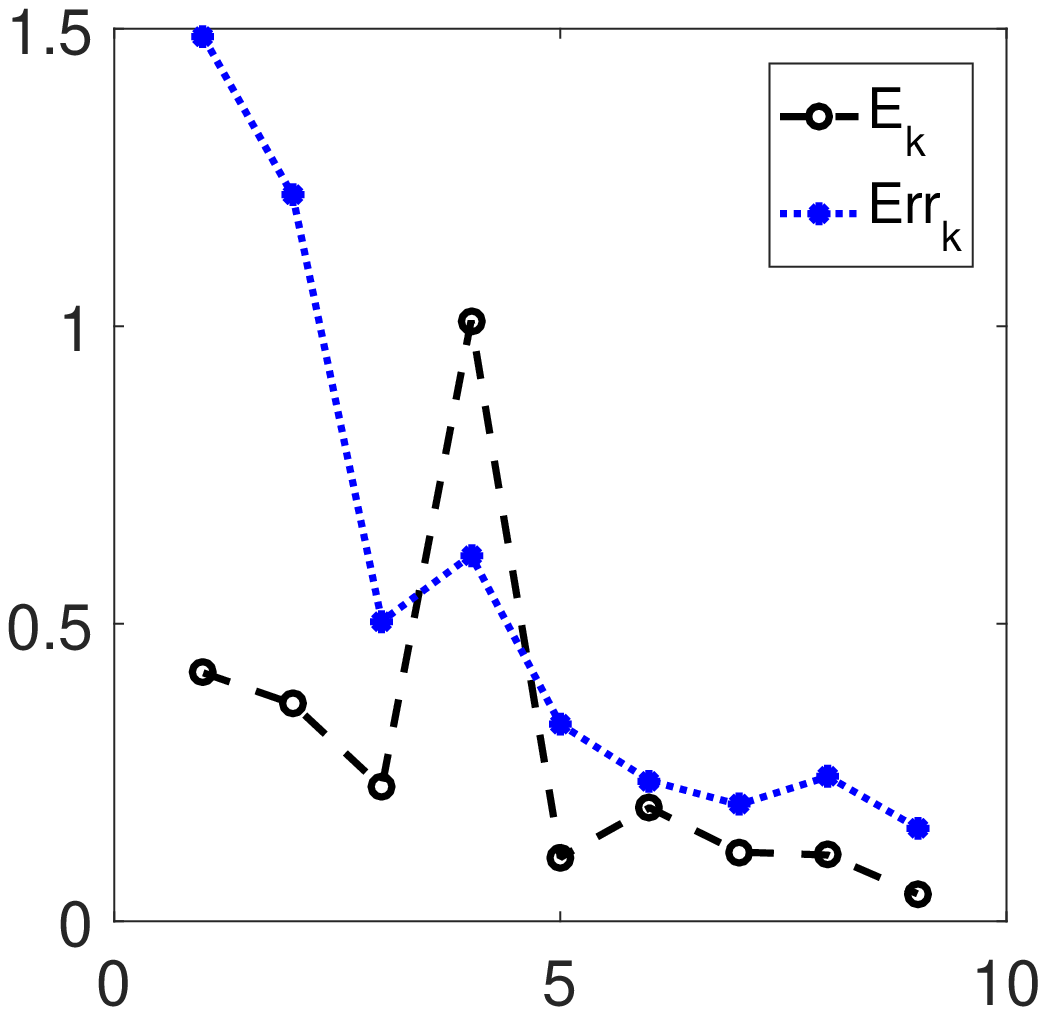}}
	\caption{Reconstructions of a peanut-shaped obstacle with different
		levels of noise by using phaseless data and a reference ball (see example 2). The initial guess is given by 
		$(c_1^{(0)},c_2^{(0)})=(-0.7, 0.2), r^{(0)}=0.3$, the incident angle
		$\theta=\pi/6$, and the reference ball is $(b_1,b_2)=(6.6, 0),
		R=0.71$.}\label{PhaselessIOSP-11}
\end{figure}

\begin{figure}
	\centering 
	\subfigure[]
	{\includegraphics[width=0.4\textwidth]{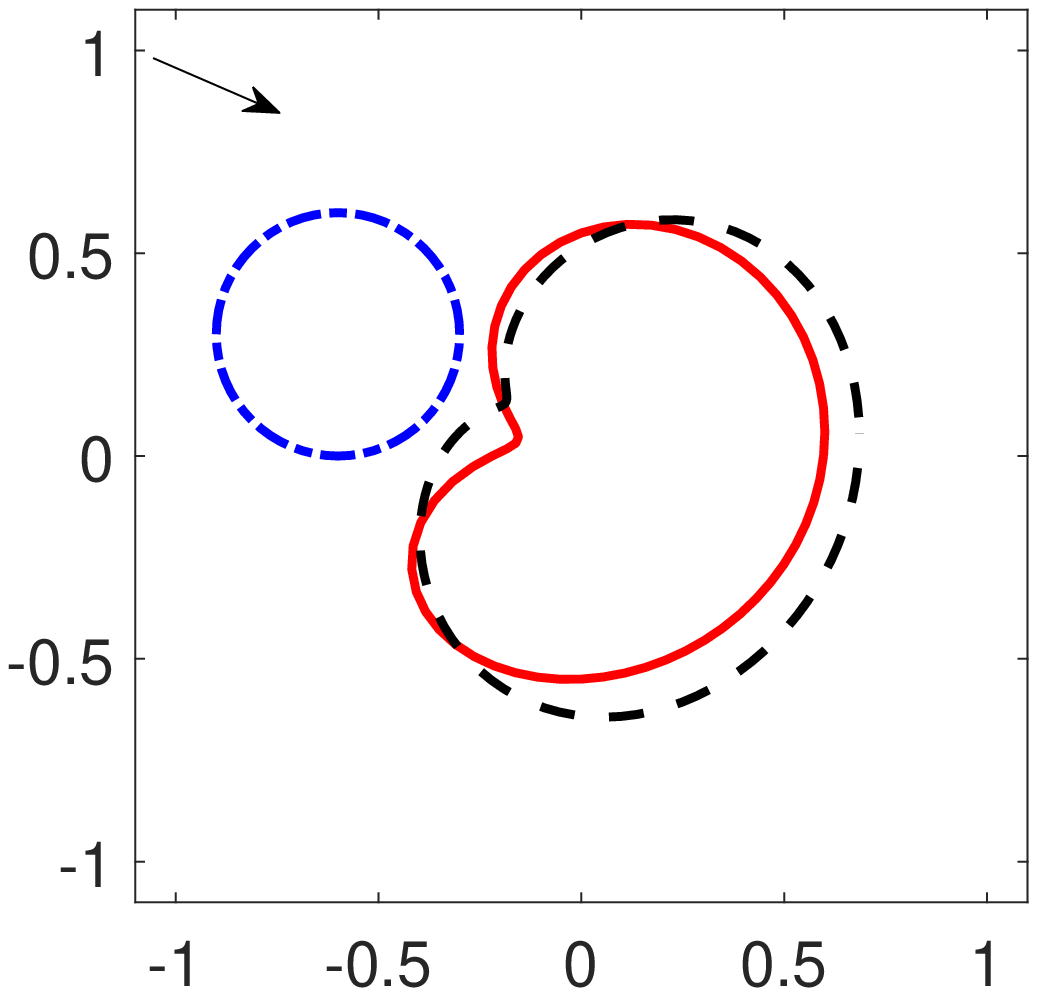}}
	\subfigure[]
	{\includegraphics[width=0.4\textwidth]{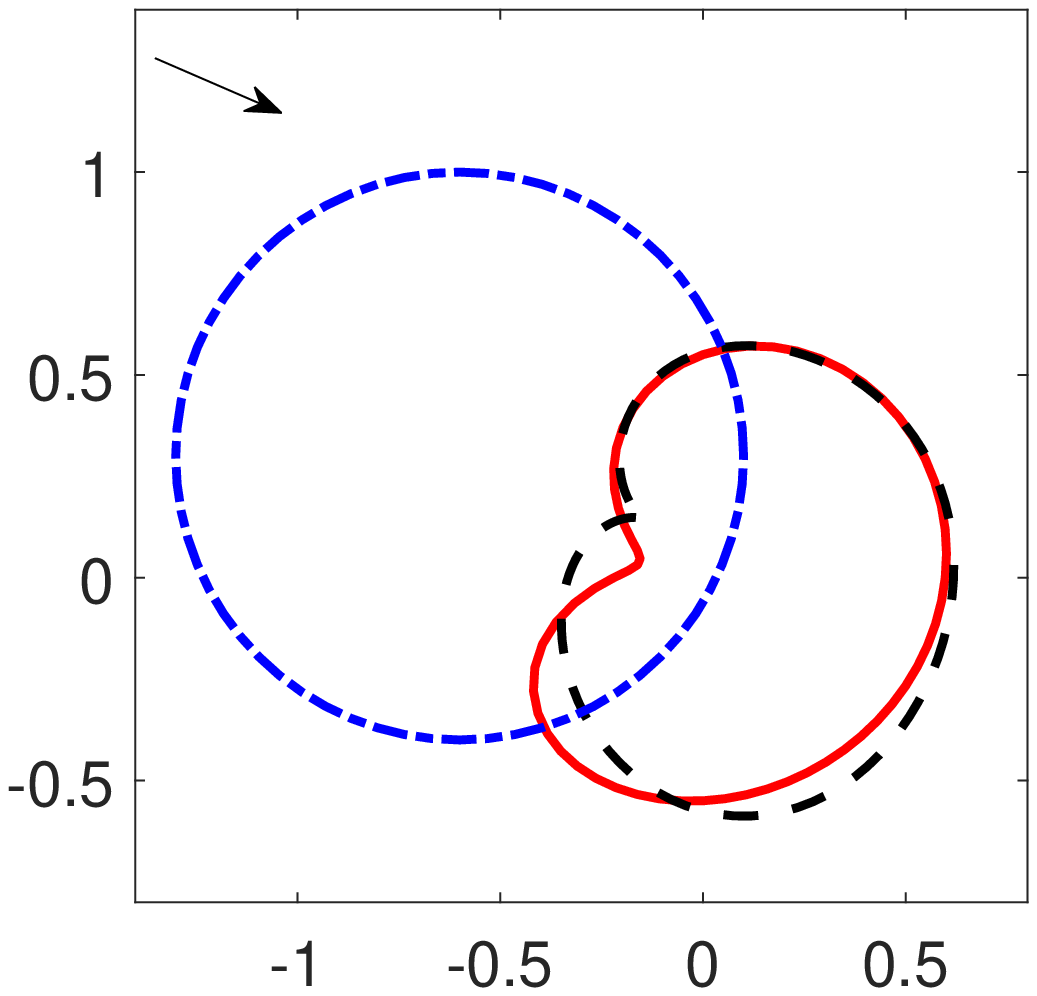}}
	\caption{Reconstructions of an apple-shaped obstacle with different initial guesses, where $1\%$ noise is added, the incident angle
		$\theta=11\pi/6$, and the reference ball is $(b_1,b_2)=(6.2, 0), R=0.65$ (see example 2). (a) $(c_1^{(0)},c_2^{(0)})=(-0.6, 0.3)$, $r^{(0)}=0.3$,
		$\epsilon=0.15$; (b) $(c_1^{(0)},c_2^{(0)})=(-0.6, 0.3)$, $r^{(0)}=0.7$,
		$\epsilon=0.15$.}\label{PhaselessIOSP-9}
\end{figure}

\begin{figure}
	\centering 
	\subfigure[]
	{\includegraphics[width=0.4\textwidth]{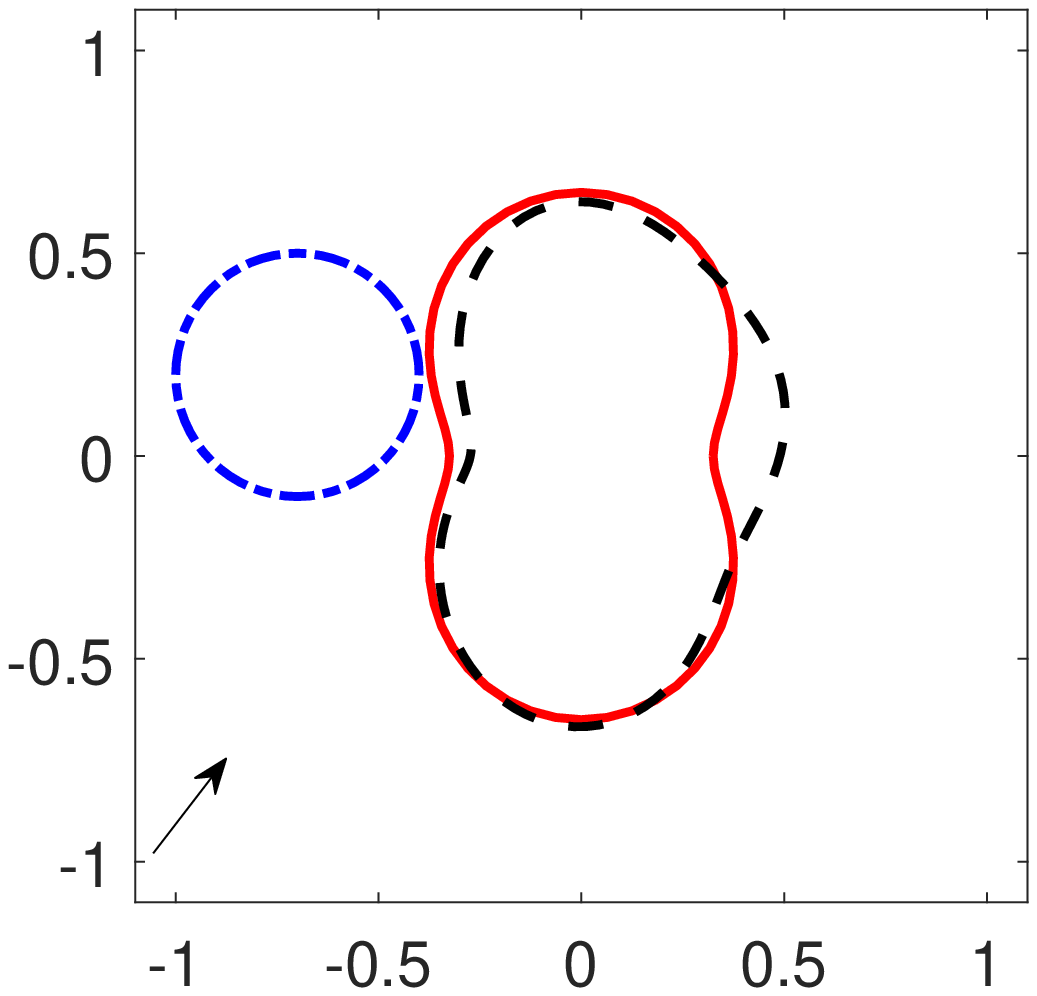}}
	\subfigure[]
	{\includegraphics[width=0.4\textwidth]{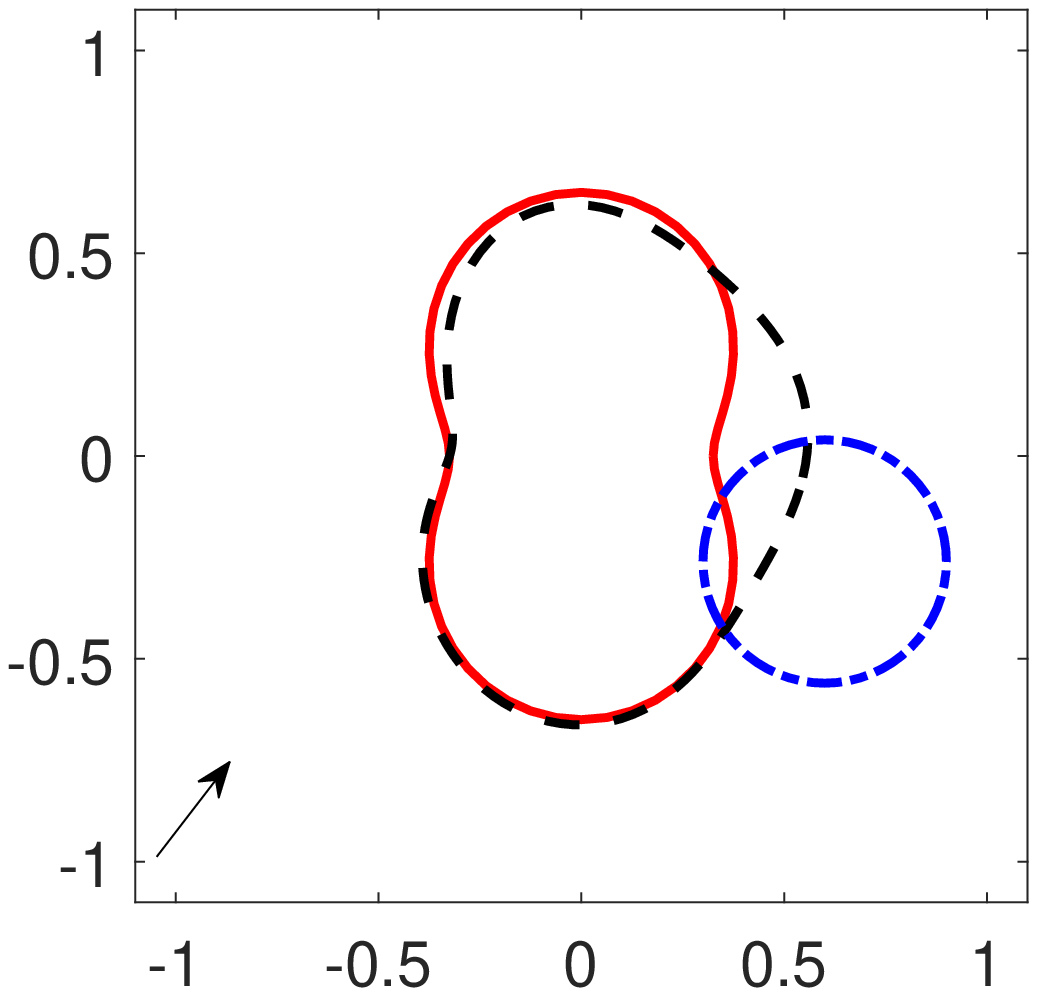}}
	\caption{Reconstructions of a peanut-shaped obstacle with different initial guesses, where $1\%$ noise is added, the incident angle
		$\theta=\pi/3$, and the reference ball is $(b_1,b_2)=(6.7, 0), R=0.67$ (see example 2). (a) $(c_1^{(0)},c_2^{(0)})=(-0.7, 0.2)$, $r^{(0)}=0.3$,
		$\epsilon=0.1$; (b) $(c_1^{(0)},c_2^{(0)})=(0.6, -0.26)$, $r^{(0)}=0.3$,
		$\epsilon=0.15$.}\label{PhaselessIOSP-12}
\end{figure}

\begin{figure}
	\centering
	\subfigure[]
	{\includegraphics[width=0.4\textwidth]{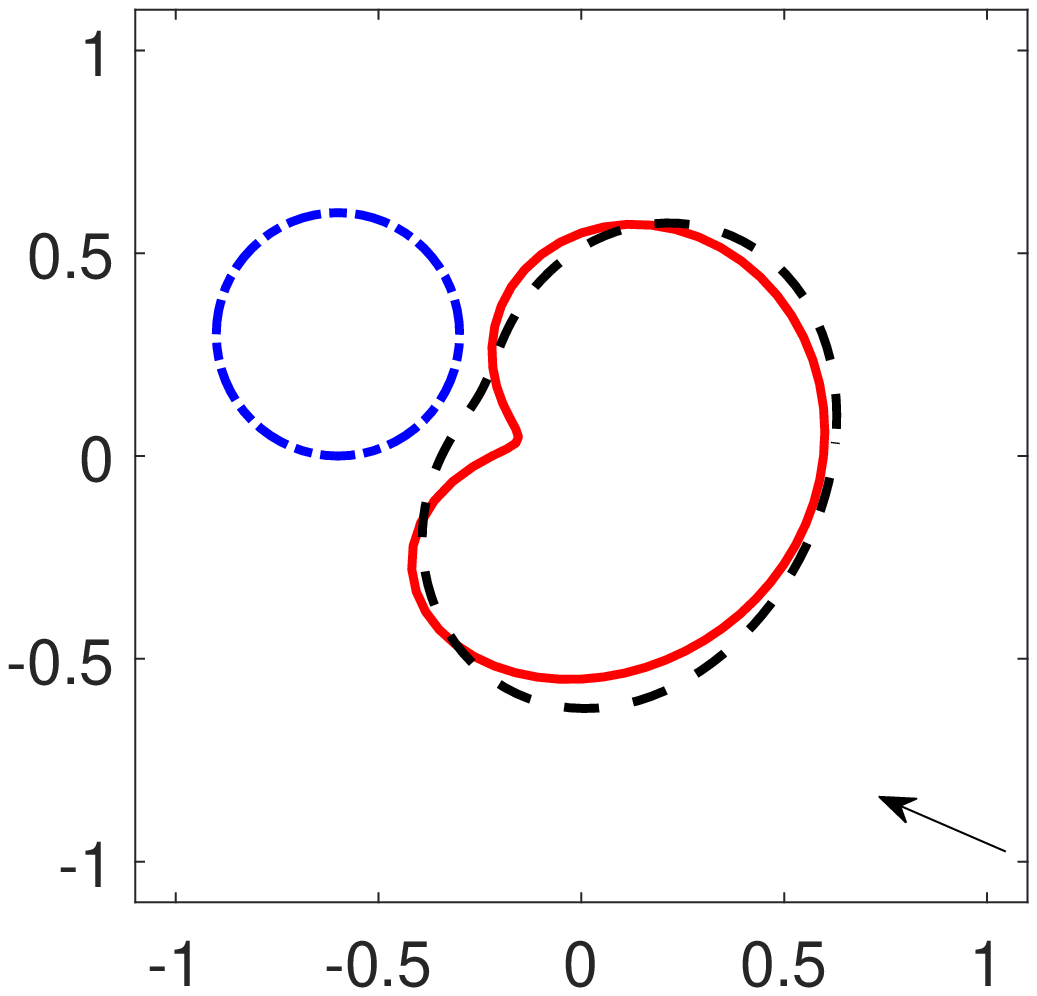}}
	\subfigure[]
	{\includegraphics[width=0.4\textwidth]{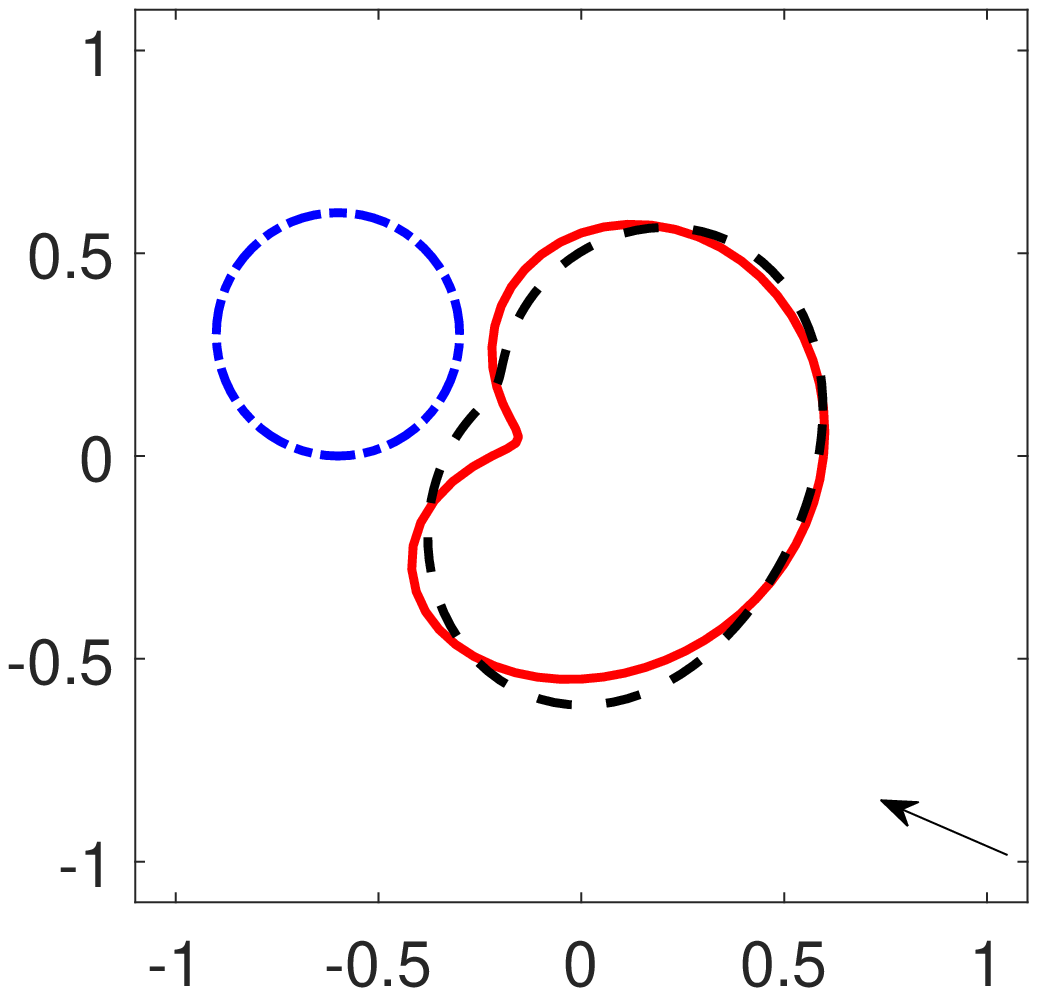}}
	\caption{Reconstructions of an apple-shaped obstacle with different reference balls, where $1\%$ noise is added, the inciedent angle 
		$\theta=5\pi/6$, and the initial guess is given by $(c_1^{(0)},c_2^{(0)})=(-0.6,
		0.3), r^{(0)}=0.3$ (see example 2). (a) $(b_1, b_2)=(6, 0), R=0.35$, $\epsilon=0.2$;
		(b) $(b_1, b_2)=(-6, 0), R=0.6$,
		$\epsilon=0.15$.}\label{PhaselessIOSP-10}
\end{figure}

\begin{figure}
	\centering
	\subfigure[]
	{\includegraphics[width=0.4\textwidth]{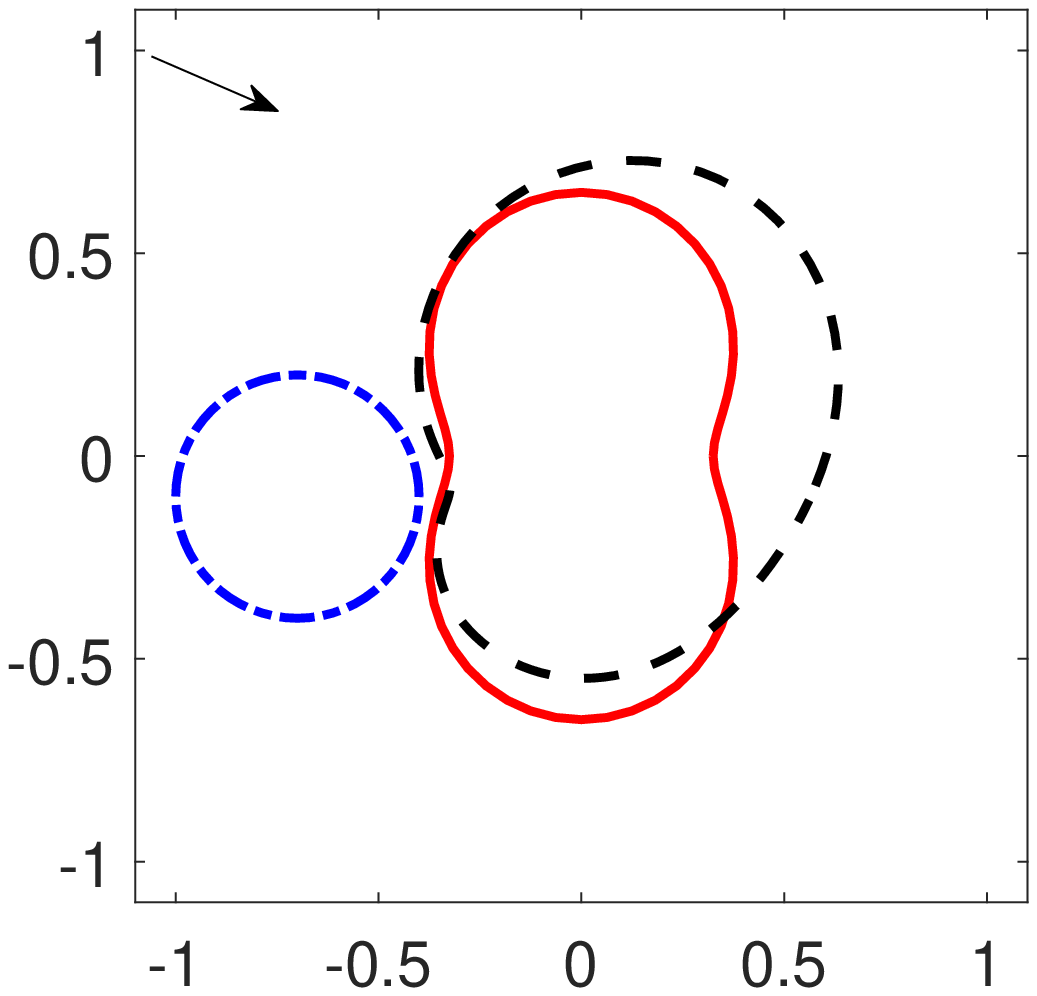}}
	\subfigure[]
	{\includegraphics[width=0.4\textwidth]{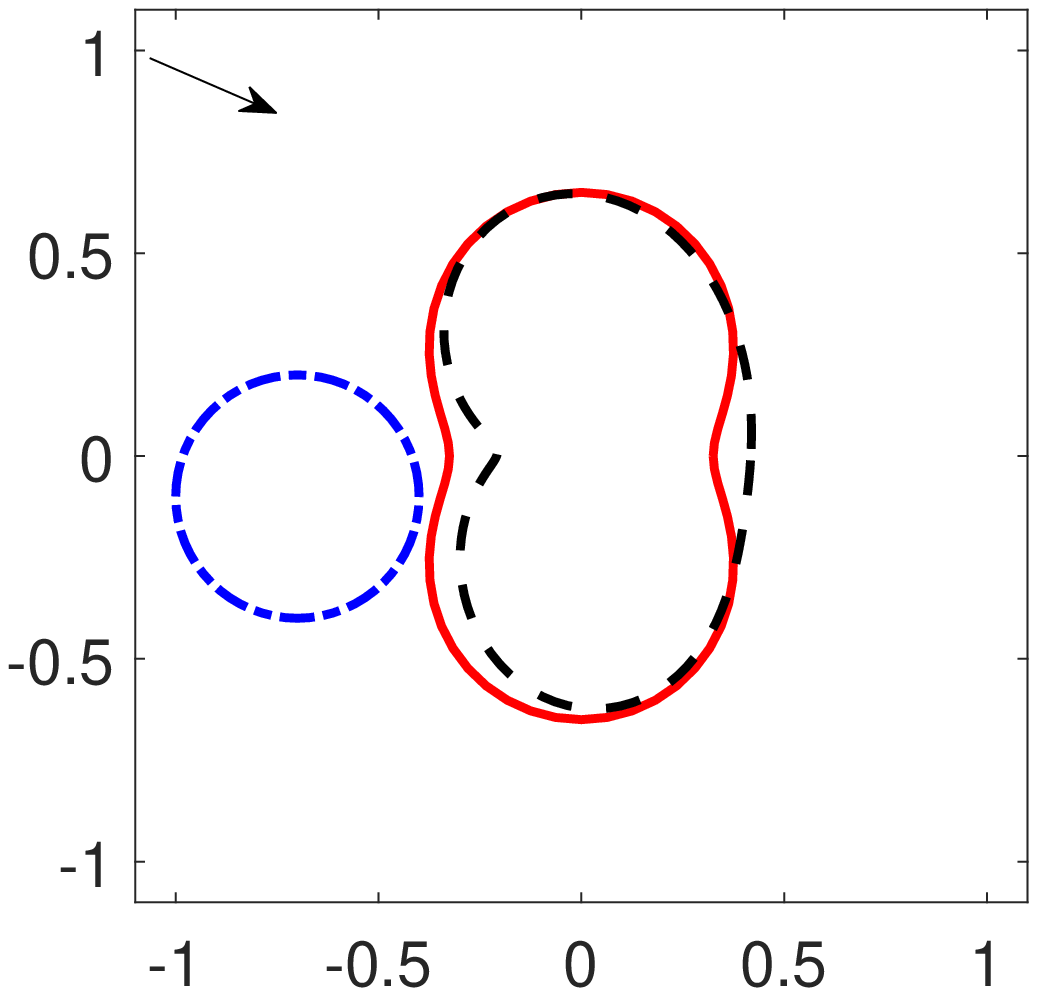}}
	\caption{Reconstructions of a peanut-shaped obstacle with different reference balls, where $1\%$ noise is added, the inciedent angle 
		$\theta=11\pi/6$, and the initial guess is given by $(c_1^{(0)},c_2^{(0)})=(-0.7,
		-0.1), r^{(0)}=0.3$ (see example 2). (a) $(b_1, b_2)=(-7.8, 0), R=0.47$, $\epsilon=0.2$;
		(b) $(b_1, b_2)=(6.6, 0), R=0.71$,
		$\epsilon=0.1$.}\label{PhaselessIOSP-13}
\end{figure}

\section{Numerical experiments}

In this section, we present some numerical examples to illustrate the feasibility of the iterative reconstruction methods. We use a single plane wave to illuminate the obstacle. The synthetic far-field data and phaseless far-field data are numerically generated  by the Nystr\"{o}m-type method described in Section 5. In order to avoid the inverse crime, the number of quadrature nodes used in the inverse solver ($n=64$) is chosen to be different from that of the forward solver ($n=100$). The noisy data $u_{\infty,\delta}$ and $|u_{\infty,\delta}|^2$ are generated in the following way,
\begin{align*}
u_{\infty,\delta}=u_{\infty}(1+\delta\breve\eta), \quad |u_{\infty,\delta}|^2=|u^\infty|^2(1+\delta\eta),
\end{align*}
where $\breve\eta=\breve\eta_1 +\mathrm{i}\breve\eta_2$, $\breve\eta_1$, $\breve\eta_2$ and $\eta$ are normally distributed random numbers ranging in $[-1,1]$, $\delta>0$ is the relative noise level. In addition,  we denote the $L^2$ relative error between the reconstructed and exact boundaries by
$$
Err_k:=\frac{\|p_D^{(k)}-p_D\|_{L^2(\Omega)}}{\|p_D\|_{L^2(\Omega)}}.
$$

In the iteration process, we obtain the update $\xi$ from a scaled Newton step with the Tikhonov regularization and $H^2$ penalty term, i.e., 
$$
\xi=\rho\bigg(\lambda\widetilde{I}+\Re(\widetilde{B}^*\widetilde{B})\bigg)^{-1}\Re(\widetilde{B}^*\widetilde{w}),
$$
where the scaling factor $\rho\geq0$ is fixed throughout the iterations. Analogously to \cite{DZhG2018}, the regularization parameter $\lambda$ in \eqref{EqualRLHuygens3} is chosen as
\[
\lambda_k:=\left\|w_\infty-S^\infty_{\kappa_{\rm a}}[\varphi_3G;p^{(k)}]\right\|_{L^2},\ k=1,2,\cdots.
\]

In all of the following figures, the exact boundary curves are displayed in solid lines, the reconstructed boundary curves are shown in dashed lines $--$, and all the initial guesses are taken to be a circle which is indicated in the dash-dotted lines $\cdot-$. The incident directions are denoted by directed line segments with arrows. Throughout all the numerical examples, we take $\lambda=3.88, \mu=2.56$, the angular frequency $\omega=0.7\pi$, the scaling factor $\rho=0.9$, and the truncation $M=6$. We present the results for two commonly used examples: an apple-shaped obstacle and a peanut-shaped obstacle. The parametrization of the exact boundary curves for these two obstacles are given in Table 1.

{\bf Example 1}: The IAEIP with far-field data. We consider the inverse problem of reconstructing an elastic obstacle from far-field data by using \textbf{Algorithm I}. The synthetic far-field data is numerically generated at 128 points, i.e. $\tilde{n}=64$. In Fig. \ref{IOSP-2} and Fig. \ref{IOSP-5}, the reconstructions of an apple-shaped and a peanut-shaped obstacles with $1\%$ and $5\%$ noise are shown, respectively. Moreover, the relative $L^2$ error $Err_k$ between the reconstructed and exact boundaries and the error $E_k$ defined in \eqref{relativeerror} are also presented with respect to the number of iterations. As we can see from the figures, the trend of two error curves is basically the same for larger number of iteration. Therefore, the choice of the stopping criteria is reasonable. The reconstructions with different initial guesses for the two curves are given in Fig. \ref{IOSP-3} and Fig. \ref{IOSP-6}, and the reconstructions with different directions of incident waves are presented in Fig. \ref{IOSP-4} and Fig. \ref{IOSP-7}. As shown in these results, the location and shape of the obstacle could be simultaneously and satisfactorily reconstructed for a single incident plane wave.

{\bf Example 2}: The IAEIP with phaseless far-field data and a reference ball. By adding a reference ball to the inverse scattering system, we consider the inverse problem of reconstructing an elastic obstacle from phaseless far-field data based on the algorithm of Table 2. The synthetic phaseless far-field data is numerically generated at 64 points, i.e. $\bar{n}=32$. The reconstructions with $1\%$ noise and $5\%$ noise are shown in Fig. \ref{PhaselessIOSP-8} and Fig. \ref{PhaselessIOSP-11}, respectively. Again, the relative $L^2$ error $Err_k$ and the error $E_k$ are presented in the figures. The reconstructions with different initial guesses for the two curves are given in Fig. \ref{PhaselessIOSP-9} and Fig. \ref{PhaselessIOSP-12}. The reconstructions with different reference balls are shown in Fig. \ref{PhaselessIOSP-10} and Fig. \ref{PhaselessIOSP-13}. From this example, we found that the translation invariance property of the phaseless far-field pattern can be broken down by introducing a reference ball. Based on this algorithm, both the location and shape of the obstacle can be satisfactorily reconstructed from the phaseless far-field data for a single incident plane wave.

\section{Conclusions}

In this paper, we have studied the two-dimensional inverse acoustic scattering problem by an elastic obstacle with the phased and phaseless far-field data for a single incident plane wave. Based on the Helmholtz decomposition, the coupled acoustic-elastic wave equation is
reformulated into a coupled boundary value problem of the Hemholtz equations, and the uniqueness of the solution for this boundary problem is proved. We investigate the jump relations for the second derivatives of single-layer potential and establish coupled boundary integral equations. We prove the well-posedness of the solution for the coupled boundary integral equations, and develop an efficient and accurate Nystr\"{o}m-type discretization to solve the coupled system. The method of nonlinear integral equations is developed for the inverse problem. In addition, we show that the phaseless far-field pattern is invariant under translation of the obstacle. To locate the obstacle, an elastic reference ball is introduced to the scattering system in order to break the translation invariance. We establish the uniqueness for the IAEIP with phaseless far-field pattern. A reference ball technique based nonlinear integral equations method is proposed for the inverse problem. Numerical results show that the location and shape of the obstacle can be satisfactorily reconstructed. Future work includes the uniqueness for the phaseless inverse scattering with one incident plane wave and the extension of the method to the three-dimensional inverse scattering problem.

\end{document}